\documentclass[10pt]{amsart}
\usepackage{pgf,tikz}%\usepackage{ulem}
\usetikzlibrary{decorations.pathreplacing}
\usetikzlibrary{decorations.markings}

\usepackage{caption}

\usepackage{amssymb}\usepackage{stmaryrd}
\usepackage{mathrsfs}
\usepackage[all]{xy}

\usepackage{shuffle}

\usepackage{tikz-cd}
\usetikzlibrary{arrows}

\makeatletter
  \def\@wrindex#1{%
    \protected@write\@indexfile{}%
      {\string\indexentry{#1}{ \S\thesubsection (p.\thepage)}}
    \endgroup
  \@esphack}

\makeatother

\makeindex

\usepackage{rotating}

\usepackage{amscd}
\usepackage{epsfig}
\author{Benjamin Enriquez}
\address{IRMA (UMR 7501) et D\'epartement de Math\'ematiques, Universit\'e de Strasbourg, 7 rue Ren\'e-Descartes, 67084 Strabourg (France)}
\email{b.enriquez@math.unistra.fr}
\author{Federico Zerbini}
\address{Université Paris-Saclay, CNRS, CEA, Institut de physique théorique, 91191, Gif-sur-Yvette, France}
\email{federico.zerbini@ipht.fr}
\newtheorem{thm}{Theorem}[section]
\newtheorem{lem}[thm]{Lemma}

\newtheorem{cor}[thm]{Corollary}
\newtheorem{prop}[thm]{Proposition}

{\theoremstyle{definition} \newtheorem{rem}[thm]{Remark}}
{\theoremstyle{definition} \newtheorem{defn}[thm]{Definition}}

{\theoremstyle{remark} }

\numberwithin{equation}{subsection}

\begin{document}
\baselineskip 16pt 
%{\tiny{RECH/projet/InterprDelta*/BettiDS*.pdf}}\hfill
%\qquad{\tiny{\today}}
%
\title[Maurer-Cartan elements over configurations spaces of curves %positive genus curves via iterated integrals
]{
Construction of Maurer-Cartan elements \\
over configuration spaces of curves %via iterated integrals
}
\begin{abstract}
For~$C$ a complex curve and $n \geq 1$, a pair $(\mathcal{P},\nabla_\mathcal{P})$ of a principal bundle $\mathcal{P}$ with meromorphic flat connection over~$C^n$, holomorphic over the configuration space $C_n(C)$ of~$n$ points over~$C$, was introduced in~\cite{EnrConfSp}. For any point $\infty \in C$, we construct a trivialisation of the restriction of~$\mathcal{P}$ to $(C\setminus\infty)^n$ and obtain a Maurer-Cartan element~$J$ over $C_n(C\setminus\infty)$ out of $\nabla_\mathcal{P}$, thus generalising a construction of Levin and Racinet when the genus of~$C$ is higher than one. We give explicit formulas for~$J$ as well as for~$\nabla_\mathcal{P}$. When $n=1$, this construction gives 
rise to elements of Hain's space of second kind iterated integrals over~$C$.
\end{abstract}
\maketitle

{\small\tableofcontents}

\section{Introduction}\label{sec:Intro}

\subsection{Motivation and context}

Let~$M$ be a smooth manifold and $a,b\in M$. By a theorem of Chen, the function ring on the pro-unipotent completion of $\pi_1(M;a,b)$ is isomorphic to the zeroth cohomology of a complex $B^*(\mathcal{A}^*(M))$, where $\mathcal{A}^*(M)$ is the dg-algebra of smooth differential forms on~$M$ and~$B^*$ is the ``bar-complex'' functor~\cite{Chen}. Elements of $Z^0(B^*(\mathcal{A}^*(M)))$ give rise to homotopy-invariant iterated integrals and therefore to functions on~$\tilde M^2$, where~$\tilde M$ denotes a universal cover of~$M$. When $M=C_n(\mathbb{P}_{\mathbb{C}}^1\setminus\{0,1,\infty\})$ is the configuration space of~$n$ points on $\mathbb{P}_{\mathbb{C}}^1\setminus\{0,1,\infty\}$, a model of $\mathcal A^*(M)$, which is a connected dg-subalgebra of this dg-algebra, is known, and the functions on~$\tilde M^2$ associated with $Z^0(B^*(A))$ are the classical multiple polylogarithms; this was extended to configuration spaces of genus-1 curves in~\cite{BrownLevin}.
Let $\mathfrak{g}$ be a pro-nilpotent Lie algebra, then we define ${\rm MC}(M,\mathfrak{g})$ to be the set of Maurer-Cartan elements\footnote{For $V$ a vector space, we set $V\,\hat{\otimes}\,\mathfrak{g}:=\lim_{\leftarrow}V\otimes (\mathfrak{g}/F^{n}\mathfrak{g})$, where $F^n\mathfrak{g}$ is the $n$-th term of the lower central series of $\mathfrak{g}$.} $\omega\in \mathcal{A}^1(M)\,\hat\otimes\,\mathfrak{g}$, i.e. such that $d\omega+(1/2)[\omega,\omega]=0$. Then a pair $(\omega, P)\in {\rm MC}(M,\mathfrak{g})\times \mathbb{C}[\exp(\mathfrak{g})]$ gives rise to an element of $Z^0(B^*(\mathcal{A}^*(M)))$ and therefore to a function on $\tilde M^2$, given by $(\tilde a,\tilde b)\to P(\mathcal{P}\exp \int_{\tilde a}^{\tilde b}\omega)$, where for a smooth path $\eta:[0,1]\to M$
$$
\mathcal{P}\exp \int_{\eta} \omega\,:=\,\sum_{n \geq 0}\int_{0 \leq t_1 \leq\cdots \leq t_n \leq 1} (\eta^*\omega)(t_1)\cdots (\eta^*\omega)(t_n) \in \exp(\mathfrak{g}) \subset U(\mathfrak{g})^\wedge\,,
$$
is a \emph{path-ordered exponential}. 
%In particular, a principal bundle with flat connection $(P,\nabla_P)$ equipped with a trivialisation of the lift of~$P$ to~$\tilde M$ gives rise to elements of~$Z^0(B^*(A^*(tilde M)))$ and therefore also to functions on $\tilde M^2$.

An analogue of the above map $Z^0(B^*(\mathcal{A}^*(M)))\to C^\infty(\tilde M^2)$ was constructed in the context of complex algebraic varieties~\cite{HainAlg}. For $(C,S)$ a pair of a complex curve~$C$ and a finite subset of points~$S$, it gives rise to a diagram $T(\Gamma(C\setminus S,\Omega^1_C))\supset \{\textrm{2nd kind IIs}\}\to \mathcal{O}((\tilde C\setminus p^{-1}(S))^2)$, where $\{\textrm{2nd kind IIs}\}$ is the subset of the tensor algebra $T(\Gamma(C\setminus S,\Omega^1_C))$ of elements $x$ such that the map ${\rm II}_x : \Pi_1(C\setminus S)\to \mathbb{C}$ induced by iterated integration of $x$ factors through the map $\Pi_1(C\setminus S)\to \Pi_1(C) \times_{C^2}(C\setminus S)^2$, where $\Pi_1(X)$ is the fundamental groupoid of a topological space $X$. Let $\mathfrak g$ be a pro-nilpotent Lie algebra, and define ${\rm MC}_{\rm alg}(C,S,\mathfrak{g})$ to be the subset of all $J \in \Gamma(C\setminus S,\Omega^1_C)\,\hat{\otimes}\,\mathfrak g$ such that $d+J$ has no monodromy at the points of $S$ (i.e. it is conjugated to $d$ at the formal neighborhood of each $s\in S$). Then a pair $(J,P) \in {\rm MC}_{\rm alg}(C,S,\mathfrak g) \times \mathbb{C}[\exp (\mathfrak{g})]$  gives rise to an element of $\{\textrm{2nd kind IIs}\}$ and therefore to a function on $(\tilde C\setminus p^{-1}(S))^2$. The aim of the present paper is the construction of an element $J\in {\rm MC}_{\rm alg}(C,\infty,\mathfrak{g})$, where $C$ is a curve of genus~$\geq 1$ and $\infty\in C$, as well as its analogue on $C_n(C\setminus\infty)$, which is an element $J \in \Gamma(C_n(C\setminus\infty),\Omega^1_C)\,\hat\otimes\, \mathfrak{g}$ such that $d+J$ is conjugated to~$d$ at each formal neighborhood of $(C\setminus\infty)^k\times \infty\times (C\setminus\infty)^{n-1-k}$, and $dJ+(1/2)[J,J]=0$.

A pro-nilpotent Lie algebra $\hat{\mathfrak{t}}_{h,n}$ is associated in~\cite{Bez} to a pair of integers $n,h \geq 1$. For $n \geq 1$ and~$C$ a curve of genus $h \geq 1$, one constructs a pair $(\mathcal{P},\nabla_{\mathcal{P}})$, where~$\mathcal{P}$ is a principal $\exp(\hat{\mathfrak{t}}_{h,n})$-bundle over~$C^n$ and~$\nabla_\mathcal{P}$ is a holomorphic flat connection on the restriction of $\mathcal{P}$ to $C_n(C)$ (\cite{CEE} for $h=1$, \cite{EnrConfSp} for $h>1$; the constructions are based respectively on~\cite{Bernard, BernardHG}) and applies it to the computation of the pro-unipotent completion of the fundamental group of~$C_n(C)$. The data of a point $\infty \in C$ and of an isomorphism of the restriction of~$\mathcal{P}$ to~$(C\setminus\infty)^n$ with a trivial bundle~$\mathcal{P}_{\rm trv}$ defines a Maurer-Cartan element~$J$ on $C_n(C\setminus\infty)$ by the condition that $(\mathcal{P},\nabla_\mathcal{P})$ is isomorphic to $(\mathcal{P}_{\rm trv},d+J)$. If $n=1$ then~$J$ is an element of ${\rm MC}_{\rm alg}(C,\infty,\hat{\mathfrak{t}}_{h,n})$.

A similar construction was carried out in~\cite{LevinRacinet} when $h=1$, in which case $C$ is a punctured elliptic curve $E\setminus 0$. A principal bundle with flat connection $(\mathcal{P}_{\rm CEE},\nabla_{\rm CEE})$ over $C_{n+1}(E)/E=C_n(E\setminus 0)$ was constructed in~\cite{CEE}, and independently in~\cite{LevinRacinet} when $n=1$. In the latter case, an isomorphism of $\mathcal{P}_{\rm CEE}$ with $\mathcal{P}_{\rm trv}$ was constructed in~\cite{LevinRacinet}, \S~5, giving rise to an element $J_{\rm LR} \in {\rm MC}_{\rm alg}(E\setminus 0,\emptyset,\hat{\mathfrak{t}}_{1,2})$. 

\subsection{Contents of the article}\label{sect:main:results}

Throughout this article, we fix a Riemann surface~$C$ of genus $h\geq 1$, a universal cover $p:\tilde C\to C$, a point~$\infty$ on~$C$ and a choice of a compatible system~$e$ of classes of parametrisations of the fundamental groups of $C\setminus\infty$ (see \S\ref{sect:topology}).

\S\ref{sec2} is devoted to the proof of the triviality of the restriction to $(C\setminus\infty)^n$ of the principal bundle~$\mathcal{P}$ from~\cite{EnrConfSp}. The proof necessitates the construction of a holomorphic map\footnote{As~$\mathrm{exp}(\hat{\mathfrak G})$ is prounipotent, it is isomorphic to an infinite product of affine lines, hence we mean that the coordinates of~$g$ are meromorphic functions on~$\tilde C$ with poles at~$p^{-1}(\infty)$.} 
$\tilde g_{\tilde x,{\pmb \beta}}:\tilde C\setminus p^{-1}(\infty)\to \exp(\hat{\mathfrak G})$ with particular properties, where~$\mathfrak G$ is a free Lie algebra with $2h$ generators. This construction occupies \S\ref{sect:topology} to \S\ref{ssec:defgtildexbeta} (see Proposition~\ref{prop:tildeg}); it makes use of iterated integrals, and it depends on the input data $(\tilde x,{\pmb \beta})$, where~$\tilde x$ is a point in~$\tilde C$ and~$\pmb\beta$ is a family of~$h$ meromorphic differential forms on~$C$ introduced in \S\ref{def:omega}. We study the dependence on these data in \S\ref{ssec:dependencetildeg}. In \S\ref{ssec:trivialisation}, we introduce the topological Lie algebra $\hat{\mathfrak{t}}_{h,n}$, recall the construction of the principal $\exp(\hat{\mathfrak{t}}_{h,n})$-bundle $\mathcal{P}$ attached to~$e$, and derive its triviality from Proposition~\ref{prop:tildeg} (see Theorem~\ref{thm:2:28}). 

\S\ref{sec3} is devoted to the application of this result to the construction of a Maurer-Cartan element
${\pmb J}_{\tilde x,{\pmb \beta}}\in \Gamma(C_n(C\setminus\infty),\Omega^1_{C^n})\,\hat{\otimes}\,\hat{\mathfrak{t}}_{h,n}$. Such ${\pmb J}_{\tilde x,{\pmb \beta}}$ is determined by the above isomorphism of~$\mathcal{P}$ with~$\mathcal{P}_{\rm trv}$, which takes~$\nabla_\mathcal{P}$ to $d+{\pmb J}_{\tilde x,{\pmb \beta}}$. In order to make~${\pmb J}_{\tilde x,{\pmb \beta}}$ explicit, one introduces the Maurer-Cartan element ${\pmb K} \in \Gamma_{\rm rat}(\hat C^n,\Omega^1_{\tilde C^n})\,\hat\otimes\,\hat{\mathfrak{t}}_{h,n}$, where $\pi:\hat C\to C$ is the Schottky cover of $C$ (see \S\ref{Schottkyformalism}), such that there is an isomorphism of~$(\pi^n)^*(\mathcal{P})$ with the trivial bundle over~$\hat C^n$ which takes $(p^n)^*(\nabla_\mathcal{P})$ to $d+{\pmb K}$. Then ${\pmb J}_{\tilde x,{\pmb \beta}}$ is determined by the identity $d+{\pmb J}_{\tilde x,{\pmb \beta}}={\pmb g}_{\tilde x,{\pmb \beta}}(d+{\pmb K}){\pmb g}_{\tilde x,{\pmb \beta}}^{-1}$, where~${\pmb g}_{\tilde x,{\pmb \beta}}:(\hat C\setminus\pi^{-1}(\infty))^n\to\exp(\hat{\mathfrak{t}}_{h,n})$ is constructed out of the function $\tilde g_{\tilde x,{\pmb \beta}}$ from \S\ref{sec2} (see Theorem~\ref{thmJflat}). In \S\ref{sec:notationchange} we study the dependence of~${\pmb J}_{\tilde x,{\pmb \beta}}$ on~$\tilde x$ and~$\pmb \beta$.

In \S\ref{sec4}, we provide certain explicit decompositions of ${\pmb K}$ and~${\pmb J}_{\tilde x,{\pmb \beta}}$. The constituents of these decompositions are computed more precisely in \S\ref{sec5}. In particular, we show that the computation of~$\tilde g_{\tilde x,\pmb \beta}$ can be used to deduce expressions of one of the constituents of~${\pmb K}$, a family of 1-forms~$\omega_{i_1\cdots i_n}$, in terms of iterated integrals and residues of classical differential forms (Corollary~\ref{cor516}). When $n=1$, this is sufficient to obtain explicit formulas for~${\pmb K}$ (combining Corollary~\ref{cor515} with eq.~\eqref{eq:decompK1}) and~${\pmb J}_{\tilde x,{\pmb \beta}}$ (combining Definition~\ref{def:210611}, Lemma~\ref{lemma:decomp}, eq.~\eqref{eq4:211006}, Corollary~\ref{cor515} and Lemma~\ref{lem49}) for n=1. Examples of such formulas in low-degree can be found in \S\ref{sec:computations}. 

%Set also\footnote{As~$\mathrm{exp}(\hat{\mathfrak G})$ is prounipotent, it is isomorphic to an infinite product of affine lines, hence in the definition of~$\mathcal{G}_{\rho}$ we mean that the coordinates of~$g$ as above are meromorphic functions on~$\tilde C$ with poles at~$p^{-1}(\infty)$.}
%$$
%\mathcal G_\rho:=\Gamma(C\setminus\infty,\mathrm{exp}(\hat{\mathfrak G})_\rho)=\{g:\tilde C\setminus p^{-1}(\infty)\to 
%\mathrm{exp}(\hat{\mathfrak G})\,|\,
%\forall\gamma\in\mathrm{Aut}(\tilde C/C),
%\gamma^*g=g\cdot\rho(\gamma)^{-1}\}
%$$ 

\section{Trivialization of the restriction of $\mathcal{P}$ to $(C\setminus\infty)^n$}\label{sec2}

In this section we introduce a function $\tilde g_{\tilde x,\pmb\beta}:\tilde C\setminus p^{-1}(\infty)\to\exp (\hat{\mathfrak{G}})$ which depends on $\tilde x\in \tilde C$ and on a tuple $\pmb \beta$ of differential forms, and we use it to show that the restriction of~$\mathcal{P}$ to~$(C\setminus\infty)^n$ is trivial. 

In the background sections \S\ref{sect:topology} to \S\ref{sec:groupoids} we recall some facts on Riemann surfaces, iterated integrals and groupoids. We devote \S\ref{sec:24} to \S\ref{ssec:defgtildexbeta} to the construction of the function~$\tilde g_{\tilde x,\pmb\beta}$. In \S\ref{ssec:dependencetildeg} we study the dependence of~$\tilde g_{\tilde x,\pmb\beta}$ on~$\tilde x$ and~$\pmb\beta$. The proof of the triviality of the restriction of~$\mathcal{P}$ to~$(C\setminus\infty)^n$ is given in \S\ref{ssec:trivialisation}.

\subsection{The geometric setup}\label{sect:topology}

In \S\ref{ssec:univcover}, we recall the relationship between the fundamental group of a topological space and the group of automorphisms of its universal cover. We will give details of the proofs, in order to set up some notation and recall some notions that will be used throughout this article. We refer to \cite{Ha} for further details. In \S\ref{sect:RS}, we introduce the geometric setup of this article. To do so, we recall and demonstrate some facts about fundamental groups of closed Riemann surfaces. 

\subsubsection{Fundamental group and universal cover}\label{ssec:univcover}

For~$X$ a topological space and $x,y\in X$,  we denote by $\pi_1(X;x,y)$ the fundamental torsor of paths from~$x$ to~$y$. The composition of paths induces a map $\pi_1(X;x,y)\times\pi_1(X;y,w)\to\pi_1(X;x,w)$, 
$(\eta,\eta')\mapsto\eta\eta'$ (topologists' convention). %\textcolor{red}{FZ: modified}
 Let $p:\tilde X\to X$ be a universal cover. As~$\tilde X$ is simply connected, for any $\tilde x,\tilde y\in\tilde X$, the set $\pi_1(\tilde X;\tilde x,\tilde y)$ has exactly one element.
\begin{defn}\label{def:eta:xy:0202} 
For $\tilde x,\tilde y\in\tilde X$, we define $\eta_{\tilde x,\tilde y}$ as the image in $\pi_1(C;p(\tilde x),p(\tilde y))$ of the unique element of $\pi_1(\tilde C;\tilde x,\tilde y)$. 
\end{defn}
Let $x\in X$. We recall the following classical result.
\begin{lem}\label{lemma:standard}
(a) There is a right action of $\pi_1(X,x)$ on $p^{-1}(x)$, denoted $(\tilde x,\eta)\mapsto \tilde x\cdot\eta$, where 
$\tilde x\cdot\eta$ is the endpoint of the unique lift starting at $\tilde x$ of any representative of $\eta$.    

(b) There is a unique map $p^{-1}(x)\times\pi_1(X,x)\to\mathrm{Aut}(\tilde X/X)$, $(\tilde x,\eta)\mapsto 
\mathrm{aut}_\eta^{\tilde x}$, such that  $\tilde x\cdot \eta=\mathrm{aut}_\eta^{\tilde x}(\tilde x)$
(equality in~$\tilde X$). 

(c) For any $\tilde x\in p^{-1}(x)$, the map $\mathrm{iso}_{\tilde x} : \pi_1(X,x)\to\mathrm{Aut}(\tilde X/X)$, $\eta\mapsto
\mathrm{aut}_\eta^{\tilde x}$ is a group isomorphism, so $\mathrm{aut}_{\eta\eta'}^{\tilde x}
=\mathrm{aut}_\eta^{\tilde x}\circ \mathrm{aut}_{\eta'}^{\tilde x}$, with inverse given by $\theta\mapsto\eta_{\tilde x,\theta(\tilde x)}$.
\end{lem}

\proof (a) follows from the equality $\mathrm{lift}_{\tilde x}(\eta\eta')=\mathrm{lift}_{\tilde x}(\eta)
\mathrm{lift}_{\tilde x\cdot\eta}(\eta')$, where  $\mathrm{lift}_{\tilde x}(\eta)$ is the lift to~$\tilde X$
of~$\eta$ starting at~$x$ and the product is the  composition of paths in~$\tilde X$. (b) follows from the 
universality of~$\tilde X$. Let us prove (c). For $\eta,\eta'\in\pi_1(X,x)$, 
\begin{equation}\label{intemediate:equality}
\mathrm{aut}_{\eta'}^{\mathrm{aut}_\eta^{\tilde x}(\tilde x)}\circ\mathrm{aut}_\eta^{\tilde x}
=\mathrm{aut}_{\eta\eta'}^{\tilde x}
\end{equation}
Indeed, the image by the right-hand side of~$\tilde x$ is $\tilde x\cdot\eta\eta'$, while its image by 
the left-hand side is (using the equality in~(b)) $(\tilde x\cdot\eta)\cdot\eta'=\tilde x\cdot\eta\eta'$. 

On the other hand, if $\theta\in\mathrm{Aut}(\tilde X/X)$, if $\eta\in\pi_1(X,x)$ and $\tilde x\in p^{-1}(x)$, then 
the fact that the lift of $\eta$ starting at $\theta(\tilde x)$ is the image by~$\theta$ of the lift of~$\eta$ starting at~$\tilde x$ implies 
\begin{equation}\label{TaggedIdentity}
\theta(\tilde x\cdot\eta)=\theta(\tilde x)\cdot \eta
\end{equation} 
(equality in~$\tilde X$). This equality
can be rewritten as $\theta\circ\mathrm{aut}_\eta^{\tilde x}(\tilde x)=\mathrm{aut}_\eta^{\theta(\tilde x)}
\circ\theta(\tilde x)$ by using the equality in~(b). By the universality of~$\tilde X$, this implies the equality 
$\theta\circ\mathrm{aut}_\eta^{\tilde x}=\mathrm{aut}_\eta^{\theta(\tilde x)}\circ\theta$ in 
$\mathrm{Aut}(\tilde X/X)$, which by replacing~$\eta$ by~$\eta'$ and by plugging $\theta=\mathrm{aut}_\eta^{\tilde x}$
implies that $\mathrm{aut}_\eta^{\tilde x}\circ\mathrm{aut}_{\eta'}^{\tilde x}
=\mathrm{aut}_{\eta'}^{\mathrm{aut}_\eta^{\tilde x}(\tilde x)}\circ\mathrm{aut}_\eta^{\tilde x}$. 
This identity, together with~\eqref{intemediate:equality}, implies that $\mathrm{iso}_{\tilde x}$ is a group morphism.  

We leave to the reader to verify that $\theta\mapsto\eta_{\tilde x,\theta(\tilde x)}$ is inverse to $\mathrm{iso}_{\tilde x}$ (see~\cite{Ha}).
\hfill\qed\medskip

%\begin{lem}\label{lemma:standard}
%For any $\tilde x$ in $\tilde C$, the map $\mathrm{iso}_{\tilde x}:\pi_1(C,p(\tilde x))\to\mathrm{Aut}(\tilde C/C)$ taking 
%$\eta\in\pi_1(C,p(\tilde x))$ to the unique automorphism taking $\tilde x$ to the element of $p^{-1}(p(\tilde x))$
%obtained as the endpoint of the lift of $\eta$ starting at $\tilde x$, is a group isomorphism.
%\end{lem}

\subsubsection{Fundamental groups of Riemann surfaces}\label{sect:RS}

%\textcolor{red}{FZ: title to be verifed}
Let $h\geq 1$ be an integer. 
Let~$C$ be a closed Riemann surface of genus~$h$, let $\infty\in C$ be a point, let $p:\tilde C\to C$ be a universal cover. 
%and let $\widetilde\infty\in p^{-1}(\infty)$. 

\begin{defn}
Let $x\in C\setminus\infty$. To a simply connected neighborhood $U\subset C$ of~$\infty$, a point~$\infty'$ in~$U\setminus\infty$, and a counterclockwise loop $l_\infty\subset U$ around~$\infty$ based at~$\infty'$, one attaches the set of conjugations of~$l_\infty$ by all the elements of the fundamental torsor of paths $\pi_1(C\setminus\infty;x,\infty')$. This is a well-defined conjugacy class of $\pi_1(C\setminus\infty,x)$, which in independent of the choices of~$U$ and~$\infty'$. We call it \emph{the conjugacy class attached to~$\infty$}, and we denote it~$c_{\infty,x}$. 
\end{defn}

\begin{lem}\label{lemma:point}
For $x \in C\setminus\infty$, the kernel of the surjective group morphism $\pi_1(C\setminus\infty,x)\to \pi_1(C,x)$ induced by $C\setminus\infty\subset C$ coincides with the normal subgroup $(c_{\infty,x})$ generated by $c_{\infty,x}$.
\end{lem}

\begin{proof} 
This is a standard consequence of the van Kampen theorem.
\end{proof}

\begin{defn}\label{def:1:5:2403}
Let $\langle A_i,B_i\rangle$ be the free group on~$2h$ generators $A_1,\ldots ,A_h,B_1,\ldots ,B_h$, and let $(A_i,B_i):=A_iB_iA^{-1}_iB^{-1}_i$. For $x\in C\setminus\infty$, we define~$E(x)$ as the quotient set 
$$
\mathrm{Iso}_\infty(\pi_1(C\setminus\infty,x),\langle A_i,B_i\rangle)/\big(\prod_i(A_i,B_i)\big),
$$ 
where 
$\mathrm{Iso}_\infty(\pi_1(C\setminus\infty,x),\langle A_i,B_i\rangle)$ is the set of group isomorphisms $\pi_1(C\setminus\infty,x)
\to\langle A_i,B_i\rangle$ taking the conjugacy class attached to $\infty$ to the conjugacy class of $\prod_i(A_i,B_i)$, and where $(\prod_i(A_i,B_i))$ is the normal subgroup of $\langle A_i,B_i\rangle$ generated by $\prod_i(A_i,B_i)$, 
which is acting on $\mathrm{Iso}_\infty(\pi_1(C\setminus\infty,x),\langle A_i,B_i\rangle)$ by conjugation. 
\end{defn}

\begin{lem}\label{lemma110}
For $\tilde x\in\tilde C$ and $f\in E(p(\tilde x))$, there exists 
a unique group isomorphism $f^{\rm fill} : \pi_1(C,p(\tilde x))\to
\langle A_i,B_i\rangle/(\prod_i(A_i,B_i))$ such that, for any group isomorphism~$\tilde f$ above~$f$, the diagram 
\begin{equation}\label{dcetdet:02032021}
\begin{tikzcd}
\pi_1(C\setminus\infty,p(\tilde x)) \arrow[r,"{\tilde f}"] \arrow[d]
& \langle A_i,B_i\rangle \arrow[d]  \\
\pi_1(C,p(\tilde x)) \arrow[r,"{f^{\rm fill}}"]
& \langle A_i,B_i\rangle/(\prod_i(A_i,B_i))
\end{tikzcd}
\end{equation}
commutes; $f^{\mathrm{fill}}$ should be viewed as a ``hole-filled'' version of $f$.   
\end{lem}

\begin{proof}
For a given~$\tilde f$, the existence and uniqueness of~$f^{\rm fill}$ such that~\eqref{dcetdet:02032021} commutes follows from the fact that~$\tilde f$ takes 
the kernel of the left vertical morphism to its analogue on the right. If~$\tilde f$ is replaced by~$\tilde{\tilde f}$, then~$\tilde{\tilde f}$ is obtained
from~$\tilde f$ by conjugation by an element of $(\prod_i(A_i,B_i))$, whose image in $\langle A_i,B_i\rangle/(\prod_i(A_i,B_i))$ is~1. It follows that~\eqref{dcetdet:02032021} commutes also with~$\tilde f$ replaced by~$\tilde{\tilde f}$. 
\end{proof}

\begin{lem}\label{lemma:28122020}
Let $\tilde x,\tilde y\in\tilde C\setminus p^{-1}(\infty)$. Let $i_{\hat x,\hat y}:\pi_1(C\setminus\infty,p(\tilde x))\to
\pi_1(C\setminus\infty,p(\tilde y))$ 
be the group isomorphism induced by the choice of a universal cover $\widetilde{C\setminus\infty}\to\tilde C\setminus p^{-1}(\infty)$ 
and of lifts $\hat x,\hat y\in\widetilde{C\setminus\infty}$ of $\tilde x,\tilde y$. 
The bijection $E(p(\tilde x))\to E(p(\tilde y))$ induced by composition with $i_{\hat y,\hat x}=i^{-1}_{\hat x,\hat y}$ is independent on the choice of $\hat x,\hat y$ and will be denoted $\mathrm{bij}_{\tilde x,\tilde y}$. 
For any $\tilde x,\tilde y,\tilde z\in\tilde C\setminus p^{-1}(\infty)$, one has $\mathrm{bij}_{\tilde y,\tilde z}\circ 
\mathrm{bij}_{\tilde x,\tilde y}=\mathrm{bij}_{\tilde x,\tilde z}$. 
\end{lem}

\begin{proof}
The replacement of the lifts $\hat x,\hat y$ has the effect of post-composing the isomorphism 
$\mathrm{iso}_{\hat x,\hat y}$ by the conjugation by an element of $\mathrm{ker}(\pi_1(C\setminus\infty,p(\tilde y))\to
\pi_1(C,p(\tilde y)))$. By Lemma~\ref{lemma:point}, this kernel is nothing but $(c_{\infty,p(\tilde y)})$, which is equal to the preimage of the subgroup $(\prod_i(A_i,B_i))$ of 
$\langle A_i,B_i\rangle$ under an element of $\mathrm{Iso}_\infty(\pi_1(C\setminus\infty,p(\tilde x)),\langle A_i,B_i\rangle)$. The 
composition statement is obvious. 
\end{proof} 

\begin{defn}\label{def:1:7:2302}
Define $E$ to be the set of assignments $\tilde C\setminus  p^{-1}(\infty)\ni\tilde z\mapsto e_{\tilde z}\in E(p(\tilde z))$, such that 
$\mathrm{bij}_{\tilde z,\tilde w}(e_{\tilde z})=e_{\tilde w}$ for any $\tilde z,\tilde w\in\tilde C\setminus p^{-1}(\infty)$. 
\end{defn}

\begin{lem}
For any $\tilde x\in\tilde C\setminus p^{-1}(\infty)$, there is a bijection between $E$ and $E(p(\tilde x))$, 
taking the assignment $\tilde C\setminus p^{-1}(\infty)\ni\tilde z\mapsto e_{\tilde z}\in E(p(\tilde z))$ to 
$e_{\tilde x}\in E(p(\tilde x))$. 
\end{lem}

\begin{proof}
Using the composition identities of Lemma~\ref{lemma:28122020}, one checks that the map taking $e\in E(p(\tilde x))$ 
to the assignment $\tilde C\setminus p^{-1}(\infty)\ni\tilde z\mapsto  \mathrm{bij}_{\tilde x,\tilde z}(e)\in E(p(\tilde z))$ is the 
inverse bijection. 
\end{proof}

\begin{lem}\label{lemma:28122020b}
Let $\tilde x\in \tilde C\setminus p^{-1}(\infty)$. Consider the composition
\begin{align*}
&E\stackrel{\sim}{\to} E(p(\tilde x))=\mathrm{Iso}_\infty(\pi_1(C\setminus\infty,p(\tilde x)),\langle A_i,B_i\rangle)/(\prod_i(A_i,B_i))\\ &
\to\mathrm{Iso}(\pi_1(C,p(\tilde x)),\langle A_i,B_i\rangle/(\prod_i(A_i,B_i))) \to\mathrm{Iso}(\mathrm{Aut}(\tilde C/C),
\langle A_i,B_i\rangle/(\prod_i(A_i,B_i))),
\end{align*}
where the second map takes $f \in \mathrm{Iso}_\infty(\pi_1(C\setminus\infty,p(\tilde x)),\langle A_i,B_i\rangle)/(\prod_i(A_i,B_i))$ to the isomorphism $f^{\rm fill}:\pi_1(C,p(\tilde x))\tilde{\to} \langle A_i,B_i\rangle/(\prod_i(A_i,B_i))$ (see Lemma~\ref{lemma110}), and the third map is the composition with the isomorphism $\mathrm{iso}^{-1}_{\tilde x}:\mathrm{Aut}(\tilde C/C)\,\tilde\to\, \pi_1(C,p(\tilde x))$ (see Lemma~\ref{lemma:standard}). This defines a map 
\begin{equation}\label{triangle:02032021}
E\to \mathrm{Iso}(\mathrm{Aut}(\tilde C/C),
\langle A_i,B_i\rangle/(\prod_i(A_i,B_i))), \quad e\mapsto\mathrm{can}_e:=e_{\tilde x}^{\rm fill}\circ \mathrm{iso}^{-1}_{\tilde x}
\end{equation}
which is independent of $\tilde x$. 
\end{lem}

\begin{proof}
Let $j_{\tilde x,\tilde y}:\pi_1(C,p(\tilde x))\to \pi_1(C,p(\tilde y))$ be the isomorphism corresponding to 
a pair of points $\tilde x,\tilde y\in \tilde C\setminus p^{-1}(\infty)$, and let $\hat x,\hat y$ be any two lifts of $\tilde x,\tilde y$ to $\widetilde{C\setminus\infty}$. The independence of the map of interest in $\tilde x$ follows from the commutativity of the following diagrams:
\[
\begin{tikzcd}
E \arrow[r] \arrow[dr] & E(p(\tilde x)) \arrow[d, "{{\rm bij}_{\tilde x,\tilde y}}"] \\
  & E(p(\tilde y))
\end{tikzcd}
\]
\[
\begin{tikzcd}
\pi_1(C\setminus\infty,p(\tilde x)) \arrow[r] \arrow[d,"{i_{\hat x,\hat y}}"]
& \pi_1(C,p(\tilde x)) \arrow[d, "{j_{\tilde x,\tilde y}}"]  \arrow[r,"{\rm iso_{\tilde x}}"] & {\rm Aut}(\tilde C/C) \\
\pi_1(C\setminus\infty,p(\tilde y)) \arrow[r]
& \pi_1(C,p(\tilde y)) \arrow[ru,"{\rm iso_{\tilde y}}"] &
\end{tikzcd}
\]
\end{proof}

In the rest of the paper, we choose an element $e\in E$, which by Lemma~\ref{lemma:28122020b} gives rise to a group isomorphism 
$\mathrm{Aut}(\tilde C/C)\simeq \langle A_i,B_i\rangle/(\prod_i(A_i,B_i))$ and by Lemma~\ref{lemma:standard} and abelianization, 
to a group isomorphism $H_1(C,\mathbb Z)\simeq\oplus_{i=1}^h(\mathbb Z\overline A_i\oplus\mathbb Z\overline B_i)$.

\subsection{1-forms on $C$ and iterated integrals}\label{def:omega}  

\subsubsection{1-forms on $C$}
Let us choose a $h$-tuple $\pmb \beta=(\beta_{1},\ldots,\beta_{h})$ of elements of $\Gamma(C\setminus\infty,\Omega^1_C)$ such that, for any $i,j\in[\![1,h]\!]$, $\int_{\overline A_i}\beta_{j}=0$,  $\int_{\overline B_i}\beta_{j}=\delta_{ij}$. Moreover, let $\pmb\omega:=(\omega_1,\ldots,\omega_h)$ be the unique $h$-tuple of 
holomorphic 1-forms on $C$ such that $\int_{\overline A_i}\omega_j=\delta_{ij}$ for $i,j\in
[\![1,h]\!]$. We denote $\tau_{ij}:=\int_{\overline B_i}\omega_j$ for any $i,j\in
[\![1,h]\!]$, and we consider the $h$-tuple $\pmb\alpha=(\alpha_1,\ldots,\alpha_h)$ defined by $\alpha_i:=
\omega_i-\sum_{j=1}^h\tau_{ij}\beta_j$. The $2h$-tuple $\pmb\gamma:=(\pmb\alpha,
\pmb\beta)$ is then a set of representatives of a basis of the algebraic de Rham cohomology group $H_{\rm dR}^1(C\setminus\infty)$, dual to the basis $(\overline C_1,\ldots ,\overline C_{2h}):=(\overline A_1,\ldots ,\overline A_h,\overline B_1,\ldots ,\overline B_h)$ of the Betti cohomology $H_{\rm B}^1(C\setminus\infty)$, because $\int_{\overline C_i}\gamma_j=\delta_{ij}$ for $i,j\in[\![1,2h]\!]$. Such choice of representatives only depends on the choice of the $h$-tuple $\pmb \beta$.

\subsubsection{Iterated integrals}\label{sssec:itint}

\begin{defn}
The iterated integral of smooth 1-forms $\gamma_1,\ldots ,\gamma_r$ over a piecewise smooth path $\eta : [0,1]\to M$ on a smooth manifold $M$ is defined as
$$
\int_\eta\gamma_1\cdots\gamma_r\,:=\,\int_{0\leq t_1\leq \cdots \leq t_r\leq 1}(\eta^*\gamma_1)(t_1)\cdots (\eta^*\gamma_1)(t_r)\,.
$$
We call $r$ the \emph{length} of the iterated integral.
\end{defn}
We recall the following well-known properties:
\begin{itemize}
\item \emph{Composition of paths:}
\begin{equation}\label{eq:comppaths}
\int_{\eta_1\eta_2}\gamma_1\cdots\gamma_r\,=\,\sum_{i=0}^r\int_{\eta_1}\gamma_1\cdots \gamma_i\int_{\eta_2}\gamma_{i+1}\cdots \gamma_{r}.
\end{equation}
\item \emph{Inversion of paths:}
\begin{equation*}
\int_{\eta^{-1}}\gamma_1\cdots \gamma_r\,=\,(-1)^r\int_{\eta}\gamma_r\cdots \gamma_1.
\end{equation*}
\item \emph{Shuffle product:}
\begin{equation}\label{eq:shuffle}
\int_{\eta}\gamma_1,\cdots \gamma_r\int_{\eta}\gamma_{r+1}\cdots \gamma_{r+s}\,=\,\sum_{\sigma\in\shuffle (r,s)}\int_{\eta}\gamma_{\sigma (1)}\cdots \gamma_{\sigma (r+s)}.
\end{equation}
\end{itemize}

Suppose now that $\gamma_1,\ldots ,\gamma_r\in\Gamma(C\setminus\infty,\Omega^1_C)$, and let us denote in the same way their pull-back to the universal cover $\widetilde{C\setminus\infty}$. Because these forms are holomorphic and $\widetilde{C\setminus\infty}$ is simply connected, the result of the integral
$\int_{\tilde x}^{\tilde y} \gamma_1\cdots \gamma_r$
does not depend\footnote{Indeed, differentiation induces an isomorphism between the subspace of $\Gamma(U,\mathcal{O}_U)$ of functions vanishing at~$\tilde x$ and $\Gamma(U,\Omega^1_U)$, where $U:=\widetilde{C\setminus \infty}$, and if~$I$ is the inverse isomorphism, then for any~$\eta$ one has 
$\int_{\eta}\gamma_1\cdots \gamma_n=I(\gamma_n\cdots I(\gamma_2\cdot I(\gamma_1)))(\tilde y)$, a quantity which does not depend on $\eta$.} on the integration path from $\tilde x\in \widetilde{C\setminus\infty}$ to $\tilde y\in\widetilde{C\setminus\infty}$. Let $x \in C\setminus\infty$. For $\eta \in \pi_1(C\setminus \infty,x)$, consider ${\rm aut}_\eta^{\tilde x}$ as defined in \S\ref{ssec:univcover}. The pairs $(\tilde x,{\rm aut}_\eta^{\tilde x}) \in \widetilde{C\setminus\infty}^2$ are all related by the action of ${\rm Aut}(\widetilde{C\setminus\infty}/C\setminus\infty)$. This implies that $\int_{\tilde x}^{{\rm aut}_\eta^{\tilde x}(\tilde x)}\gamma_1\ldots \gamma_r$ is independent of the choice of $\tilde x$, and we set
\begin{equation}\label{itintovercycle}
\int_\eta\gamma_1\ldots \gamma_r:=\int_{\tilde x}^{{\rm aut}_\eta^{\tilde x}(\tilde x)}\gamma_1\ldots \gamma_r\,.
\end{equation}

\subsection{Groupoids}\label{sec:groupoids} Recall that a groupoid is a small category where all the morphisms are invertible. If~$G$ is a groupoid, we denote by $S$ its set of objects and, for $x,y \in S$, we denote by $G(x,y)$ the set of morphisms from~$x$ to~$y$. 

\subsubsection{1-nerves} To a groupoid~$G$, one attaches a set $\pmb{G}:=\{(x,y,\gamma)\,|\,x,y \in S,\gamma \in G(x,y)\}$, called \emph{the 1-nerve of $G$}. This set is equipped with source and target maps $s,t:\pmb{G}\to S$, with $s(x,y,\gamma):=x$, $t(x,y,\gamma):=y$. The product, unit and inverse induce maps $\pmb{G}\times_S \pmb{G}\to\pmb{G}$, $S\to \pmb{G}$ and $\pmb{G}\to \pmb{G}$ satisfying some axioms (e.g. associativity). This gives rise to an equivalence between groupoids and pairs of sets $(\pmb{G},S)$ equipped with maps $s,t:\pmb{G}\to S$ as well as $m_{\pmb{G}} : \pmb{G} \times_S \pmb{G}\to \pmb{G}$, $S\to \pmb{G}$ and $\pmb{G}\to \pmb{G}$ satisfying these axioms. 

Explicitly, if $G,G'$ are two groupoids and $(\pmb{G},S,s,t)$ and $(\pmb{G}',S',s',t')$ are the corresponding 
sets and maps, then a groupoid morphism $G\to G'$ is the same as a pair of maps $\pmb{f}:\pmb{G}\to \pmb{G}'$, $f:S\to S'$ such that $f \circ s=s' \circ \pmb{f}$, $f \circ t=t' \circ \pmb{f}$ (which guarantees that $\pmb{f}$ induces a map $\pmb{f}^{(2)}:\pmb{G} \times_S \pmb{G}\to \pmb{G}' \times_{S'} \pmb{G}'$), and $\pmb{f} \circ m_{\pmb{G}}=m_{\pmb{G}'} \circ \pmb{f}^{(2)}$. 

In particular, if~$G'$ is the groupoid corresponding to a group $K$, a morphism  $G\to G'$ is the same  as a map $\pmb{f}:\pmb{G}\to K$, such that $\pmb{f} \circ m_{\pmb{G}}=m_K \circ \pmb{f}^{(2)}$, where $m_K:K \times K\to K$ is the product map.

\subsubsection{Quotients}\label{sec:groupoidquotient}
Let $G$ be a groupoid with set of objects $S$. For $x\in S$, we set $G(x):=G(x,x)$. If $N\triangleleft G(x)$ is a normal subgroup, the groupoid quotient $G/N$ is the groupoid with the same set of objects as~$G$ and such that for 
$y,z\in S$, $(G/N)(y,z)$ is the quotient of $G(y,z)$ by the equivalence 
relation generated by $g\circ f\sim g\circ n\circ f$, where $f\in G(y,x)$, $g\in G(x,z)$, $n\in N$. 
In this situation, if~$G'$ is a groupoid, there is a bijection between $\mathrm{Hom}(G/N,G')$
and the subset of $\mathrm{Hom}(G,G')$ of all functors $F:G\to G'$ such that~$N$
is contained in the kernel of the group morphism $F \big|_{G(x)}:G(x)\to G'(F(x))$, the functor 
$\tilde F:G/N\to G'$ associated to~$F$ being such that its composition with the projection 
$G\to G/N$ coincides with~$F$.

\subsubsection{The groupoid $\Pi_1(X)$} 

For $X$ a topological space, we denote by $\Pi_1(X)$ its fundamental groupoid. Its set of objects is  $X$
and for $x,y\in X$, the set of morphisms $\Pi_1(X)(x,y)$ from $x$ to $y$ is the fundamental torsor
$\pi_1(X;x,y)$ of paths from $x$ to $y$. We denote by $\pmb\Pi_1(X)$ the 1-nerve of $\Pi_1(X)$, i.e. the set of triples $(x,y,\gamma)$
where $x,y\in X$ and $\gamma\in\pi_1(X;x,y)$.

Note that if $x\in X$, then there is a tautological groupoid morphism $\iota_x : \pi_1(X,x)\to \Pi_1(X)$, 
with underlying map between sets of objects $\{*\}\to X$, $*\mapsto x$. The corresponding 
map between 1-nerves is then $\pmb\iota_x : \pi_1(X,x)\to\pmb\Pi_1(X)$, $\gamma\mapsto(x,x,\gamma)$. 

Define also $\mathrm{diag} : X\to \pmb\Pi_1(X)$ to be the map $x\mapsto (x,x,1)$. 

\subsection{The groupoid morphism $F_{\pmb\Lambda} : \Pi_1(C\setminus\infty)\to \mathrm{exp}(\hat{\mathfrak G})$}\label{sec:24}

Let $\mathfrak G:=\text{Lie}(a_1,\ldots ,a_h,b_1,\ldots ,b_h)$ be the free complex Lie algebra with generators $a_i,b_i$, 
$i\in[\![1,h]\!]$. This Lie algebra is equipped with a $\mathbb Z_{\geq 0}$-grading by $\mathrm{deg}(a_i)=\mathrm{deg}(b_i)=1$. 
We denote by $\hat{\mathfrak G}$ the corresponding completion. The degree completion $U(\mathfrak G)^\wedge$ of the 
enveloping algebra of $\mathfrak G$ is a topological Hopf algebra; the exponential map $\mathrm{exp} : 
U(\mathfrak G)^\wedge_+\to U(\mathfrak G)^\wedge$ (where the source is the direct product of the positive degree 
components of $U(\mathfrak G)$) sets up a bijection between $\hat{\mathfrak G}$ and its group $\mathcal G(U(\mathfrak G)^\wedge)$
of group-like elements, which will henceforth be denoted $\mathrm{exp}(\hat{\mathfrak G})$. 
Throughout this section, we will use the notation $(c_1,\ldots ,c_{2h}):=(a_1,\ldots ,a_h,b_1,\ldots ,b_h)$. 

\begin{lem}\label{lemme:3:1:0202}
For $\pmb\Lambda=(\Lambda_1,\ldots ,\Lambda_{2h})\in\hat{\mathfrak G}^{2h}$, there is a unique map $F_{\pmb\Lambda}: \pmb\Pi_1(C\setminus\infty)\to \mathrm{exp}(\hat{\mathfrak G})$ which is holomorphic in the first two entries 
$x,y\in C\setminus\infty$ and such that 

(a) the map $F_{\pmb\Lambda}\circ\mathrm{diag} : C\setminus\infty\to \mathrm{exp}(\hat{\mathfrak G})$ is equal to $1$. 

(b) $F_{\pmb\Lambda}$ satisfies the differential equation 
\begin{equation}\label{eq:diffeqFLambda}
dF_{\pmb\Lambda}\,=\,F_{\pmb\Lambda}\cdot t^*\bigg(\sum_{i=1}^{2h}\,\Lambda_i\,\gamma_i\bigg)\,-\,s^*\bigg(\sum_{i=1}^{2h}\,\Lambda_i\,\gamma_i\bigg)\cdot F_{\pmb\Lambda}\,. 
\end{equation}
\end{lem}

\begin{proof}

Recall that the algebra $\mathbb C\langle\langle c_1,\ldots,c_{2h}\rangle\rangle$ is equipped with a topological Hopf algebra 
structure, defined by the condition that the elements $c_i$, $i\in[\![1,2h]\!]$, are primitive. Then $\hat{\mathfrak G}$ (resp. 
$\mathrm{exp}(\hat{\mathfrak G})$) is its Lie algebra of primitive elements (resp. its group of group-like elements). 

The existence and uniqueness of a solution to~$(a)$ and~$(b)$ valued in 
$\mathbb C\langle\langle c_1,\ldots,c_{2h}\rangle\rangle$ 
is standard, the map $F_{\pmb \Lambda}$ being given in terms of iterated integrals (see \S\ref{sssec:itint}) by\footnote{The fact that this map is well-defined, i.e. that it only depends on the endpoints of~$\eta$ and on its homotopy class, follows from the holomorphicity of~$\gamma_{\pmb \Lambda}$.} 
\begin{equation}\label{expansion:0202}
(x,y,\eta)\to 1\,+\,\int_{\eta}\gamma_{\pmb \Lambda}\,+\,\int_{\eta}\gamma_{\pmb \Lambda}\,\gamma_{\pmb \Lambda}\,+\,\int_{\eta}\gamma_{\pmb \Lambda}\,\gamma_{\pmb \Lambda}\,\gamma_{\pmb \Lambda}\,+\,\cdots\,,
\end{equation}
where $\gamma_{\pmb \Lambda}:=\sum_{i=1}^{2h}\Lambda_i\gamma_i$. 

The fact that the values of the function $F_{\pmb\Lambda}$ belong to the group $\mathrm{exp}(\hat{\mathfrak G})$ of group-like elements in 
$\mathbb C\langle\langle c_1,\ldots,c_{2h}\rangle\rangle$ is also standard, and follows from the fact that the form $\gamma_{\pmb\Lambda}$ is 
valued in $\hat{\mathfrak G}$, and that the values of $F_{\pmb\Lambda}$ are path-ordered exponentials of 
$\gamma_{\pmb\Lambda}$. 
\end{proof} 
Notice that $F_{\pmb \Lambda}$ also depends on the choice of~$\pmb \beta$ (see \S\ref{def:omega}), which determines the differential forms~$\gamma_i$. We will not keep track of such dependence for the moment, but it will play a role after \S\ref{ssec:defgtildexbeta}.

\begin{lem}\label{lem:2912a} 
The map $F_{\pmb\Lambda} : \pmb\Pi_1(C\setminus\infty)\to \mathrm{exp}(\hat{\mathfrak G})$ is induced by a groupoid morphism 
$F_{\pmb\Lambda} : \Pi_1(C\setminus\infty)\to \mathrm{exp}(\hat{\mathfrak G})$, whose underlying map between 
sets of objects is the only map $C\setminus\infty\to\{*\}$.
\end{lem}

\begin{proof}
As remarked at the end of \S\ref{sec:groupoids}, a groupoid morphism $F_{\pmb\Lambda}:\Pi_1(C\setminus\infty)\to 
\mathrm{exp}(\hat{\mathfrak G})$ is the same as a map $F_{\pmb\Lambda}:\pmb\Pi_1(C\setminus\infty)\to \mathrm{exp}(\hat{\mathfrak G})$ such that $F_{\pmb\Lambda}\circ m_{\pmb\Pi_1(C\setminus\infty)}=m_{\mathrm{exp}(\hat{\mathfrak G})}\circ F_{\pmb\Lambda}^{(2)}$, which follows in our case from the composition-of-paths property~\eqref{eq:comppaths} of iterated integrals.
\end{proof}

\subsection{The function $(\rho,f)\mapsto\pmb\Lambda(\rho,f)$}

It follows from Lemma~\ref{lem:2912a} that for any $\pmb\Lambda\in\hat{\mathfrak G}^{2h}$ and any $\tilde x\in \tilde C\setminus p^{-1}(\infty)$ the composite map 
$F_{\pmb\Lambda}\circ\iota_{p(\tilde x)} : \pi_1(C\setminus\infty,p(\tilde x))\to \mathrm{exp}(\hat{\mathfrak G})$ is a 
group morphism. 

\begin{lem}\label{lem:29121529}
For any $f\in\mathrm{Iso}_\infty(\pi_1(C\setminus\infty,p(\tilde x)),\langle A_i,B_i\rangle)$ and $\rho\in\mathrm{Hom}(\langle A_i,B_i\rangle,\mathrm{exp}(\hat{\mathfrak G}))$ there is a unique solution $\pmb\Lambda(\rho,f)\in\hat{\mathfrak G}^{2h}$ 
to the equation
\begin{equation}\label{identity:Lambda}
F_{\pmb\Lambda}\circ\iota_{p(\tilde x)}\circ f^{-1}=\rho
\end{equation}
(equality of group morphisms $\langle A_i,B_i\rangle\to \mathrm{exp}(\hat{\mathfrak G})$).
\end{lem}

\begin{proof}
Denote by $\mathrm{hol}_{f}=(\mathrm{hol}_{f}^1,\ldots ,\mathrm{hol}_{f}^{2h})$ the self-map of $\hat{\mathfrak G}^{2h}$ given 
by $\pmb\Lambda\mapsto(\mathrm{log} F_{\pmb\Lambda}\circ\iota_{p(\tilde x)}\circ f^{-1}(A_1),\ldots ,\mathrm{log} F_{\pmb\Lambda}
\circ\iota_{p(\tilde x)}\circ f^{-1}(B_h))$. Set $\pmb\Lambda_0=(c_i)_{i\in[\![1,2h]\!]}$. By~\eqref{expansion:0202}, one has 
the degree expansion $\mathrm{hol}_f^i(\pmb\Lambda_0)=\sum_{j=1}^{2h}\int_{C_i}\gamma_jc_j+l_i^f$, where 
$l_i^f\in U(\mathfrak G)^\wedge_{\geq 2}$. Therefore  $\mathrm{hol}_f^i(\pmb\Lambda_0)=c_i+l_i^f$.  

It follows from Lemma~\ref{lemme:3:1:0202} that $F_{\pmb\Lambda_0}$ takes its values in $\mathrm{exp}(\hat{\mathfrak G})$, 
which implies that $l_i^f\in\hat{\mathfrak G}_{\geq 2}$. If now $\pmb\Lambda\in\hat{\mathfrak G}^{2h}$ is arbitrary, there is a unique 
endomorphism of $\hat{\mathfrak G}$ induced by $c_i\mapsto\Lambda_i$, which we denote by 
$a\mapsto \mathrm{ev}_{\pmb\Lambda}(a)$; it naturally extends to 
endomorphisms of $\mathbb C\langle\langle c_1,\ldots,c_{2c}\rangle\rangle$ and of its subset $\mathrm{exp}(\hat{\mathfrak G})$ of group-like 
elements. Then the 
function $F_{\pmb\Lambda}$ is the image of $F_{\pmb\Lambda_0}$ by this endomorphism. In particular, 
$\mathrm{hol}_i^f(\pmb\Lambda)=\mathrm{ev}_{\pmb\Lambda}(c_i+l_i^f)$. 

The endomorphism of $\hat{\mathfrak G}$ induced by $c_i\mapsto c_i+l^{f}_i$ for 
$i\in[\![1,2h]\!]$ is invertible, its inverse being given by $c_i\mapsto c_i+m^{f}_i$ for certain elements 
$m^{f}_1,\ldots ,m^{f}_{2h}\in\hat{\mathfrak G}_{\geq 2}$. There is an anti-morphism from the semigroup of 
endomorphisms of the topological Lie algebra $\hat{\mathfrak G}$ to that of self-maps of $\hat{\mathfrak G}^{2h}$, taking 
the endomorphism $\theta$ to the map $\pmb\Lambda\mapsto(\mathrm{ev}_{\pmb\Lambda}(\theta(c_1)),\ldots,
\mathrm{ev}_{\pmb\Lambda}(\theta(c_{2h})))$. It follows that 
$\mathrm{hol}_{f}$ is invertible, with inverse $\mathrm{hol}_{f}^{-1}$ given by 
$\pmb\alpha=(\alpha_1,\ldots ,\alpha_{2h})\mapsto(\mathrm{ev}_{\pmb\alpha}(c_1+m_1^{f}),\ldots,
\mathrm{ev}_{\pmb\alpha}(c_{2h}+m_{2h}^{f}))$. 

Then~\eqref{identity:Lambda} is equivalent to $\mathrm{hol}_{f}(\pmb\Lambda)=(\mathrm{log}\rho(A_1),\ldots ,
\mathrm{log}\rho(B_h))$. The unique solution is $\pmb\Lambda=\mathrm{hol}_{f}^{-1}(\mathrm{log}\rho(A_1),\ldots ,
\mathrm{log}\rho(B_h))$, which is therefore the value of $\pmb\Lambda(\rho,f)$. 
\end{proof}
Notice that such $\Lambda(\rho,f)$ implicitly depends also on the choice of $\tilde x\in \tilde C\setminus p^{-1}(\infty)$. Such dependence will be studied in \S\ref{ssec:dependencetildeg}.

\begin{rem}\label{rem:210616}
Condition~\eqref{identity:Lambda} is equivalent to the system of equations 
\begin{equation*}
\sum_{m\geq 0}\int_{\iota_{p(\tilde x)}(f^{-1}(A_i))}\underbrace{\gamma_{\pmb\Lambda}\cdots\gamma_{\pmb\Lambda}}_m
=\rho(A_i), \quad \sum_{m\geq 0}\int_{\iota_{p(\tilde x)}(f^{-1}(B_i))}\underbrace{\gamma_{\pmb\Lambda}\cdots\gamma_{\pmb\Lambda}}_m
=\rho(B_i), 
\end{equation*}
where $\iota_{p(\tilde x)}(f^{-1}(A_i)),\iota_{p(\tilde x)}(f^{-1}(B_i))$ are the classes of the loops on $C\setminus\infty$ based at 
$p(\tilde x)$ corresponding to $A_i,B_i$, and the iterated integrals are defined as in eq.~\eqref{itintovercycle}. 
\end{rem}

\subsection{The particular case of $\rho_0$}\label{ssec:26}

Let $\mathrm{Lie}(b_1,\ldots,b_h)$ be the Lie subalgebra of $\mathfrak G$ generated by $b_1,\ldots,b_h$
and let $\mathrm{Lie}(b_1,\ldots,b_h)^\wedge$ be its degree completion. We denote by $\mathrm{exp}(\mathrm{Lie}(b_1,\ldots,b_h)^\wedge)$ the corresponding subgroup of $\mathrm{exp}(\hat{\mathfrak G})$. 

Define $\rho_0:\langle A_i,B_i\rangle /(\Pi_i(A_i,B_i))\to \mathrm{exp}(\mathrm{Lie}(b_1,\ldots,b_h)^\wedge)$ to be the morphism induced by 
$A_i\mapsto 1$, $B_i\mapsto \mathrm{exp}(b_i)$ for $i\in[\![1,h]\!]$.  

\begin{lem}\label{lem:210731n1}
For any $f\in\mathrm{Iso}_\infty(\pi_1(C\setminus\infty,p(\tilde x)),
\langle A_i,B_i\rangle)$ one has $\pmb\Lambda(\rho_0,f)\in (\mathrm{Lie}(b_1,\ldots,b_h)^\wedge)^{2h}$, and the image of $F_{\pmb{\Lambda}(\rho_0,f)}:\Pi_1(C\setminus\infty)\to\exp(\hat{\mathfrak{G}})$ is contained in $\mathrm{exp}(\mathrm{Lie}(b_1,\ldots,b_h)^\wedge)$.
\end{lem}

\begin{proof}
Let us first show that $\pmb\Lambda(\rho_0,f)\in (\mathrm{Lie}(b_1,\ldots,b_h)^\wedge)^{2h}$. By the proof of 
Lemma~\ref{lem:29121529}, one has  $\pmb\Lambda(\rho_0,f)=\mathrm{hol}_f^{-1}(\mathrm{log}\rho_0(A_1),\ldots,
\mathrm{log}\rho_0(B_{h}))=(\mathrm{ev}_{\pmb\alpha}(c_1+m_1^f),\ldots,\mathrm{ev}_{\pmb\alpha}(c_{2h}+m_{2h}^f))$, 
where $\pmb\alpha=(\mathrm{log}\rho_0(A_1),\ldots,\mathrm{log}\rho_0(B_{h}))=(0,\ldots,0,b_1,\ldots,b_h)$, and 
$\mathrm{ev}_{\pmb\alpha}$ is the endomorphism of $\hat{\mathfrak G}$ such that
$a_i\mapsto 0$, $b_i\mapsto b_i$ for $i\in[\![1,h]\!]$. It follows that the image of~$\mathrm{ev}_{\pmb\alpha}$
is contained in $\mathrm{Lie}(b_1,\ldots,b_h)^\wedge$, and therefore that 
$\pmb\Lambda(\rho_0,f)\in (\mathrm{Lie}(b_1,\ldots,b_h)^\wedge)^{2h}$.

Let us now show that the image of the groupoid morphism $F_{\pmb\Lambda(\rho_0,f)}:\Pi_1(C\setminus\infty)\to
\mathrm{exp}(\hat{\mathfrak G})$ is contained in $\mathrm{exp}(\mathrm{Lie}(b_1,\ldots,b_h)^\wedge)$. 
By~\eqref{expansion:0202}, for any $(x,y,\eta)\in\Pi_1(C\setminus\infty)$ one has $F_{\pmb\Lambda(\rho_0,f)}(x,y,\eta)
=1+\int_\eta\gamma_{\pmb\Lambda(\rho_0,f)}+\int_\eta\gamma_{\pmb\Lambda(\rho_0,f)}\gamma_{\pmb\Lambda(\rho_0,f)}
+\ldots$, which together with $\gamma_{\pmb\Lambda(\rho_0,f)}=\sum_{i=1}^{2h}\gamma_i\Lambda_i(\rho_0,f)$
implies $F_{\pmb\Lambda(\rho_0,f)}(x,y,\eta)\in U(\mathrm{Lie}(b_1,\ldots,b_h))^\wedge$. Since 
$F_{\pmb\Lambda(\rho_0,f)}(x,y,\eta)\in\mathcal G(U(\mathfrak G)^\wedge)$, it follows that  
$F_{\pmb\Lambda(\rho_0,f)}(x,y,\eta)\in\mathcal G(U(\mathrm{Lie}(b_1,\ldots,b_h))^\wedge)
=\mathrm{exp}(\mathrm{Lie}(b_1,\ldots,b_h)^\wedge)$. 
\end{proof}

\begin{prop}\label{prop:1:1732}
If $f,f'\in\mathrm{Iso}_\infty(\pi_1(C\setminus\infty,p(\tilde x)),
\langle A_i,B_i\rangle)$ are related by the conjugation action of $(\prod_i(A_i,B_i))$, then
$\pmb\Lambda(\rho_0,f)=\pmb\Lambda(\rho_0,f')$.  
\end{prop}

\proof By assumption, there exists $g\in(\prod_{i=1}^h(A_i,B_i))$ such that $f'=\mathrm{Ad}_g\circ f$.  

One has $\rho_0(\prod_{i=1}^h(A_i,B_i))=\prod_{i=1}^h(\rho_0(A_i),\rho_0(B_i))=\prod_{i=1}^h(1,e^{b_i})=1$. 
It follows that  
\begin{equation}\label{stat:kernel}
\bigg(\prod_{i=1}^h(A_i,B_i)\bigg)\subset\mathrm{ker}(\rho_0). 
\end{equation}
In particular, 
\begin{equation}\label{van:rho0:g}
\rho_0(g)=1.
\end{equation} 

Then 
$$
F_{\pmb\Lambda(\rho_0,f')}\circ\iota_{p(\tilde x)}=\rho_0\circ f'
=\rho_0\circ\mathrm{Ad}_g\circ f=\mathrm{Ad}_{\rho_0(g)}\circ\rho_0\circ f=\rho_0\circ f,  
$$
where the first identity follows from~\eqref{identity:Lambda}, the second from $f'=\mathrm{Ad}_g\circ f$, 
and the last identity from~\eqref{van:rho0:g}. It then follows from the resulting equality 
$F_{\pmb\Lambda(\rho_0,f')}\circ\iota_{p(\tilde x)}=\rho_0\circ f$ and from the uniqueness
proved in Lemma~\ref{lem:29121529} of solutions of~\eqref{identity:Lambda}, that 
$\pmb\Lambda(\rho_0,f')=\pmb\Lambda(\rho_0,f)$. 
\hfill\qed\medskip 

\begin{defn}
Define $\Pi_1(C)\big|_{C\setminus\infty}$ to be the groupoid obtained from $\Pi_1(C)$
by restricting the set of objects to $C\setminus\infty$. The 1-nerve of this groupoid is the fibered product 
$\pmb\Pi_1(C)\times_{C^2}(C\setminus\infty)^2$.  
\end{defn}

The inclusion $C\setminus\infty\subset C$ induces a groupoid morphism $\Pi_1(C\setminus\infty)\to\Pi_1(C)$, 
which admits a factorization $\Pi_1(C\setminus\infty)\to\Pi_1(C)\big|_{C\setminus\infty}\to\Pi_1(C)$. 

\begin{prop}\label{prop:2:1741}
If $f\in\mathrm{Iso}_\infty(\pi_1(C\setminus\infty,p(\tilde x)),\langle A_i,B_i\rangle)$, then there is a groupoid morphism 
$G_f:\Pi_1(C)\big|_{C\setminus\infty}\to \mathrm{exp}(\mathrm{Lie}(b_1,\ldots,b_h)^\wedge)$, uniquely determined by the 
condition that $F_{\pmb\Lambda(\rho_0,f)}$ coincides with the composition 
$\Pi_1(C\setminus\infty)\to\Pi_1(C)\big|_{C\setminus\infty}\stackrel{G_f}{\longrightarrow} 
\mathrm{exp}(\mathrm{Lie}(b_1,\ldots,b_h)^\wedge)$. 
\end{prop}

\begin{proof} Let $H$ be the kernel of the group morphism $\pi_1(C\setminus\infty,p(\tilde x))\to\pi_1(C,p(\tilde x))$. This is the normal subgroup of $\pi_1(C\setminus\infty,p(\tilde x))=\Pi_1(C\setminus\infty)(p(\tilde x))$ generated by a loop around~$\infty$. The groupoid quotient $\Pi_1(C\setminus\infty)/H$, defined as in \S\ref{sec:groupoidquotient}, coincides with the restriction $\Pi_1(C)\big|_{C\setminus\infty}$
of the fundamental groupoid~$\Pi_1(C)$ to the subset of objects $C\setminus\infty\subset C$. Let us also consider the trivial groupoid 
corresponding to the group~$\mathrm{exp}(\mathrm{Lie}(b_1,\ldots,b_h)^\wedge)$, whose set of objects is~$\{*\}$, and 
$\mathrm{exp}(\mathrm{Lie}(b_1,\ldots,b_h)^\wedge)(*)=\mathrm{exp}(\mathrm{Lie}(b_1,\ldots,b_h)^\wedge)$. 

By Lemmas~\ref{lem:2912a} and~\ref{lem:210731n1}, $F_{\pmb\Lambda(\rho_0,f)}$ is a groupoid morphism $\Pi_1(C\setminus\infty)\to \mathrm{exp}(\mathrm{Lie}(b_1,\ldots,b_h)^\wedge)$. Its restriction to $\Pi_1(C\setminus\infty)(p(\tilde x))$ is simply the group morphism $F_{\pmb\Lambda(\rho_0,f)}\circ \iota_{p(\tilde x)}$, which by~\eqref{identity:Lambda} coincides with $\rho_0\circ f$, and 
$$
\rho_0\circ f(H)
=\rho_0\Bigg(\bigg(\prod_{i=1}^h(A_i,B_i)\bigg)\Bigg)=1, 
$$
where the first equality follows from $f(H)
=(\prod_{i=1}^h(A_i,B_i))$, and the second equality follows from~\eqref{stat:kernel}. 

Therefore the restriction of $F_{\pmb\Lambda(\rho_0,f)}$ to $H$ is trivial, and so by a property of groupoid quotients mentioned in \S\ref{sec:groupoidquotient} there exists
a groupoid morphism $G_f:\Pi_1(C\setminus\infty)/H\to \mathrm{exp}(\mathrm{Lie}(b_1,\ldots,b_h)^\wedge)$, such that its 
composition with $\Pi_1(C\setminus\infty)\to \Pi_1(C\setminus\infty)/H$ coincides with $F_{\pmb\Lambda(\rho_0,f)}$. 
\end{proof}

\begin{lem}\label{lemma:conj:inv}
If $f,f'\in\mathrm{Iso}_\infty(\pi_1(C\setminus\infty,p(\tilde x)),\langle A_i,B_i\rangle)$ are conjugated by the action of 
an element of $(\prod_{i=1}^h(A_i,B_i))$, then the groupoid morphisms~$G_f$ and 
$G_{f'}:\Pi_1(C)\big|_{C\setminus\infty}\to \mathrm{exp}(\mathrm{Lie}(b_1,\ldots,b_h)^\wedge)$ are equal. 
\end{lem}

\begin{proof}
By Proposition~\ref{prop:2:1741}, the groupoid morphisms~$G_f$ and~$G_{f'}$ are determined by the conditions $F_{\pmb\Lambda(\rho_0,f)}
=G_f\circ\mathrm{pr}$ and $F_{\pmb\Lambda(\rho_0,f')}
=G_{f'}\circ\mathrm{pr}$, where $\mathrm{pr}$ is the groupoid morphism 
$\Pi_1(C\setminus\infty)\to\Pi_1(C)\big|_{C\setminus\infty}$. By Proposition~\ref{prop:1:1732}, the groupoid morphisms 
$F_{\pmb\Lambda(\rho_0,f)}$ and $F_{\pmb\Lambda(\rho_0,f')}:\Pi_1(C\setminus\infty)\to \mathrm{exp}(\mathrm{Lie}(b_1,\ldots,b_h)^\wedge)$ 
coincide. It follows that the two said conditions are the same, so that~$G_f$ and~$G_{f'}$ coincide.
\end{proof}

\subsection{Construction and properties of~$\tilde g_{\tilde x,\pmb\beta}$}\label{ssec:defgtildexbeta}

Recall that an element~$e\in E$ is fixed; it associates to any $\tilde x\in \tilde C\setminus p^{-1}(\infty)$ an element $e_{\tilde x}\in E(p(\tilde x))=\mathrm{Iso}_\infty(\pi_1(C\setminus\infty,x),\langle A_i,B_i\rangle)/\big(\prod_i(A_i,B_i)\big)$, and it gives rise to an isomorphism ${\rm can}_e:{\rm Aut}(\tilde C/C)\tilde{\to} \langle A_i,B_i\rangle/(\prod_i(A_i,B_i))$ (see \S\ref{sect:topology}). 

\begin{lem}
The groupoid morphisms $G_f:\Pi_1(C)\big|_{C\setminus\infty}\to \mathrm{exp}(\mathrm{Lie}(b_1,\ldots,b_h)^\wedge)$, 
where $f\in\mathrm{Iso}_\infty(\pi_1(C\setminus\infty,p(\tilde x)),\langle A_i,B_i\rangle)$ is a lift of~$e_{\tilde x}$, 
are all equal.
\end{lem}

\begin{proof}
Any two such lifts are conjugated by the action of an element of $(\prod_{i=1}^h(A_i,B_i))$. The result then follows from 
Lemma~\ref{lemma:conj:inv}.
\end{proof}

\begin{defn}\label{def:210731}
We denote by~$G_{(e_{\tilde x})}$ the common value of the groupoid morphisms $G_f:\Pi_1(C)\big|_{C\setminus\infty}\to \mathrm{exp}(\mathrm{Lie}(b_1,\ldots,b_h)^\wedge)$, 
where $f\in\mathrm{Iso}_\infty(\pi_1(C\setminus\infty,p(\tilde x)),\langle A_i,B_i\rangle)$ is a lift of~$e_{\tilde x}$. The map between 1-nerves $\pmb\Pi_1(C) \times_{C^2}(C\setminus\infty)^2\to \mathrm{exp}(\mathrm{Lie}(b_1,\ldots,b_h)^\wedge)$ induced by the groupoid morphism~$G_{(e_{\tilde x})}$ will also be denoted~$G_{(e_{\tilde x})}$.
\end{defn}

Let $u,v,w \in C\setminus\infty$, $\alpha \in\pi_1(C;u,v)$ and $\beta \in \pi_1(C;v,w)$, then $\alpha\cdot \beta \in \pi_1(C;u,w)$, and the fact that~$G_{(e_{\tilde x})}$ is a groupoid morphism implies that, at the level of 1-nerves,
\begin{equation}\label{eq:Fcomposit}
G_{(e_{\tilde x})}(u,w,\alpha\cdot\beta)\,=\,G_{(e_{\tilde x})}(u,v,\alpha)\,G_{(e_{\tilde x})}(v,w,\beta)\,.
\end{equation} 

%Recall that $\tilde x \in \tilde C-p^{-1}(\infty)$. For $\tilde y \in \tilde C-p^{-1}(\infty)$, the set 
%$\pi_1(\tilde C;\tilde x,\tilde y)$ has only one element, because $\tilde C$ is simply connected. The image of this element by the 
%map $p_*:\pi_1(\tilde C;\tilde x,\tilde y)\to \pi_1(C;p(\tilde x),p(\tilde y))$ induced by~$p$ is therefore a well-defined element of  
%$\pi_1(C;p(\tilde x),p(\tilde y))$, which we denote by $\eta_{\tilde x,\tilde y}$. 

Define a map 
$$
\textrm{map}_{\tilde x} : \tilde C\setminus p^{-1}(\infty)\to \pmb\Pi_1(C) \times_{C^2}(C\setminus\infty)^2
$$ 
by setting $\textrm{map}_{\tilde x}(\tilde y):=(p(\tilde x),p(\tilde y),\eta_{\tilde x,\tilde y})$, where $\eta_{\tilde x,\tilde y}$ is as in Definition~\ref{def:eta:xy:0202}.

\begin{defn}\label{def:g:30032021}
We denote\footnote{$G_{(e_{\tilde x})}$ depends on $e$ and $\tilde x$ but also, implicitly, on the tuple of differential forms~$\pmb \beta$ through eq.~\eqref{eq:diffeqFLambda}. The system of compatible classes of parametrisation $e$ is considered as part of the topological assumptions for our construction, and is therefore omitted from now on.} by~$\tilde g_{\tilde x,\pmb \beta}$ the map $\tilde C\setminus p^{-1}(\infty)\to \mathrm{exp}(\mathrm{Lie}(b_1,\ldots,b_h)^\wedge)$ given by 
$$
\tilde y\mapsto \big(G_{(e_{\tilde x})} \circ \textrm{map}_{\tilde x}(\tilde y)\big)^{-1}.
$$ 
\end{defn}

\begin{prop}\label{prop:tildeg}
The map $\tilde g_{\tilde x,\pmb \beta}$ has the property that, for any $\theta\in\mathrm{Aut}(\tilde C/C)$,
$$
\theta^*\tilde g_{\tilde x,\pmb \beta}=\tilde g_{\tilde x,\pmb \beta}\cdot(\rho_0\circ {\rm can}_e(\theta))^{-1}\,.
$$
\end{prop}

\begin{proof}
Let us denote $\phi:=G_{(e_{\tilde x})} \circ \textrm{map}_{\tilde x}$, then the statement is equivalent to saying that $\theta^*\phi=\rho_0 \circ \textrm{can}_e(\theta)\cdot \phi$ for any $\theta \in \textrm{Aut}(\tilde C/C)$,  
i.e. that for any $\tilde y \in \tilde C\setminus p^{-1}(\infty)$, $\phi(\theta(\tilde y))=\rho_0 \circ \textrm{can}_e(\theta)\cdot \phi(\tilde y)$, which we now prove.

For any $\tilde y\in \tilde C\setminus p^{-1}(\infty)$, we have $\textrm{map}_{\tilde x}(\theta(\tilde y))=(p(\tilde x),p(\tilde y),\eta_{\tilde x,\theta(\tilde y)})$. 
Since
$$
\eta_{\tilde x,\theta(\tilde y)}=\eta_{\tilde x,\theta(\tilde x)}\cdot\eta_{\theta(\tilde x),\theta(\tilde y)}
=\eta_{\tilde x,\theta(\tilde x)}\cdot\eta_{\tilde x,\tilde y},
$$ 
we can use eq.~\eqref{eq:Fcomposit} and find
%because a path from~$\tilde x$ to~$\theta(\tilde y)$ can be written as the composition of a path from~$\tilde x$ to~$\theta(\tilde x)$ and a path from~$\theta(\tilde x)$ to~$\theta(\tilde y)$, and $\eta_{\theta(\tilde x),\theta(\tilde y)}=\eta_{\tilde x,\tilde y}$ for any $\tilde x,\tilde y\in \tilde C$ and any $\theta\in \textrm{Aut}(\tilde C/C)$. It then follows from~
$$
\phi(\theta(\tilde y))=G_{(e_{\tilde x})}(p(\tilde x),p(\tilde x),\eta_{\tilde x,\theta(\tilde x)})\,\phi(\tilde y),
$$ 
so that we are left with proving that 
\begin{equation}\label{eq:210731n1}
G_{(e_{\tilde x})}(p(\tilde x),p(\tilde x),\eta_{\tilde x,\theta(\tilde x)})=\rho_0 \circ \textrm{can}_e(\theta)\,.
\end{equation}

The choice of~$\tilde x$ naturally identifies $\textrm{Aut}(\tilde C/C)$ with $\pi_1(C,p(\tilde x))$ (see Lemma~\ref{lemma:standard}). 
Under this identification, we can write $(p(\tilde x),p(\tilde x),\eta_{\tilde x,\theta(\tilde x)})=\iota_{p(\tilde x)}(\theta)$
(see the remark following Definition~\ref{def:eta:xy:0202}). The result follows from Lemma~\ref{lem:29121529}.
\end{proof}

\subsection{Dependence of $\tilde g_{\tilde x,\pmb \beta}$ in $\tilde x$ and $\pmb\beta$}\label{ssec:dependencetildeg}

We first study the dependence on~$\tilde x$ for fixed $\pmb \beta$. We recall that, for any~$\rho$ and any $f\in\mathrm{Iso}_\infty
(\pi_1(C\setminus\infty,p(\tilde x)),\langle A_i,B_i\rangle)$, the elements $\Lambda(\rho,f)\in 
\hat{\mathfrak G}^{2h}$ defined by Lemma~\ref{lem:29121529} implicitly depend also on $\tilde x$ (and on $\pmb \beta$, but for the moment this is fixed). In particular, if~$\rho$ is fixed to be the morphism~$\rho_0$ from \S\ref{ssec:26}, and if $f$ is a representative of the class $e_{\tilde x}\in E(\tilde x)$, with~$e$ fixed, then by Proposition~\ref{prop:1:1732} the associated element ${\pmb \Lambda}(\rho_0,f)$ only depends on~$\tilde x$, and so in this section we will denote it by~${\pmb \Lambda}(\tilde x)$.
\begin{lem}
For any $\tilde y\in \tilde C\setminus p^{-1}(\infty)$ one has
\begin{equation*}
{\pmb \Lambda}(\tilde y)\,=\,{\rm Ad}_{F_{{\pmb \Lambda}(\tilde x)}(p(\tilde x),p(\tilde y),\eta_{\tilde x,\tilde y})}\big({\pmb \Lambda}(\tilde x)\big)\,.
\end{equation*}
\end{lem}
\begin{proof}
For any $\tilde x\in \tilde C\setminus p^{-1}(\infty)$ and any $i\in [\![1,2h]\!]$ let $\nu_{\tilde x}^i:\hat{\mathfrak G}^{2h}\to \exp(\hat{\mathfrak G})$ be the morphism given by
\begin{equation*}
\pmb \Lambda \,\to\, \sum_{m\geq 0}\int_{\tilde x}^{{\rm can}_e^{-1}(C_i)(\tilde x)}\underbrace{\gamma_{\pmb \Lambda}\cdots \gamma_{\pmb\Lambda}}_m\,,
\end{equation*}
where $C_i:=A_i$ for $i\in [\![1,h]\!]$ and $C_i:=B_{i-h}$ for $i\in [\![h+1,2h]\!]$.
Splitting a path from~$\tilde x$ to ${\rm can}_e^{-1}(C_i)(\tilde x)$ into the composition of a path from~$\tilde x$ to~$\tilde y$, a path from~$\tilde y$ to ${\rm can}_e^{-1}(C_i)(\tilde y)$ and a path from ${\rm can}_e^{-1}(C_i)(\tilde y)$ to ${\rm can}_e^{-1}(C_i)(\tilde x)$, and using the composition-of-paths property~\eqref{eq:comppaths} of iterated integrals, one finds that, for any ${\pmb \Lambda} \in \hat{\mathfrak G}^{2h}$,
\begin{align}\label{eq:211013n1}
\nu_{\tilde x}^i(\pmb\Lambda)&\,=\,F_{\pmb \Lambda}(p(\tilde x),p(\tilde y),\eta_{\tilde x,\tilde y})\,\nu_{\tilde y}^i(\pmb\Lambda)F_{\pmb \Lambda}(p(\tilde x),p(\tilde y),\eta_{\tilde x,\tilde y})^{-1}\\
&\,=\,{\rm Ad}_{F_{{\pmb \Lambda}(\tilde x)}(p(\tilde x),p(\tilde y),\eta_{\tilde x,\tilde y})}\big(\nu_{\tilde y}^i(\pmb\Lambda)\big)\,.\notag
\end{align}
By Remark~\ref{rem:210616}, one has $\rho_0(C_i)=\nu_{\tilde x}^i(\pmb\Lambda(\tilde x))$. Since this value does not depends on $\tilde x$, we find that $\nu_{\tilde x}^i(\pmb\Lambda(\tilde x))=\nu_{\tilde y}^i(\pmb\Lambda(\tilde y))$, and combining this identity with eq.~\eqref{eq:211013n1} we conclude that, for any $i\in [\![1,2h]\!]$,
\begin{equation}
\nu_{\tilde y}^i(\pmb\Lambda(\tilde y))\,=\,{\rm Ad}_{F_{{\pmb \Lambda}(\tilde x)}(p(\tilde x),p(\tilde y),\eta_{\tilde x,\tilde y})}\big(\nu_{\tilde y}^i(\pmb\Lambda(\tilde x))\big)\,.
\end{equation}
The statement of this lemma follows from the fact that $(\nu_{\tilde y}^1,\cdots ,\nu_{\tilde y}^{2h}):\hat{\mathfrak{G}}^{2h}\to \exp(\hat{\mathfrak{G}})^{2h}$ is a bijection, which can be inferred from the proof of Lemma~\ref{lem:29121529}.
\end{proof}
An immediate consequence of this lemma is the following result.
\begin{cor}\label{cor:211013}
For any $\tilde y\in \tilde C\setminus p^{-1}(\infty)$ one has
\begin{equation*}
\gamma_{{\pmb \Lambda}(\tilde y)}\,=\,{\rm Ad}_{F_{{\pmb \Lambda}(\tilde x)}(p(\tilde x),p(\tilde y),\eta_{\tilde x,\tilde y})}\big(\gamma_{{\pmb \Lambda}(\tilde x)}\big)\,,
\end{equation*}
where we recall that $\gamma_{\pmb \Lambda}=\sum_{i=1}^{2h}\Lambda_i\gamma_i$.
\end{cor}

\begin{prop}\label{prop:211013}
For any $\tilde y\in \tilde C\setminus p^{-1}(\infty)$ and any $\pmb \beta$ one has the identity
\begin{equation*}
\tilde g_{\tilde x,\pmb \beta}(\bullet)\,=\,\tilde g_{\tilde x,\pmb \beta}(\tilde y)\,\tilde g_{\tilde y,\pmb \beta}(\bullet)\,.
\end{equation*}
\end{prop}
\begin{proof}
Since for any $\tilde u,\tilde v\in  \tilde C\setminus p^{-1}(\infty)$ one has $\tilde g_{\tilde u,\pmb \beta}(\tilde v)=F_{\pmb\Lambda(\tilde u)}(p(\tilde u),p(\tilde v),\eta_{\tilde u,\tilde v})^{-1}$, then the statement is equivalent to the identity
\begin{equation}\label{eq:211013n2}
F_{\pmb\Lambda(\tilde y)}(p(\tilde y),p(\tilde z),\eta_{\tilde y,\tilde z})\,=\,F_{\pmb\Lambda(\tilde x)}(p(\tilde x),p(\tilde z),\eta_{\tilde x,\tilde z})\,F_{\pmb\Lambda(\tilde x)}(p(\tilde x),p(\tilde y),\eta_{\tilde x,\tilde y})^{-1}\,
\end{equation}
for any $\tilde x, \tilde y, \tilde z\in  \tilde C\setminus p^{-1}(\infty)$. 

On the one hand, by the definition of~$F_{\Lambda}$ we have
\begin{equation*}
d_z\,F_{\pmb\Lambda(\tilde y)}(p(\tilde y),p(\tilde z),\eta_{\tilde y,\tilde z})\,=\,F_{\pmb\Lambda(\tilde y)}(p(\tilde y),p(\tilde z),\eta_{\tilde y,\tilde z})\,\gamma_{\pmb\Lambda(\tilde y)}\,.
\end{equation*}
One the other hand, we have
\begin{align*}
&d_z\bigg({\rm Ad}_{F_{{\pmb \Lambda}(\tilde x)}(p(\tilde x),p(\tilde y),\eta_{\tilde x,\tilde y})}\big(F_{\pmb\Lambda(\tilde x)}(p(\tilde x),p(\tilde z),\eta_{\tilde x,\tilde z})\big)\bigg)\\
&=\,{\rm Ad}_{F_{{\pmb \Lambda}(\tilde x)}(p(\tilde x),p(\tilde y),\eta_{\tilde x,\tilde y})}\big(d_zF_{\pmb\Lambda(\tilde x)}(p(\tilde x),p(\tilde z),\eta_{\tilde x,\tilde z})\big)\\
&=\,{\rm Ad}_{F_{{\pmb \Lambda}(\tilde x)}(p(\tilde x),p(\tilde y),\eta_{\tilde x,\tilde y})}\big(F_{\pmb\Lambda(\tilde x)}(p(\tilde x),p(\tilde z),\eta_{\tilde x,\tilde z})\,\gamma_{\pmb\Lambda(\tilde x)}\big)\\
&=\,{\rm Ad}_{F_{{\pmb \Lambda}(\tilde x)}(p(\tilde x),p(\tilde y),\eta_{\tilde x,\tilde y})}\big(F_{\pmb\Lambda(\tilde x)}(p(\tilde x),p(\tilde z),\eta_{\tilde x,\tilde z})\big)\,{\rm Ad}_{F_{{\pmb \Lambda}(\tilde x)}(p(\tilde x),p(\tilde y),\eta_{\tilde x,\tilde y})}\big(\gamma_{\pmb\Lambda(\tilde x)}\big)\\
&=\,{\rm Ad}_{F_{{\pmb \Lambda}(\tilde x)}(p(\tilde x),p(\tilde y),\eta_{\tilde x,\tilde y})}\big(F_{\pmb\Lambda(\tilde x)}(p(\tilde x),p(\tilde z),\eta_{\tilde x,\tilde z})\big)\,\gamma_{\pmb\Lambda(\tilde y)}\,,
\end{align*}
where the second equality follows by the definition of~$F_{\Lambda}$, and the fourth by Corollary~\ref{cor:211013}.

Therefore both $F_{\pmb\Lambda(\tilde y)}(p(\tilde y),p(\tilde z),\eta_{\tilde y,\tilde z})$ and ${\rm Ad}_{F_{{\pmb \Lambda}(\tilde x)}(p(\tilde x),p(\tilde y),\eta_{\tilde x,\tilde y})}\big(F_{\pmb\Lambda(\tilde x)}(p(\tilde x),p(\tilde z),\eta_{\tilde x,\tilde z})\big)$ satisfy the differential equation $d_zF=F\cdot\gamma_{\pmb\Lambda(\tilde y)}$, and therefore there exists an element $C\in \exp(\hat{\mathfrak G})$ which does not depend on $\tilde z$ such that
\begin{equation}\label{eq:211013n3}
F_{\pmb\Lambda(\tilde y)}(p(\tilde y),p(\tilde z),\eta_{\tilde y,\tilde z})\,=\,C\,{\rm Ad}_{F_{{\pmb \Lambda}(\tilde x)}(p(\tilde x),p(\tilde y),\eta_{\tilde x,\tilde y})}\big(F_{\pmb\Lambda(\tilde x)}(p(\tilde x),p(\tilde z),\eta_{\tilde x,\tilde z})\big)\,.
\end{equation}
Specialising this identity at $\tilde z=\tilde y$ implies that $C=F_{{\pmb \Lambda}(\tilde x)}(p(\tilde x),p(\tilde y),\eta_{\tilde x,\tilde y})^{-1}$, and substituting this value into~\eqref{eq:211013n3} proves the identity~\eqref{eq:211013n2}.
\end{proof}
In particular, this shows that $\tilde g_{\tilde x,\pmb \beta}\tilde g_{\tilde y,\pmb \beta}^{-1}$ is  constant. If now we want to vary both $\tilde x$ and $\pmb \beta$, we obtain the following:
\begin{prop}\label{prop:211013n2}
For any $\tilde x,\tilde y\in \tilde C\setminus p^{-1}(\infty)$, any two families ${\pmb \beta}, {\pmb \beta}'$ as in \S\ref{def:omega}, the function $\tilde g_{\tilde x,\pmb \beta}\tilde g_{\tilde y,\pmb \beta'}^{-1}$ on $\tilde C\setminus p^{-1}(\infty)$ is the pull-back of a function defined on $C\setminus \infty$.
\end{prop}
\begin{proof}
By Proposition~\ref{prop:tildeg}, for any $\theta\in {\rm Aut}(\tilde C/C)$ one has $\theta^*\tilde g_{\tilde x,\pmb \beta}=\tilde g_{\tilde x,\pmb \beta}\,(\rho_0\circ {\rm can}_e(\theta))^{-1}$ for any $\tilde x$ and any $\pmb \beta$. This implies that 
$$
\theta^*\big(\tilde g_{\tilde x,\pmb \beta}\,\tilde g_{\tilde y,\pmb \beta'}^{-1}\big)\,=\,\tilde g_{\tilde x,\pmb \beta}\,(\rho_0\circ {\rm can}_e(\theta))^{-1}\,(\rho_0\circ {\rm can}_e(\theta))\,g_{\tilde y,\pmb \beta'}^{-1}\,=\,\tilde g_{\tilde x,\pmb \beta}\,\tilde g_{\tilde y,\pmb \beta'}^{-1}\,.
$$
\end{proof}

\subsection{Trivialising the restriction of $\mathcal{P}$ to $(C\setminus\infty)^n$}\label{ssec:trivialisation}

\subsubsection{The Lie algebra $\mathfrak t_{h,n}$}

Let~$n$ be an integer~$\geq1$ and~$\mathfrak t_{h,n}$ be the Lie algebra generated by elements
$a_i^{(r)},b_i^{(r)}$, with $i\in[\![1,h]\!]$, $r\in[\![1,n]\!]$, and by elements~$t_{rs}$, with $r,s\in[\![1,n]\!]$, $r\neq s$, subject to the following relations:
\begin{itemize}
\item[(i)] For any $i\in[\![1,h]\!]$ and for pairwise distinct $r,l,m\in[\![1,n]\!]$, 
\begin{equation*}
[a_i^{(r)},t_{lm}]\,=\,[b_i^{(r)},t_{lm}]\,=\,0\,.
\end{equation*}
\item[(ii)] For any $r\in[\![1,n]\!]$,
\begin{equation}\label{reltgnii}
\sum_{i\in[\![1,h]\!]}\,[b^{(r)}_i,a^{(r)}_i]\,+\,\sum_{s\in[\![1,n]\!]\setminus\{r\}}t_{rs}\,=0.
\end{equation}
\item[(iii)] For any $i,j\in[\![1,h]\!]$ and any $r, s\in[\![1,n]\!]$, $r\neq s$,
\begin{align*}
&[b^{(r)}_i,a^{(s)}_j]\,=\,  -[a^{(r)}_i,b^{(s)}_j] \,=\,   \delta_{ij}\,t_{rs}\,,\\
%&[a^{(r)}_i,b^{(s)}_j]\,=\,-\delta_{ij}\,t_{rs}\,,\\
&[a_i^{(r)},a_j^{(s)}]\,=\,[b_i^{(r)},b_j^{(s)}]\,=\,0\,.
\end{align*}
\end{itemize}

These relations imply $t_{rs}=t_{sr}$ for any $r\neq s$ (see \cite{EnrConfSp}, Lemma~18~(1)). 
The Lie algebra~$\mathfrak t_{h,n}$ is equipped with a $\mathbb Z_{\geq0}$-grading by 
$\mathrm{deg}(a_i^{(r)})=\mathrm{deg}(b_i^{(r)})=1$, $\mathrm{deg}(t_{rs})=2$; for $d\geq0$, 
we denote by $\mathfrak t_{h,n}[d]$ its degree-$d$ component. We define $\hat{\mathfrak t}_{h,n}$ 
to be the degree completion of $\mathfrak t_{h,n}$. The degree completion $U(\mathfrak t_{h,n})^\wedge$ 
of the enveloping algebra of $\mathfrak t_{h,n}$ is a topological Hopf algebra; the exponential map 
$\mathrm{exp}:U(\mathfrak t_{h,n})^\wedge_+\to U(\mathfrak t_{h,n})^\wedge$ (where the source is the 
product of positive degree components of $U(\mathfrak t_{h,n})^\wedge$) sets up a bijection between 
$\hat{\mathfrak t}_{h,n}$ and its group $\mathcal G(U(\mathfrak t_{h,n})^\wedge)$
of group-like elements, which will henceforth be denoted $\mathrm{exp}(\hat{\mathfrak t}_{h,n})$.

\subsubsection{Construction of $\mathcal{P}$}\label{sssec:consP}

\begin{lem}\label{lemconf}
For any $r\in[\![1,n]\!]$, there is a unique morphism of graded Lie algebras $(r):\mathfrak{G}\to\mathfrak t_{h,n}$, 
denoted $X\mapsto X^{(r)}$, such that $a_i\mapsto a_i^{(r)}$ and $b_i\mapsto b_i^{(r)}$ for any $i\in[\![1,h]\!]$. It induces a morphism of topological Hopf algebras $(r):U(\mathfrak{G})^\wedge\to U(\mathfrak t_{h,n})^\wedge$ and a morphism of groups $(r):\mathrm{exp}(\hat{\mathfrak{G}})\to\mathrm{exp}(\hat{\mathfrak t}_{h,n})$. 
\end{lem}

\begin{proof}
This follows from the fact that~$\mathfrak G$ is freely generated by the elements $a_i,b_i$.
\end{proof} 

Let $\mu:\mathrm{Aut}(\tilde C/C)\to \mathrm{exp}(\hat{\mathfrak G})$ be a group morphism, and let 
$\mathrm{exp}(\hat{\mathfrak t}_{h,n})_\mu$ be the principal $\mathrm{exp}(\hat{\mathfrak t}_{h,n})$-bundle over~$C^n$ attached to~$\mu$, obtained as follows. For any $\theta\in\mathrm{Aut}(\tilde C/C)$, we set~$\theta^{(r)}$ to be the automorphism of~$\tilde C^n$ acting as~$\theta$ on the $r$-th component and as the identity on all the other ones.  For $U\subset C^n$ an open subset, the set of sections $\Gamma(U,\mathrm{exp}(\hat{\mathfrak t}_{h,n})_\mu)$ is defined to be the set of holomorphic maps 
$g: (p^n)^{-1}(U)\to \mathrm{exp}(\hat{\mathfrak t}_{h,n})$ 
such that\footnote{Here one should not confuse the superscript notation $\bullet^{(r)}$ for the automorphism $\theta$ with that, introduced in Lemma~\ref{lemconf}, for $\mu(\theta)\in\mathrm{exp}(\hat{\mathfrak G})$.} $(\theta^{(r)})^*g=g \cdot (\mu(\theta)^{(r)})^{-1}$ for any $\theta\in\mathrm{Aut}(\tilde C/C)$ and any $r\in[\![1,n]\!]$.
\begin{defn}
The principal $\mathrm{exp}(\hat{\mathfrak t}_{h,n})$-bundle~$\mathcal{P}$ introduced in~\cite{EnrConfSp} is defined by setting $\mathcal{P}:=\mathrm{exp}(\hat{\mathfrak t}_{h,n})_{\rho_0\circ {\rm can}_e}$.
\end{defn}

\subsubsection{Trivialisation of $\mathcal{P}$}\label{sssec:trivP}

From now on, for any function or 1-form~$\phi$ taking values in $U(\mathrm{Lie}(b_1,\ldots,b_h))^\wedge$, and for any $r\in[\![1,n]\!]$, we will denote by~$\phi^{(r)}$ the composition $(r)\circ \phi$, which takes value in $U(\mathfrak t_{h,n})^\wedge$. Moreover, if~$\phi$ is defined on a Riemann surface $\Sigma$ and if ${\rm pr}_r:\Sigma^n\to \Sigma$ is the projection on the $r$-th component, then we set 
\begin{equation}\label{defgraffe}
\phi^{\{r\}}\,:=\,{\rm pr}_r^*\,\phi^{(r)}\,.
\end{equation}
Notice that, if~$\phi$ takes values in $\exp(\mathrm{Lie}(b_1,\ldots,b_h)^\wedge)$, then~$\phi^{(r)}$ and~$\phi^{\{r\}}$ take values in $\mathrm{exp}(\hat{\mathfrak t}_{h,n})$. In particular, since $\tilde{g}_{\tilde x,\pmb\beta}$ takes values in $\exp(\mathrm{Lie}(b_1,\ldots,b_h)^\wedge)$, one can consider the $\mathrm{exp}(\hat{\mathfrak t}_{h,n})$-valued function $\tilde{g}^{\{r\}}_{\tilde x,\pmb\beta}$.

\begin{defn}
We define $\tilde{\pmb g}_{\tilde x,\pmb\beta} : (\tilde C\setminus p^{-1}(\infty))^n\to\mathrm{exp}(\hat{\mathfrak t}_{h,n})$
by setting
\begin{equation*}
\tilde{\pmb g}_{\tilde x,\pmb\beta}\,:=\,
\tilde{g}^{\{1\}}_{\tilde x,\pmb\beta}\cdots \tilde{g}^{\{n\}}_{\tilde x,\pmb\beta}. 
\end{equation*}
\end{defn}

\begin{lem}\label{lemme:transporté}
For any $\theta\in {\rm Aut}(\tilde C/C)$ and any $r\in[\![1,n]\!]$, one has 
\begin{equation*}
(\theta^{(r)})^*\,\tilde{\pmb g}_{\tilde x,\pmb\beta}
\,=\,\tilde{\pmb g}_{\tilde x,\pmb\beta}\cdot (\rho_0\circ {\rm can}_e(\theta)^{(r)})^{-1}\,. 
\end{equation*}
\end{lem}

\begin{proof}
One has 
\begin{align*}
(\theta^{(r)})^*\,\tilde{\pmb g}_{\tilde x,\pmb\beta}
&\,=\,\tilde{g}^{\{1\}}_{\tilde x,\pmb\beta}
\cdots
{\rm pr}_r^*\,((r)\,\circ\,\theta^*\,\tilde{g}_{\tilde x,\pmb\beta})
\cdots \tilde{g}^{\{n\}}_{\tilde x,\pmb\beta}
\\ & 
\,=\,\tilde{g}^{\{1\}}_{\tilde x,\pmb\beta}
\cdots
{\rm pr}_r^*\,((r)\,\circ \,\tilde{g}_{\tilde x,\pmb\beta}\cdot (\rho_0\circ {\rm can}_e(\theta))^{-1})
\cdots \tilde{g}^{\{n\}}_{\tilde x,\pmb\beta}
\\ & \,=\,\tilde{g}^{\{1\}}_{\tilde x,\pmb\beta}
\cdots \tilde{g}^{\{r\}}_{\tilde x,\pmb\beta}\cdot (\rho_0\circ {\rm can}_e(\theta)^{(r)})^{-1}
\cdots \tilde{g}^{\{n\}}_{\tilde x,\pmb\beta}
\\ & 
\,=\,\tilde{g}^{\{1\}}_{\tilde x,\pmb\beta}
\cdots \tilde{g}^{\{r\}}_{\tilde x,\pmb\beta}
\cdots \tilde{g}^{\{n\}}_{\tilde x,\pmb\beta}\cdot(\rho_0\circ {\rm can}_e(\theta)^{(r)})^{-1}
\,=\,\tilde{\pmb g}_{\tilde x,\pmb\beta}
\cdot (\rho_0\circ {\rm can}_e(\theta)^{(r)})^{-1}
\end{align*}
where the second equality follows from Proposition~\ref{prop:tildeg}, and the fourth equality follows from the fact that $\rho_0$ takes values in $\exp(\mathrm{Lie}(b_1,\ldots,b_h)^\wedge)$ and that~$b_i^{(r)}$ commutes with the image of $\exp(\mathrm{Lie}(b_1,\ldots,b_h)^\wedge)$ by~$(s)$ for any $s\neq r$.
\end{proof}

\begin{thm}\label{thm:2:28}
The restriction of the principal $\mathrm{exp}(\hat{\mathfrak t}_{h,n})$-bundle~$\mathcal{P}$ to $(C\setminus\infty)^n$ is trivial. 
\end{thm}

\begin{proof}
Let~$\mathcal{P}_{\rm trv}$ be the trivial principal bundle over $(C\setminus\infty)^n$ with group $\mathrm{exp}(\hat{\mathfrak t}_{h,n})$. Then the set of isomorphisms of principal bundles over $(C\setminus\infty)^n$ from $\mathcal{P}_{\rm trv}$ to $\mathcal{P}$ is $\Gamma((\tilde C\setminus p^{-1}(\infty))^n,\mathcal{P})=\{g:(\tilde C\setminus p^{-1}(\infty))^n\to \mathrm{exp}(\hat{\mathfrak t}_{h,n}) \,|\, \forall r \in [\![1,n]\!], \forall \theta \in {\rm Aut}(\tilde C/C),(\theta^{(r)})^*g=g\cdot (\rho_0\circ {\rm can}_e(\theta)^{(r)})^{-1}\}$. By Lemma~\ref{lemme:transporté}, $\tilde{\pmb g}_{\tilde x,\pmb \beta}$ belongs to this set.
\end{proof}

\section{Construction of the Maurer-Cartan elements ${\pmb J}_{\tilde x,\pmb \beta}$}\label{sec3}

In this section we construct a Maurer-Cartan element ${\pmb J}_{\tilde x,\pmb \beta}$ on $C_n(C\setminus\infty)$, making use of the function $\tilde g_{\tilde x,\pmb \beta}$ from the previous section.

In the background sections \S\ref{Schottkyformalism} and \S\ref{sec:backgroundres} we introduce the Schottky cover $\pi:\hat C\to C$, we recall some facts about Poincaré residues and we introduce the fundamental forms of second and third kind on closed Riemann surfaces. In \S\ref{sec:introofformK} we present the Maurer-Cartan element~$\pmb K$ on $C_n(\hat C)$ which gives rise to the connection $\nabla_\mathcal{P}$ on $\mathcal{P}$ from~\cite{EnrConfSp}. We devote \S\ref{sec:finalcorrection} and \S\ref{ssec:connecJ} to the construction of the element~${\pmb J}_{\tilde x,\pmb \beta}$, which is obtained from~$\pmb K$ by gauge conjugation. In \S\ref{sec:notationchange} we discuss the dependence of~${\pmb J}_{\tilde x,\pmb \beta}$ on~$\tilde x$ and~$\pmb \beta$.

\subsection{The Schottky formalism}\label{Schottkyformalism}

Throughout this section, we set $x:=p(\tilde x)\in C$.

\subsubsection{A fundamental domain of $C$ in $\tilde C$}
In \S\ref{sect:RS}, we made the choice of an element $e\in E$, a universal covering $p:\tilde C\to C$ and a point $\infty\in C$. This choice induces an isomorphism $\mathrm{can}_e:\mathrm{Aut}(\tilde C/C)\tilde{\to}\langle A_i,B_i\rangle/(\prod_i(A_i,B_i))$ (see Lemma~\ref{lemma:28122020b}). 
For $i\in[\![1,h]\!]$ define $\mathsf{A}_i,\mathsf{B}_i\in\mathrm{Aut}(\tilde C/C)$ as $\mathrm{can}_e^{-1}(A_i),\mathrm{can}_e^{-1}(B_i)$.

By Definitions~\ref{def:1:7:2302} and~\ref{def:1:5:2403}, the choice of $\tilde x\in \tilde C \setminus p^{-1}(\infty)$ gives rise to an element $e_{\tilde x}\in \mathrm{Iso}_\infty(\pi_1(C\setminus\infty,x),\langle A_i,B_i\rangle)/(\prod_i(A_i,B_i))$, where $\mathrm{Iso}_\infty(\pi_1(C\setminus\infty,x),\langle A_i,B_i\rangle)$ was defined as the set of isomorphisms $\pi_1(C\setminus\infty,x)
\to\langle A_i,B_i\rangle$ which take the conjugacy class attached to~$\infty$ to the conjugacy class of $\prod_i(A_i,B_i)$, and $(\prod_i(A_i,B_i))$ acts on it by conjugation.

Let $\tilde e_{\tilde x}\in\mathrm{Iso}_\infty(\pi_1(C\setminus\infty,x),\langle A_i,B_i\rangle)$ 
be a lift of $e_{\tilde x}$. For $i\in\mathbb Z/g\mathbb Z$, let $\mathcal A_i,\mathcal B_i$ be  loops in 
$C\setminus\infty$, based at~$x$ and intersecting only at 
this point, representing the elements $\tilde e_{\tilde x}^{-1}(A_{\tilde i}),\tilde e_{\tilde x}^{-1}(B_{\tilde i})\in\pi_1(C\setminus\infty,x)$, where~$\tilde i$ is the representative of~$i$ in $[\![1,h]\!]$ . 
Denote by $D$ the complement in $C$ of the union of these loops; it is a simply-connected open subset of~$C$.
See Figure~\ref{fig1} for a picture of~$D$ and the identifications of~$\partial D$ for~$h=2$.

\begin{figure}[h!]
\begin{center}
\tikzpicture[scale=0.9]
\scope[xshift=-5cm,yshift=-0.4cm,decoration={
    markings,
    mark=at position 0.6 with {\arrow{>}}}]
\draw[thick][postaction={decorate}] (0,0) node{$\bullet$} ..controls (1,0) .. (2,0) node{$\bullet$} ;
\draw[thick][postaction={decorate}] (2,0) node{$\bullet$} ..controls (2.705,0.705) .. (3.41,1.41) node{$\bullet$} ;
\draw[thick][postaction={decorate}] (3.41,1.41) node{$\bullet$} ..controls (3.41,2.41) .. (3.41,3.41) node{$\bullet$} ;
\draw[thick][postaction={decorate}] (3.41,3.41) node{$\bullet$} ..controls (2.705,4.115) .. (2,4.82) node{$\bullet$} ;
\draw[thick][postaction={decorate}] (2,4.82) node{$\bullet$} ..controls (1,4.82) .. (0,4.82) node{$\bullet$} ;
\draw[thick][postaction={decorate}] (0,4.82) node{$\bullet$} ..controls (-0.705,4.115) .. (-1.41,3.41) node{$\bullet$} ;
\draw[thick][postaction={decorate}] (-1.41,3.41) node{$\bullet$} ..controls (-1.41,2.41) .. (-1.41,1.41) node{$\bullet$} ;
\draw[thick][postaction={decorate}] (-1.41,1.41) node{$\bullet$} ..controls (-0.705,0.705) .. (0,0) node{$\bullet$} ;
\node at (-0.1,-0.3) {$x$};
\node at (2.1,-0.3) {$x$};
\node at (3.71,1.31) {$x$};
\node at (3.71,3.51) {$x$};
\node at (2.1,5.12) {$x$};
\node at (-0.1,5.12) {$x$};
\node at (-1.71,3.51) {$x$};
\node at (-1.71,1.31) {$x$};
\node at (1.5,2.91) {$\bullet$};
\node at (1.9,3.11) {$\infty$};
\node at (0.6,1.91) {${\pmb D}$};
\node at (1,-0.45) {$\mathcal{A}_1$};
\node at (3.1,0.4) {$\mathcal{B}_1$};
\node at (4,2.41) {$\mathcal{A}^{-1}_1$};
\node at (3.22,4.45) {$\mathcal{B}^{-1}_1$};
\node at (1,5.25) {$\mathcal{A}_2$};
\node at (-1,4.45) {$\mathcal{B}_2$};
\node at (-2,2.41) {$\mathcal{A}^{-1}_2$};
\node at (-1.2,0.4) {$\mathcal{B}^{-1}_2$};
\endscope
\endtikzpicture
\caption{The domain $D\subset C$ as the interior of a polygon with boundary identifications.}\label{fig1}
\end{center}
\end{figure}

Set $\theta_i:=\prod_{j=1}^{i-1}(\mathsf{A}_j,\mathsf{B}_j)\in
\mathrm{Aut}(\tilde C/C)$, and consider in~$\tilde C$ the collection of points $u_i := \theta_i (\tilde x)$, 
$v_i := \theta_i \mathsf{A}_i (\tilde x)$, 
$w_i := \theta_i \mathsf{A}_i \mathsf{B}_i (\tilde x)$, 
$z_i := \theta_i \mathsf{A}_i \mathsf{B}_i \mathsf{A}^{-1}_i (\tilde x)$
for $i\in\mathbb Z/g\mathbb Z$, where we confuse elements of~$\mathbb Z/g\mathbb Z$ with 
their representatives in $[\![1,h]\!]$; in particular, $u_1=\tilde x$. 
These are all elements of~$p^{-1}(x)$.

\begin{lem}\label{lemma:4:1:02032021}
For any $i\in\mathbb Z/g\mathbb Z$, the lift of~$\mathcal A_{i}$ (resp.~$\mathcal B_i$,~$\mathcal A_i^{-1}$,~$\mathcal B_i^{-1}$) 
starting at~$u_i$ (resp. $v_i$, $w_i$, $z_i$) is a path in~$\tilde C$ ending at~$v_i$ (resp. $w_i$, $z_i$, $u_{i+1}$). 
This path will be denoted $\tilde{\mathcal A}_i$ (resp.~$\tilde{\mathcal B}_i$,~$\widetilde{\mathcal A_i^{-1}}$,~$\widetilde{\mathcal B_i^{-1}}$).  
\end{lem}

\begin{proof} We prove this statement only in the case of~$u_i$. Then $u_i=\theta_i(\tilde x)$. The class of~$\mathcal A_i$ in $\pi_1(C,x)$ is the image 
under the morphism $\pi_1(C\setminus\infty,x)\to\pi_1(C,x)$ of $\tilde e_{\tilde x}^{-1}(A_i)$ which by~\eqref{dcetdet:02032021} is equal to $(e^{\mathrm{fill}}_{\tilde x})^{-1}(A_i)$. The endpoint of the lift to~$\tilde C$ 
of~$\mathcal A_i$ starting at~$u_i$ is then, by Lemma~\ref{lemma:standard} (a), $u_i\cdot (e^{\mathrm{fill}}_{\tilde x})^{-1}(A_i)$. 
Since $u_i=\theta_i(\tilde x)$, eq.~\eqref{TaggedIdentity} implies that this point is equal to 
$\theta_i(\tilde x\cdot (e^{\mathrm{fill}}_{\tilde x})^{-1}(A_i))$, which by Lemma~\ref{lemma:standard} (b) is equal to 
$\theta_i(\mathrm{aut}^{\tilde x}_{(e^{\mathrm{fill}}_{\tilde x})^{-1}(A_i)}(\tilde x))$. By the definition of ${\rm can}_e$ (see eq.~\eqref{triangle:02032021}),
$\mathrm{aut}^{\tilde x}_{(e^{\mathrm{fill}}_{\tilde x})^{-1}(A_i)}=\mathrm{can}_e^{-1}(A_i)$, which is equal to~$\mathsf{A}_i$.  
Therefore the endpoint is equal to $\theta_i\circ\mathsf{A}_i(\tilde x)=v_i$. 
\end{proof}

\begin{lem}\label{lem:1829:02032021}
The paths~$\widetilde{\mathcal A_i^{-1}}$ (resp.~$\widetilde{\mathcal B_i^{-1}}$)  and the inverses of the paths~$\tilde{\mathcal A}_i$ (resp.~$\tilde{\mathcal B}_i$)  are pairwise identified, up to change of orientation, through the actions of elements of $\mathrm{Aut}(\tilde C/C)$ as follows:
$$
\widetilde{\mathcal A_i^{-1}}={\rm Ad}_{\theta_i\mathsf{A}_i } \mathsf{B}_i
\big( \tilde{\mathcal A}_i^{-1}\big)\quad\quad\quad {\rm (resp.  }\,\,\widetilde{\mathcal B_i^{-1}}={\rm Ad}_{\theta_i\mathsf{A}_i\mathsf{B}_i}(\mathsf{A}_i^{-1})\big(\tilde{\mathcal B}_i^{-1}\big){ \rm)}\,.
$$ 
\end{lem}

\begin{proof} 
In the first (resp. second) case, one checks that the said element maps the endpoints $(v_i,u_i)$ of~$\tilde{\mathcal A}_i^{-1}$ to the endpoints $(w_i,z_i)$ of~$\widetilde{\mathcal A_i^{-1}}$ (resp. the endpoints $(w_i,v_i)$ of~$\tilde{\mathcal B}_i^{-1}$ to the endpoints $(z_i,x_{i+1})$ of~$\widetilde{\mathcal B_i^{-1}}$).
\end{proof}

\begin{lem}\label{lemma43}
(a) The collection of paths in~$\tilde C$ defined in Lemma~\ref{lemma:4:1:02032021} defines a non-intersecting loop in~$\tilde C$.

(b) There exists a unique continuous map $D\to\tilde C$ such that $D\to\tilde C\to C$ is the canonical injection
and such that the boundary of its image~$\tilde D$ is the loop from (a). The restriction of $p:\tilde C\to C$ to 
$\tilde D$ induces a bijection $D\to\tilde D$. One may construct a fundamental domain of~$C$ in~$\tilde C$ as the union of~$\tilde D$ with a suitable subset of its boundary.  

(c) The intersection $\tilde D\cap p^{-1}(\infty)$ contains exactly one point, which will be denoted~$\tilde\infty$. 
\end{lem}

\begin{proof} (a) follows from the fact that any pair of paths among the $\mathcal A_i,\mathcal B_i$ meet only at~$x$, and from the fact that the elements of $\mathrm{Aut}(\tilde C/C)$ taking~$\tilde{\mathcal A}_i^{-1}$ to~$\widetilde{\mathcal A_i^{-1}}$ and~$\tilde{\mathcal B}_i^{-1}$ to~$\widetilde{\mathcal B_i^{-1}}$ are nontrivial by Lemma~\ref{lem:1829:02032021}. 

(b) follows from the uniformization theorem, i.e. from the fact that~$\tilde C$ is isomorphic, as an analytic variety, 
to the complex upper half-plane. 

(c) follows from (b) and from the fact that the paths~$\mathcal A_i$ and~$\mathcal B_i$ do not meet the point~$\infty$.
\end{proof}

%We have: \widetilde{A_i^{-1}} = {\rm Ad}_{\prod_{j=1}^{i-1} (\alpha_j,\beta_j)\alpha_i } \beta_i
%times \tilde A_i^{-1}

The domain~$\tilde D$, its boundary components and their identifications are illustrated for~$h=2$ in Figure~\ref{fig2}. 

\begin{figure}[h!]
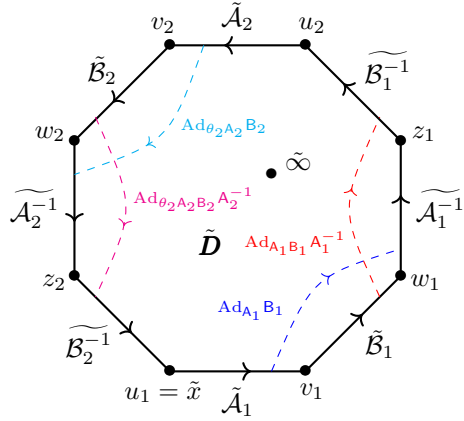

\begin{center}
\tikzpicture[scale=0.9]
\scope[xshift=-5cm,yshift=-0.4cm,decoration={
    markings,
    mark=at position 0.6 with {\arrow{>}}}]
\draw[thick][postaction={decorate}] (0,0) node{$\bullet$} ..controls (1,0) .. (2,0) node{$\bullet$} ;
\draw[thick][postaction={decorate}] (2,0) node{$\bullet$} ..controls (2.705,0.705) .. (3.41,1.41) node{$\bullet$} ;
\draw[thick][postaction={decorate}] (3.41,1.41) node{$\bullet$} ..controls (3.41,2.41) .. (3.41,3.41) node{$\bullet$} ;
\draw[thick][postaction={decorate}] (3.41,3.41) node{$\bullet$} ..controls (2.705,4.115) .. (2,4.82) node{$\bullet$} ;
\draw[thick][postaction={decorate}] (2,4.82) node{$\bullet$} ..controls (1,4.82) .. (0,4.82) node{$\bullet$} ;
\draw[thick][postaction={decorate}] (0,4.82) node{$\bullet$} ..controls (-0.705,4.115) .. (-1.41,3.41) node{$\bullet$} ;
\draw[thick][postaction={decorate}] (-1.41,3.41) node{$\bullet$} ..controls (-1.41,2.41) .. (-1.41,1.41) node{$\bullet$} ;
\draw[thick][postaction={decorate}] (-1.41,1.41) node{$\bullet$} ..controls (-0.705,0.705) .. (0,0) node{$\bullet$} ;
\node at (-0.1,-0.3) {$u_1=\tilde x$};
\node at (2.1,-0.3) {$v_1$};
\node at (3.78,1.31) {$w_1$};
\node at (3.75,3.51) {$z_1$};
\node at (2.1,5.12) {$u_2$};
\node at (-0.1,5.12) {$v_2$};
\node at (-1.77,3.51) {$w_2$};
\node at (-1.71,1.31) {$z_2$};
\node at (1.5,2.91) {$\bullet$};
\node at (1.9,3.11) {$\tilde\infty$};
\node at (0.6,1.91) {$\tilde{\pmb D}$};
\node at (1,-0.45) {$\tilde{\mathcal{A}}_1$};
\node at (3.1,0.4) {$\tilde{\mathcal{B}}_1$};
\node at (4,2.41) {$\widetilde{\mathcal{A}^{-1}_1}$};
\node at (3.22,4.45) {$\widetilde{\mathcal{B}^{-1}_1}$};
\node at (1,5.25) {$\tilde{\mathcal{A}}_2$};
\node at (-1,4.45) {$\tilde{\mathcal{B}}_2$};
\node at (-2,2.41) {$\widetilde{\mathcal{A}^{-1}_2}$};
\node at (-1.2,0.4) {$\widetilde{\mathcal{B}^{-1}_2}$};
\draw[dashed][blue][postaction={decorate}] (1.5,0) ..controls (2.1,1.4) .. (3.41,1.81)  ;
\node[blue] at (1.2,0.9) {\tiny{${\rm Ad}_{\mathsf{A}_1}\mathsf{B}_1$}};
\draw[dashed][red][postaction={decorate}] (3.1,1.1) ..controls (2.5,2.41) .. (3.1,3.76)  ;
\node[red] at (1.85,1.9) {\tiny{${\rm Ad}_{\mathsf{A}_1\mathsf{B}_1}\mathsf{A}^{-1}_1$}};
\draw[dashed][cyan][postaction={decorate}] (0.5,4.82) ..controls (0,3.4) .. (-1.41,2.91)  ;
\node[cyan] at (0.8,3.6) {\tiny{${\rm Ad}_{\theta_2\mathsf{A}_2}\mathsf{B}_2$}};
\draw[dashed][magenta][postaction={decorate}] (-1.1,3.76) ..controls (-0.5,2.41) .. (-1.1,1.1)  ;
\node[magenta] at (0.35,2.55) {\tiny{${\rm Ad}_{\theta_2\mathsf{A}_2\mathsf{B}_2}\mathsf{A}_2^{-1}$}};
\endscope
\endtikzpicture
\caption{The domain $\tilde D\subset \tilde C$.}\label{fig2}
\end{center}
\end{figure}

\subsubsection{Construction of the Schottky covering $\hat C$ of $C$}

Let $N\triangleleft\langle A_i,B_i\rangle/(\prod_i(A_i,B_i))$ be the normal subgroup generated by the elements $A_i$ ($i\in[\![1,h]\!]$). 
Its preimage $\mathrm{can}_e^{-1}(N)$ under the isomorphism $\mathrm{can}_e$ is a normal subgroup of $\mathrm{Aut}(\tilde C/C)$.

%By Lemma \ref{lemma:28122020b}, the choice of $e\in E$ gives rise to an isomorphism $\mathrm{can}_e:\mathrm{Aut}(\tilde C/C)\to
%\langle A_i,B_i\rangle/(\prod_i(A_i,B_i))$. \textcolor{red}{FZ: already mentioned in the previous section, perhaps better: "}

\begin{defn}We  set $\hat C:=\tilde C/\mathrm{can}_e^{-1}(N)$. The natural map $\pi:\hat C\to C$
is called the Schottky covering of $C$. 
\end{defn}

Since $N$ is normal in $\langle A_i,B_i\rangle/(\prod_i(A_i,B_i))$, and since the quotient group is freely generated by 
the classes of the elements $B_i$ ($i\in[\![1,h]\!]$), the projection $\hat C\to C$ is a Galois covering and 
$\mathrm{Aut}(\hat C/C)$ is isomorphic to the free group $\langle B_i\rangle$. 

\subsubsection{Some geometry in $\hat C$}

Let $\hat p:\tilde C\to\hat C$ be the canonical projection. There is a group morphism 
$$
\mathrm{Aut}(\tilde C/C)\to
\mathrm{Aut}(\hat C/C)\,,\qquad\qquad \theta\mapsto\hat\theta \,,
$$
where $\hat\theta\in\mathrm{Aut}(\hat C/C)$ is uniquely determined by $\hat p\circ\theta=\hat\theta\circ\hat p$. 

\begin{lem}\label{lemma45}
(a) There is a unique group isomorphism $\mathrm{can}_e^\wedge:\mathrm{Aut}(\hat C/C)\tilde\to\langle B_i\rangle$, such that the 
following diagram commutes:
\begin{equation*}
\begin{tikzcd}
\mathrm{Aut}(\tilde C/C) \arrow[r,"{\mathrm{can}_e}"] \arrow[d]
& \langle A_i,B_i\rangle/(\prod_i(A_i,B_i)) \arrow[d]  \\
\mathrm{Aut}(\hat C/C) \arrow[r,"{\mathrm{can}_e^\wedge}"]
& \langle B_i\rangle
\end{tikzcd}
\end{equation*}

(b) One has $\hat{\mathsf{A}}_i=1$ and $\hat{\mathsf{B}}_i=(\mathrm{can}_e^\wedge)^{-1}(B_i)$.  
\end{lem}

\begin{proof}
For the proof of (a), recall that $\langle B_i\rangle = \big(\langle A_i,B_i\rangle/(\prod_i(A_i,B_i))\big)/N$, so we need to prove that ${\rm can}_e$ takes the kernel of the left vertical morphism ${\rm Aut}(\tilde C/C)\to {\rm Aut}(\hat C/C)$ to the normal subgroup $N\triangleleft\langle A_i,B_i\rangle/(\prod_i(A_i,B_i))$. 

By definition of $\hat\theta$, and because $\hat p$ is surjective, $\hat\theta =1$ if and only if $\hat p\circ\theta =\hat p$. Therefore ${\rm ker}({\rm Aut}(\tilde C/C)\to {\rm Aut}(\hat C/C))=\{\theta\in{\rm Aut}(\tilde C/C)\,|\,\,\hat p\,\circ\,\theta =\hat p\}$, which is in turn equal to ${\rm can}_e^{-1}(N)$, because $\hat p: \tilde C\twoheadrightarrow \hat C=\tilde C/{\rm can}_e^{-1}(N)$.

(b) follows from (a) and from the fact that we have defined $\mathsf{A}_i$ (resp. $\mathsf{B}_i$) as ${\rm can}_e^{-1}(A_i)$ (resp. ${\rm can}_e^{-1}(B_i)$).
\end{proof}

The group $\mathrm{Aut}(\hat C/C)$ is then freely generated by the elements $\hat{\mathsf{B}}_i$, $i\in[\![1,h]\!]$. 

\begin{defn}
(a) We set $\hat x:=\hat p(\tilde x)$; this is a point in $\hat C$.

(b) For $i\in[\![1,h]\!]$, we set $\hat{\mathcal A}_i:=\hat p(\tilde{\mathcal A}_i)$, 
$\hat{\mathcal B}_i:=\hat p(\tilde{\mathcal B}_i)$; these are paths in $\hat C$. 
\end{defn}

\begin{lem}\label{lemma47}
(a) For $i\in[\![1,h]\!]$, $\hat{\mathcal A}_i$ is a loop in $\hat C$ based at $\hat x$. 

(b) For $i\in[\![1,h]\!]$, $\hat{\mathcal B}_i$ is a path in $\hat C$ starting at $\hat x$ and ending at $\hat{\mathsf{B}}_i(\hat x)$. 

(c) For $i\in[\![1,h]\!]$, $\hat p(\widetilde{\mathcal A_i^{-1}})=\hat{\mathsf{B}}_i(\hat{\mathcal A}_i)^{-1}$ and 
$\hat p(\widetilde{\mathcal B_i^{-1}})=\hat{\mathcal B}_i^{-1}$

(d) The nontrivial intersections between the paths $\hat{\mathcal A}_i$, $\hat{\mathsf{B}}_i(\hat{\mathcal A}_i)$, 
$\hat{\mathcal B}_i$, $i\in[\![1,h]\!]$ are: 

$\forall i\neq j$, $\hat{\mathcal A}_i\cap \hat{\mathcal A}_j=\hat{\mathcal B}_i\cap \hat{\mathcal B}_j=\{\hat x\}$; 
$\forall i,j$, $\hat{\mathcal A}_i\cap\hat{\mathcal B}_j=\{\hat x\}$; 
$\forall $i, $\hat{\mathsf{B}}_i(\hat{\mathcal A}_i)\cap\hat{\mathcal B}_i=\{\hat{\mathsf{B}}_i(\hat x)\}$. 
\end{lem}

\begin{proof}
By Lemma~\ref{lemma:4:1:02032021},~$\hat{\mathcal{A}}_i$ (resp.~$\hat{\mathcal{B}}_i$) starts at the point~$\hat p(u_i)$ (resp.~$\hat p(v_i)$) and ends at the point~$\hat p(v_i)$ (resp.~$\hat p(w_i)$). Recall that $u_i=\theta_i(\tilde x)$, $v_i=\theta_i\mathsf{A}_i(\tilde x)$ and $u_i=\theta_i\mathsf{A}_i\mathsf{B}_i(\tilde x)$, where $\theta_i=\prod_{j=1}^{i-1}(\mathsf{A}_i,\mathsf{B}_i)$. Statements~(a) and~(b) immediately follow from the fact (see point (b) of Lemma~\ref{lemma45}) that $\hat{\mathsf{A}}_i=1$, which implies that also $\hat{\theta_i}=1$ and so that $\hat p(u_i)=\hat p(v_i)=\hat x$ and that $\hat p(w_i)=\hat{\mathsf{B}}(\hat x)$.

Moreover, combining the fact that $\hat{\mathsf{A}}_i=\hat{\theta_i}=1$ with Lemma~\ref{lem:1829:02032021} we obtain the statement (c).

Finally,~(d) follows from~(a) and~(b) and from the fact that the only common intersection of the loops $\mathcal{A}_i,\mathcal{B}_i$ on the curve~$C$ is the point~$p(\tilde x)$.
\end{proof}

\begin{defn}
We define $\hat D:=\hat p(\tilde D)$. This is a subset of~$\hat C$. 
\end{defn}

\begin{lem} 
$\hat D$ is an open subset of~$\hat C$, whose boundary is equal to $\partial\hat D=\cup_i(\hat{\mathcal A}_i \cup \hat{\mathsf{B}}_i(\hat{\mathcal A}_i)\cup\hat{\mathcal B}_i)$.  
The projection $\pi:\hat C\to C$ sets up a bijection $\hat D\to D$. 
\end{lem}

\begin{proof}
All the statements follow from part~(b) of Lemma~\ref{lemma43}, which must be combined with part~(c) of Lemma~\ref{lemma47} to obtain $\partial\hat D$.
\end{proof}

\begin{defn}
We define~$\overline{\hat D}$ to be the closure of~$\hat D$ in~$\hat C$, so 
$\overline{\hat D}=\hat D\cup\partial\hat D$. 
\end{defn}

\begin{lem}
The boundary of~$\overline{\hat D}$  in~$\hat C$ is equal to 
$\partial\overline{\hat D}=\cup_i(\hat{\mathcal A}_i \cup \hat{\mathsf{B}}_i(\hat{\mathcal A}_i))$.  
\end{lem}

\begin{proof}
Lemma~\ref{lemma43} implies that the domain~$\hat D$ is bounded by the loop $\prod_{i=1}^h\hat{\mathcal{A}}_i\hat{\mathcal{B}}_i\hat{\mathsf{B}}_i(\hat{\mathcal A}_i)^{-1}\hat{\mathcal{B}}_i^{-1}$ (see Figure~\ref{fig3} for the $h=2$ case). The fact that both~$\hat{\mathcal{B}}_i$ and~$\hat{\mathcal{B}}_i^{-1}$ belong to this loop implies that any point in $\hat{\mathcal{B}}_i\setminus\{\hat x,\hat{\mathsf{B}}_i(\hat x)\}$ has a neighborhood which is completely contained in~$\overline{\hat D}$, and therefore they do not belong to~$\partial\overline{\hat D}$. On the other hand, any point in $\hat{\mathcal{A}}_i\setminus\{\hat x\}$ or $\hat{\mathsf{B}}_i(\hat{\mathcal{A}}_i)\setminus\{\hat{\mathsf{B}}_i(\hat x)\}$ belongs to the boundary~$\partial\overline{\hat D}$ because, by part~(d) of Lemma~\ref{lemma47}, they only appear once in the loop bounding~$\hat D$. Finally, the points $\hat x$ and $\hat{\mathsf{B}}_i(\hat x)$ must also belong to~$\partial\overline{\hat D}$ because otherwise~$\partial\overline{\hat D}$ would not be closed.
\end{proof}

A fundamental domain of~$C$ in~$\hat C$ is $\overline{\hat D}-\cup_i\hat{\mathsf{B}}_i(\hat{\mathcal A}_i)$, which 
both contains~$\hat D$ and is contained in~$\overline{\hat D}$ . 
We will set $\hat\infty:=\hat p(\tilde\infty)$. Since~$\tilde\infty\in\tilde D$,~$\hat\infty$ belongs to~$\hat D$. 
The domain~$\hat D$, its boundary components and their identifications are illustrated for~$h=2$ in Figures~\ref{fig3} and~\ref{fig4}.

\begin{figure}[h!]
\begin{center}
\tikzpicture[scale=0.9]
\scope[xshift=-5cm,yshift=-0.4cm,decoration={
    markings,
    mark=at position 0.6 with {\arrow{>}}}]
\draw[thick][postaction={decorate}] (0,0) node{$\bullet$} ..controls (1,0) .. (2,0) node{$\bullet$} ;
\draw[thick][postaction={decorate}] (2,0) node{$\bullet$} ..controls (2.705,0.705) .. (3.41,1.41) node{$\bullet$} ;
\draw[thick][postaction={decorate}] (3.41,1.41) node{$\bullet$} ..controls (3.41,2.41) .. (3.41,3.41) node{$\bullet$} ;
\draw[thick][postaction={decorate}] (3.41,3.41) node{$\bullet$} ..controls (2.705,4.115) .. (2,4.82) node{$\bullet$} ;
\draw[thick][postaction={decorate}] (2,4.82) node{$\bullet$} ..controls (1,4.82) .. (0,4.82) node{$\bullet$} ;
\draw[thick][postaction={decorate}] (0,4.82) node{$\bullet$} ..controls (-0.705,4.115) .. (-1.41,3.41) node{$\bullet$} ;
\draw[thick][postaction={decorate}] (-1.41,3.41) node{$\bullet$} ..controls (-1.41,2.41) .. (-1.41,1.41) node{$\bullet$} ;
\draw[thick][postaction={decorate}] (-1.41,1.41) node{$\bullet$} ..controls (-0.705,0.705) .. (0,0) node{$\bullet$} ;
\node at (-0.1,-0.3) {$\hat x$};
\node at (2.1,-0.3) {$\hat x$};
\node at (4,1.31) {$\hat{\mathsf{B}}_1(\hat x)$};
\node at (4.05,3.51) {$\hat{\mathsf{B}}_1(\hat x)$};
\node at (2.1,5.12) {$\hat x$};
\node at (-0.1,5.12) {$\hat x$};
\node at (-2,3.51) {$\hat{\mathsf{B}}_2(\hat x)$};
\node at (-2,1.31) {$\hat{\mathsf{B}}_2(\hat x)$};
\node at (1.5,2.91) {$\bullet$};
\node at (1.9,3.11) {$\hat\infty$};
\node at (0.6,1.91) {$\hat{\pmb D}$};
\node at (1,-0.45) {$\hat{\mathcal{A}}_1$};
\node at (3.1,0.4) {$\hat{\mathcal{B}}_1$};
\node at (4.35,2.41) {$\hat{\mathsf{B}}_1(\hat{\mathcal{A}}_1)^{-1}$};
\node at (3.22,4.45) {$\hat{\mathcal{B}}^{-1}_1$};
\node at (1,5.25) {$\hat{\mathcal{A}}_2$};
\node at (-1,4.45) {$\hat{\mathcal{B}}_2$};
\node at (-2.35,2.41) {$\hat{\mathsf{B}}_2(\hat{\mathcal{A}}_2)^{-1}$};
\node at (-1.2,0.4) {$\hat{\mathcal{B}}^{-1}_2$};
\draw[dashed][blue][postaction={decorate}] (1.5,0) ..controls (2.1,1.4) .. (3.41,1.81)  ;
\node[blue] at (1.55,0.9) {\tiny{$\hat{\mathsf{B}}_1$}};
\draw[dashed][cyan][postaction={decorate}] (0.5,4.82) ..controls (0,3.4) .. (-1.41,2.91)  ;
\node[cyan] at (0.35,3.6) {\tiny{$\hat{\mathsf{B}}_2$}};
\endscope
\endtikzpicture
\caption{The domain $\hat D\subset \hat C$ as the interior of a polygon with boundary identifications.}\label{fig3}
\end{center}
\end{figure}

\begin{figure}[h!]
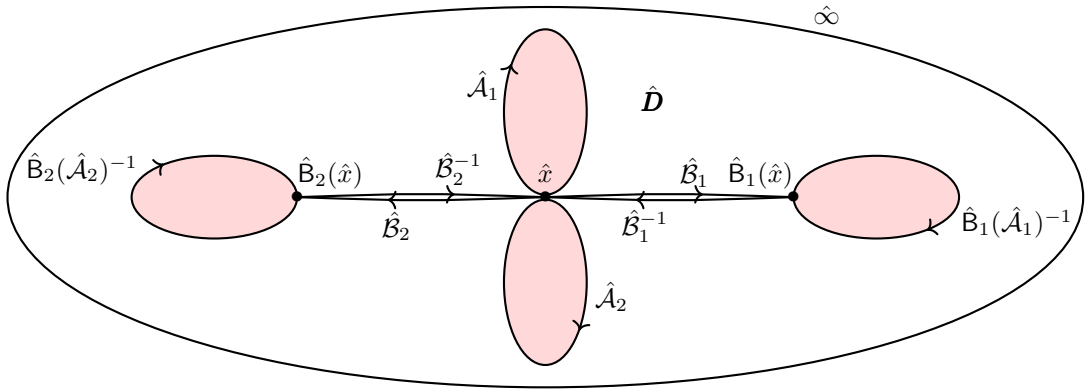

\begin{center}
\tikzpicture[scale=1.1]
\scope[xshift=-5cm,yshift=-0.4cm,decoration={
    markings,
    mark=at position 0.4 with {\arrow{<}}}]
\filldraw[color=black, fill=red!15, thick][postaction={decorate}] (0,-1.03) ellipse (-0.5 and -1);
\filldraw[color=black, fill=red!15, thick][postaction={decorate}] (0,1.03) ellipse (0.5 and 1);
\draw[thick][postaction={decorate}] (3,0) ..controls (1.5,0.05) .. (0,0) ;
\draw[thick][postaction={decorate}] (0,0) ..controls (1.5,-0.05) .. (3,0) ;
\draw[thick][postaction={decorate}] (-3,0) ..controls (-1.5,-0.05) .. (0,0) ;
\draw[thick][postaction={decorate}] (0,0) ..controls (-1.5,0.05) .. (-3,0) ;
\filldraw[color=black, fill=red!15, thick][postaction={decorate}] (-4,0) ellipse (1 and 0.5);
\filldraw[color=black, fill=red!15, thick][postaction={decorate}] (4,0) ellipse (-1 and -0.5);
\draw[thick] (0,0) ellipse (6.5 and 2.3);
\node at (0,0) {$\bullet$};
\node at (3,0) {$\bullet$};
\node at (-3,0) {$\bullet$};
\node at (0,0.3) {$\hat x$};
\node at (2.6,0.3) {$\hat{\mathsf{B}}_1(\hat x)$};
\node at (-2.6,0.3) {$\hat{\mathsf{B}}_2(\hat x)$};
\node at (-0.75,1.35) {$\hat{\mathcal{A}}_1$};
\node at (0.8,-1.2) {$\hat{\mathcal{A}}_2$};
\node at (1.8,0.3) {$\hat{\mathcal{B}}_1$};
\node at (1.2,-0.35) {$\hat{\mathcal{B}}_1^{-1}$};
\node at (-5.6,0.35) {$\hat{\mathsf{B}}_2(\hat{\mathcal{A}}_2)^{-1}$};
\node at (5.7,-0.3) {$\hat{\mathsf{B}}_1(\hat{\mathcal{A}}_1)^{-1}$};
\node at (-1.8,-0.35) {$\hat{\mathcal{B}}_2$};
\node at (-1.05,0.35) {$\hat{\mathcal{B}}_2^{-1}$};
\node at (1.3,1.2) {$\hat{\pmb D}$};
%\node at (2.7,1.4) {$\bullet$};
\node at (3.4,2.2) {$\hat{\infty}$};
\endscope
\endtikzpicture
\caption{The external boundary should be thought of as the point $\hat{\infty}$. The domain $\hat D\subset \hat C$ is the complement of the coloured regions.}\label{fig4}
\end{center}
\end{figure}

The Riemann surface structure on~$C$ induces a Riemann surface structure on~$\hat C$. The residue theorem on $\overline{\hat D}$ implies that for any meromorphic differential $\omega$ on $\hat C$
without poles on~$\partial\overline{\hat D}$, one has\footnote{We set 
$\mathrm i:=\sqrt{-1}$.}
\begin{equation}\label{residue:formula}
2\pi\mathrm{i}\sum_{P\in\hat D}\mathrm{res}_P(\omega)+\sum_{i=1}^h\int_{\hat{\mathcal A}_i}(\hat{\mathsf{B}}_i^*-1)(\omega)=0. 
\end{equation}

\subsection{Residues along diagonals and fundamental forms}\label{sec:backgroundres}

\subsubsection{Poincaré residues}\label{sec:res}

Let~$X$ be a complex manifold, $i:Y\hookrightarrow X$ be a smooth analytic hypersurface, and~$\omega_X$,~$\omega_Y$ be the respective canonical bundles. The adjunction
formula says that\footnote{For~$E$ a holomorphic vector bundle on~$X$ and~$D=\sum n_iY_i$ a divisor ($Y_i\subset X$ irreducible analytic hypersurfaces, $n_i\in\mathbb{Z}$), we set $E(D):=E\otimes \mathcal{O}_X(D)$, where $\mathcal{O}_X(D)$ is the line bundle of meromorphic functions with poles of order at most $n_i$ at $Y_i$ if $n_i$ is positive, and zeros of order at least $-n_i$ at $Y_i$ if $n_i$ is negative.} $i^*(\omega_X(Y)) \simeq \omega_Y$. It follows that, if~$E$ is a vector bundle on~$X$, then $i^*(E \otimes \omega_X(Y)) \simeq i^*(E) \otimes
\omega_Y$. This implies the existence of a map
$\textrm{res}_Y^E : \Gamma(X,E \otimes \omega_X(Y))\rightarrow \Gamma(Y,i^*(E) \otimes \omega_Y)$, called the \emph{Poincaré residue} (we will sometimes drop~$E$ from the notation).

\subsubsection{Residue for $(\Omega^1_C)^{\boxtimes 2}(2\Delta)/(\Omega^1_C)^{\boxtimes 2}$}

Let~$C$ be a closed Riemann surface, so that $\omega_C=\Omega^1_C$, and let~$\Delta$ be the diagonal in~$C^2$. Consider the exact sequence of sheaves of $\mathcal{O}_{C^2}$-modules
\begin{equation}\label{eq:210712n1}
0\longrightarrow (\Omega^1_C)^{\boxtimes 2}\longrightarrow (\Omega^1_C)^{\boxtimes 2}(2\Delta)\longrightarrow (\Omega^1_C)^{\boxtimes 2}(2\Delta)/(\Omega^1_C)^{\boxtimes 2}\longrightarrow 0\,.
\end{equation}
Observe that if $U \subset C$ is an open subset, then the class $[d_1d_2\log(F_U^{(1)}-F_U^{(2)})] \in \Gamma(U,(\Omega^1_C)^{\boxtimes 2}(2\Delta)/(\Omega^1_C)^{\boxtimes 2})$ is
independent of the choice of $F_U\in\Gamma(U,\mathcal{O}_C)$ such that~$dF_U$ does not vanish on~$U$; here $F_U^{(1)},F_U^{(2)} \in \Gamma(U^2,\mathcal{O}_{C^2})$ denote $F_U \otimes 1$ and $1\otimes F_U$, respectively, and $d_1,d_2$ denote the differentials on the first and on the second copy of~$C$ in~$C^2$, respectively. Moreover, these various sections glue together and define a canonical section in $\Gamma(C^2,(\Omega^1_C)^{\boxtimes 2}(2\Delta)/(\Omega^1_C)^{\boxtimes 2})$. Multiplication by this section gives rise to an isomorphism of sheaves of $\mathcal{O}_{C^2}$-modules
\begin{equation}\label{eq:210712n2}
\mathcal{O}_{C^2}/\mathcal{O}_{C^2}(-2\Delta)\,\,\tilde{\rightarrow} \,\,(\Omega^1_C)^{\boxtimes 2}(2\Delta)/(\Omega^1_C)^{\boxtimes 2}\,.
\end{equation}
A splitting of the exact sequence 
$$
0\longrightarrow \mathcal{O}_{C^2}(-\Delta)/\mathcal{O}_{C^2}(-2\Delta)\longrightarrow\mathcal{O}_{C^2}/\mathcal{O}_{C^2}(-2\Delta)\longrightarrow\mathcal{O}_{C^2}/\mathcal{O}_{C^2}(-\Delta)\longrightarrow 0
$$
of sheaves of vector spaces over~$C^2$ can be constructed as follows. For $U \subset C$ open subset, let $A_U:=\Gamma(U,\mathcal{O}_C)$ and $J_U:=\mathrm{ker}(m_U:A_U \otimes A_U
\to A_U)$. Then $\Gamma(U^2,\mathcal{O}_{C^2}/\mathcal{O}_{C^2}(-2\Delta)) \simeq (A_U \otimes A_U)/J_U^2$, while $\Gamma(U^2,\mathcal{O}_{C^2}/\mathcal{O}_{C^2}(-\Delta)) \simeq (A_U
\otimes A_U)/J_U \simeq A_U$, the last identification being induced by $m_U$. The map 
$$
\Gamma(U^2,\mathcal{O}_{C^2}/\mathcal{O}_{C^2}(-\Delta)) \simeq A_U \rightarrow (A_U\otimes A_U)/J_U^2 \simeq \Gamma(U^2,\mathcal{O}_{C^2}/\mathcal{O}_{C^2}(-2\Delta))\,,
$$
where the middle map is $a\mapsto a\otimes 1$, is a section of the canonical map
$$
\Gamma(U^2,\mathcal{O}_{C^2}/\mathcal{O}_{C^2}(-2\Delta))\rightarrow \Gamma(U^2,\mathcal{O}_{C^2}/\mathcal{O}_{C^2}(-\Delta))
$$
which induces an isomorphism
\begin{equation}\label{eq:210712n3}
\mathcal{O}_{C^2}/\mathcal{O}_{C^2}(-2\Delta) \simeq \mathcal{O}_{C^2}/\mathcal{O}_{C^2}(-\Delta) \oplus \mathcal{O}_{C^2}(-\Delta)/\mathcal{O}_{C^2}(-2\Delta)\,.
\end{equation}
of sheaves of vector spaces over $C^2$.

Combining the second morphism in~\eqref{eq:210712n1}, the inverse of~\eqref{eq:210712n2} and the isomorphism~\eqref{eq:210712n3}, one obtains a morphism of sheaves of vector spaces over $C^2$
\begin{equation*}
(\Omega^1_C)^{\boxtimes 2}(2\Delta)\,\rightarrow \,\mathcal{O}_{C^2}/\mathcal{O}_{C^2}(-\Delta) \oplus \mathcal{O}_{C^2}(-\Delta)/\mathcal{O}_{C^2}(-2\Delta)
\end{equation*}
which, upon taking global sections and identifying $\Gamma(C^2,\mathcal{O}_{C^2}/\mathcal{O}_{C^2}(-\Delta))\simeq \Gamma(C,\mathcal{O}_C)$ and $\Gamma(C^2,\mathcal{O}_{C^2}(-\Delta)/\mathcal{O}_{C^2}(-2\Delta))\simeq \Gamma(C,\Omega^1_C)$, induces a linear map
\begin{align*}
&\mathcal{R}\mathrm{es} : \Gamma(C^2,(\Omega^1_C)^{\boxtimes 2}(2\Delta))\,\rightarrow
\Gamma(C,\mathcal{O}_C) \oplus \Gamma(C,\Omega^1_C)\,.
\end{align*}

Using the isomorphism\footnote{Let~$N_{Y/X}$ denote the normal bundle of a complex submanifold $Y\subset X$, $\Delta$ be the diagonal in $X^2$ and $T^{1,0}\Delta$ be its holomorphic tangent bundle, then $N_{\Delta/X^2}\simeq T^{1,0}\Delta$. If the case of a Riemann surface~$C$, this implies that $N^*_{\Delta/C^2}\simeq \Omega^1_\Delta$. On the other hand, if $i:Y\hookrightarrow X$ is a smooth analytic hypersurface of a complex manifold $X$, then by the adjunction formula $N^*_{Y/X}\simeq i^*(\mathcal{O}_X(-Y))$, and therefore $N^*_{\Delta/C^2}\simeq i^*(\mathcal{O}_{C^2}(-\Delta))$. Combining these two remarks gives $i^*(\mathcal{O}_{C^2}(-\Delta)) \simeq \Omega^1_C$.} $i^*(\mathcal{O}_{C^2}(-\Delta))\simeq \Omega^1_C$, where $i:\Delta\hookrightarrow C^2$, one checks that the first component of $\mathcal{R}\mathrm{es}$ is equal to the Poincaré residue $\mathrm{res}_\Delta^{O_{C^2}(\Delta)}$.

If $\alpha\in \Gamma(C^2,(\Omega^1_C)^{\boxtimes 2}(2\Delta))$, then in a neighborhood of the diagonal one has $\alpha=a_{-2}(z)dzdw/(w-z)^2+a_{-1}(z)dzdw/(w-z)+O(1)dzdw$ in terms of local coordinates $z,w$ and holomorphic functions~$a_{-2}$ and~$a_{-1}$. One can check that $\mathcal{R}\mathrm{es}(\alpha)=(a_{-2}(z),a_{-1}(z)dz)$, so in particular $a_{-2}(z)$ and $a_{-1}(z)dz$ do not depend on the choice of local coordinates. For this reason, we will call $\mathcal{R}\mathrm{es}_{-2}(\alpha)$ and $\mathcal{R}\mathrm{es}_{-1}(\alpha)$ the first and second components of $\mathcal{R}\mathrm{es}(\alpha)$, respectively.

\subsubsection{Fundamental forms of the second and third kind}\label{sec:fundformIIkind}

In~\cite{Fay}, one defines an element $\varphi \in \Gamma(C^2,(\Omega^1_C)^{\boxtimes 2}(2\Delta))$ (denoted by $\omega$ in~\cite{Fay}, and by $\psi^{\underline{z}\underline{w}}$ in~\cite{EnrConfSp}), which satisfies $\mathcal{R}\mathrm{es}(\varphi)=(1,0)$, called the \emph{fundamental form of the second kind}.
It follows from~\cite{Fay}, eq.~(28) and~(ii) p.~16, that~$\varphi$ is invariant under exchanging the two factors of $C^2$. We will need two properties of $\varphi$, proven in~\cite{Fay}, Cor.~2.6~(iii), which we reformulate as follows. 

For any $(P,v)\in T^{1,0}C$, there is a well-defined map $\big|_{C\times (P,v)}: \Gamma(C^2,(\Omega^1_C)^{\boxtimes 2}(2\Delta))\to\Gamma(C,\Omega^1_C(2P))$ given by the pairing with $(P,v)$ in the second variable. Moreover, by the residue theorem, elements of $\Gamma(C,\Omega^1_C(2P))$ are differentials of the second kind, so that integration gives a pairing $H_1(C,\mathbb{Q})\times\Gamma(C,\Omega^1_C(2P))\to \mathbb{C}$ (for $\eta$ a loop in $C$ avoiding $P$, and $\alpha\in\Gamma(C,\Omega^1_C(2P))$, one has $\langle [\eta],\alpha\rangle=\int_\eta \alpha)$). Then one has $\langle A_i, \varphi\big|_{C\times (P,v)}\rangle =0$, $\langle B_i, \varphi\big|_{C\times (P,v)}\rangle =\omega_i(P,v)$, where $\omega_i(P,v)\in\mathbb{C}$ denotes the pairing between the holomorphic 1-form $\omega_i\in\Gamma(C,\Omega^1_C)$ introduced in \S\ref{def:omega} and $(P,v)$. We will therefore write $\int^{(1)}_{\mathcal{A}_i}\varphi=0$ and $\int^{(1)}_{\mathcal{B}_i}\varphi=\omega_i$, where the superscript $(1)$ means that we integrate in the first coordinate. Since~$\varphi$ is symmetric in its two arguments, one also has $\int^{(2)}_{\mathcal{A}_i}\varphi=0$ and $\int^{(2)}_{\mathcal{B}_i}\varphi=\omega_i$. 

Now for $r,s\in[\![1,n]\!]$ with $r\neq s$, let $\Delta_{rs}\subset C^n$ be the divisor associated with the diagonal $z_r=z_s$. 
For $r\in[\![1,n]\!]$, set $\Delta_r:=\sum_{s\neq r}\Delta_{rs}$  and let $\Delta:=\sum_{r<s}\Delta_{rs}$. 
Let~$\hat\Delta_{rs}$,~$\hat\Delta_r$ and~$\hat\Delta$ be the divisors of~$\hat C^n$, respectively defined as the
pull-backs of $\Delta_{rs}$, $\Delta_r$ and~$\Delta$ under the projection $\pi^n:\hat C^n\to C^n$. Let also $\Delta_{rs}(\hat C)$ be the divisor of~$\hat C^n$ associated 
with the diagonal $\hat z_r=\hat z_s$, so that $\hat \Delta_{rs}=\sum_{\theta\in\mathrm{Aut}(\hat C/C)}
\theta^{(r)}(\Delta_{rs}(\hat C))$, where $\theta^{(r)}$ denotes the automorphism of~$\hat C^n$ acting as~$\theta$ on the $r$-th component and as the identity on all the other ones.  

Consider the restriction $\varphi\big|_{(P,v)\times C}\in \Gamma(C,\Omega^1_{C}(2P))$. Since $\int_{\hat{\mathcal{A}}_i}\varphi\big|_{(P,v)\times C}={\rm res}_{\infty}\varphi\big|_{(P,v)\times C}=0$, then the assignment $(\hat{Q},\hat{R})\to\int^{\hat Q}_{\hat R}\varphi\big|_{(P,v)\times C}$ gives rise to an element of $\Gamma(\hat C^2,\mathcal{O}_{\hat C}^{\boxtimes 2}(\pi^{-1}(\pi(P))\times \hat C+\hat C\times\pi^{-1}(\pi(P))-\Delta(\hat C)))$, which we denote by $\int^\bullet_\bullet\varphi\big|_{(P,v)\times C}$. Moreover, the map $(P,v)\to \int^\bullet_\bullet\varphi\big|_{(P,v)\times C}$ is linear in~$v$ and analytic in~$(P,v)$, so it gives rise to an element in $\Gamma(\hat C^3,\Omega^1_{\hat C}\boxtimes\mathcal{O}_{\hat C}^{\boxtimes 2}(\hat \Delta_{12}+\hat \Delta_{13}-\Delta_{23}(\hat C)))$, which fixing the second and third coordinate is the pull-back of a meromorphic 1-form on~$C$. This element, which we denote by $\psi$ or by $(\int^\bullet_{\bullet})^{(2)}\varphi$, is called the \emph{fundamental form of the third kind}. Moreover, for any $P\in\hat C$ let
\begin{equation}
\psi_{P}\,:=\,\psi\big|_{\hat C^2\times P}\,=\,\Big(\int^\bullet_{P}\Big)^{(2)}\varphi\,\in\Gamma(\hat C^2,\Omega^1_{\hat C}\boxtimes\mathcal{O}_{\hat C}(\hat \Delta_{12}+\pi^{-1}(\pi(\hat{P}))\times\hat C-\hat C\times\hat P))\,.
\end{equation}
In the rest of the article, we will only be interested in $\psi_P$, especially in the case where $P=\hat\infty$, and we will call also such restrictions ``fundamental forms of the third kind''.

\subsubsection{Residues on the Schottky covering}\label{sec:resSchottky}

Let $\Delta(\hat C)$ denote the diagonal in $\hat C^2$, and let $\hat\Delta:=\pi^{-1}(\Delta)=\bigsqcup_{\theta\in\mathrm{Aut}(\hat C/C)}(\theta,1)(\Delta(\hat C))$. Pulling back by the map $\pi:\hat C\to C$ the first morphism of~\eqref{eq:210712n1}, which remains injective, and identifying $(\Omega^1_{\hat C})^{\boxtimes 2}\simeq \pi^*((\Omega^1_C)^{\boxtimes 2})$ and $(\Omega^1_{\hat C})^{\boxtimes 2}(2\hat\Delta)\simeq \pi^*((\Omega^1_C)^{\boxtimes 2}(2\Delta))$, one obtains a morphism
\begin{equation}\label{eq:210712n1'}
(\Omega^1_{\hat C})^{\boxtimes 2}(2\hat\Delta)\,\rightarrow\,(\Omega^1_{\hat C})^{\boxtimes 2}(2\hat\Delta)/ (\Omega^1_{\hat C})^{\boxtimes 2}\,.
\end{equation}
If we identify $\mathcal{O}_{\hat C^2}/\mathcal{O}_{\hat C^2}(-2\hat\Delta)\simeq (\pi^2)^*(\mathcal{O}_{C^2}/\mathcal{O}_{C^2}(-2\Delta))$, then the isomorphism~\eqref{eq:210712n2} gives rise to an isomorphism
\begin{equation}\label{eq:210712n2'}
\mathcal{O}_{\hat C^2}/\mathcal{O}_{\hat C^2}(-2\hat\Delta)\,\tilde{\rightarrow}\,(\Omega^1_{\hat C})^{\boxtimes 2}(2\hat\Delta)/(\Omega^1_{\hat C})^{\boxtimes 2}\,.
\end{equation}
A splitting of the exact sequence
$$
0\rightarrow \mathcal{O}_{\hat C^2}(-\hat\Delta)/\mathcal{O}_{\hat C^2}(-2\hat\Delta)\rightarrow \mathcal{O}_{\hat C^2}/\mathcal{O}_{\hat C^2}(-2\hat\Delta)\rightarrow \mathcal{O}_{\hat C^2}/\mathcal{O}_{\hat C^2}(-\hat\Delta)\rightarrow 0
$$
of sheaves of vector spaces over $\hat C^2$ can be constructed as follows. For $U\subset C$ open subset and $U_1,U_2$ lifts to the first and the second component of $\hat C^2$, let $A_U^i:
=\Gamma(U_i,\mathcal{O}_{\hat C})$ and $J_U^{12}:=\textrm{ker}(A_U^1 \otimes A_U^2 \to A_U)$. Then $\Gamma(U_1\times U_2,\mathcal{O}_{\hat C^2}/\mathcal{O}_{\hat C^2}(-2\hat\Delta)) \simeq (A_U^1
\otimes A_U^2)/(J_U^{12})^2$, while $\Gamma(U_1\times U_2,\mathcal{O}_{\hat C^2}/\mathcal{O}_{\hat C^2}(-\hat\Delta)) \simeq (A_U^1 \otimes A_U^2)/J_U \simeq A_U$. The map
$$
\Gamma(U_1\times U_2,\mathcal{O}_{\hat C^2}/\mathcal{O}_{\hat C^2}(-\hat\Delta)) \simeq A_U\, \to \,(A_U \otimes A_U)/J_U^2 \simeq \Gamma(U_1\times U_2,\mathcal{O}_{\hat C^2}/\mathcal{O}_{\hat C^2}(-2\hat\Delta))\,,
$$ 
where the middle map is $a\mapsto a \otimes 1$, is a section of the canonical map 
$$
\Gamma(U_1\times U_2,\mathcal{O}_{\hat C^2}/\mathcal{O}_{\hat C^2}(-2\hat\Delta))\to
\Gamma(U_1\times U_2,\mathcal{O}_{\hat C^2}/\mathcal{O}_{\hat C^2}(-\hat\Delta))\,.
$$
It follows that there is an isomorphism
\begin{equation}\label{eq:210712n3'}
\mathcal{O}_{\hat C^2}/\mathcal{O}_{\hat C^2}(-2\hat\Delta) \simeq \mathcal{O}_{\hat C^2}/\mathcal{O}_{\hat C^2}(-\hat\Delta) \oplus \mathcal{O}_{\hat C^2}(-\hat\Delta)/\mathcal{O}_{\hat C^2}(-2\hat\Delta)
\end{equation}
of sheaves of vector spaces over $\hat C^2$. 

Combining the second morphism in~\eqref{eq:210712n1'}, the inverse of~\eqref{eq:210712n2'} and the isomorphism~\eqref{eq:210712n3'}, one obtains a morphism of sheaves of vector spaces over $\hat C^2$
\begin{equation*}
(\Omega^1_{\hat C})^{\boxtimes 2}(2\hat\Delta)\,\to\,\mathcal{O}_{\hat C^2}/\mathcal{O}_{\hat C^2}(-\hat \Delta) \oplus \mathcal{O}_{\hat C^2}(-\hat \Delta)/\mathcal{O}_{\hat C^2}(-2\hat\Delta)
\end{equation*}
which, upon taking global sections and identifying\footnote{The notation $\Gamma(\ldots)^{\mathrm{Aut}(\hat C/C)}$ is to be intended in the set-theoretical sense. The map to the component of the target indexed by $\theta\in\mathrm{Aut}(\hat C/C)$ is induced by the restriction to the shifted diagonal $\{(P,\theta (P)), P \in \hat C\}$.} $\Gamma(\hat C^2,\mathcal{O}_{\hat C^2}/\mathcal{O}_{\hat C^2}(-\Delta))\simeq \Gamma(\hat C,\mathcal{O}_{\hat C})^{\mathrm{Aut}(\hat C/C)}$ and $\Gamma(\hat C^2,\mathcal{O}_{\hat C^2}(-\hat \Delta)/\mathcal{O}_{\hat C^2}(-2\hat \Delta))\simeq \Gamma(\hat C,(\Omega^1_{\hat C})^{\mathrm{Aut}(\hat C/C)}$, induces a linear map
$$
\widehat{\mathcal{R}\mathrm{es}}=\oplus_{\theta \in\mathrm{Aut}(\hat C/C)}\widehat{\mathcal{R}\mathrm{es}}^\theta : \Gamma(\hat C^2,(\Omega^1_{\hat C})^{\boxtimes 2}(2\hat\Delta))\,\to\, \Gamma(\hat C,\mathcal{O}_{\hat C})^{\mathrm{Aut}(\hat C/C)} \oplus \Gamma(\hat C,\Omega^1_{\hat C})^{\mathrm{Aut}(\hat C/C)}\,,
$$ 
whose first and second component will be denoted $\widehat{\mathcal{R}\mathrm{es}}_{-2}=\oplus_{\theta \in\mathrm{Aut}(\hat C/C)}\widehat{\mathcal{R}\mathrm{es}}_{-2}^\theta$ and $\widehat{\mathcal{R}\mathrm{es}}_{-1}=\oplus_{\theta \in\mathrm{Aut}(\hat C/C)}\widehat{\mathcal{R}\mathrm{es}}_{-1}^\theta$, respectively.

If $\alpha\in\Gamma(\hat C^2,(\Omega^1_{\hat C})^{\boxtimes 2}(2\hat\Delta))$, then in a neighborhood of any shifted diagonal $\{(P,\theta(P)), P \in \hat C\}$ there are holomorphic functions~$a^\theta_{-2}$ and~$a^\theta_{-1}$ such that, in local coordinates, $\alpha=a^{\theta}_{-2}(z)dzdw/(w-z)^2+a_{-1}^\theta (z)dzdw/(w-z)+O(1)dzdw$. One can check that $\widehat{\mathcal{R}\mathrm{es}}^\theta(\alpha)=(a^{\theta}_{-2}(z),a_{-1}^\theta (z)dz)$, and therefore $a^{\theta}_{-2}(z)=\widehat{\mathcal{R}\mathrm{es}}^\theta_{-2}(\alpha)$ and $a_{-1}^\theta (z)dz=\widehat{\mathcal{R}\mathrm{es}}^\theta_{-1}(\alpha)$ do not depend on the choice of local coordinates. We conclude with an explicitly coordinate-independent formula for $\widehat{\mathcal{R}\mathrm{es}}^\theta_{-1}$:
\begin{lem}\label{lem:210714n1}
Let $\mathcal{C}(P)$ denote a positive oriented small circle around a point~$P$. If $\alpha\in\Gamma(\hat C^2,(\Omega^1_{\hat C})^{\boxtimes 2}(2\hat \Delta))$, $\theta\in\mathrm{Aut}(\hat C/C)$ and $(P_0,v_0)\in T^{1,0}C$, then
\begin{equation*}
\widehat{\mathcal{R}\mathrm{es}}^\theta_{-1}(\alpha)(P_0,v_0)=\int_{\mathcal{C}(\theta(P_0))}\alpha\big|_{(P_0,v_0)\times C}\,.
\end{equation*}
\end{lem}
\begin{proof}
Let $\hat U$ be a neighborhood of $P_0$, and fix a local coordinate on $\hat U$ which vanishes at $P_0$. This induces a local coordinate $w$ on $\theta(\hat U)$, and we obtain a coordinate system $(z,w)$ on $\hat U\times \theta(\hat U)$ such that, locally, $\alpha=a^{\theta}_{-2}(z)dzdw/(w-z)^2+a_{-1}^\theta (z)dzdw/(w-z)+O(1)dzdw$. On the one hand we know that $\widehat{\mathcal{R}\mathrm{es}}^\theta_{-1}(\alpha)=a_{-1}^\theta (z)dz$, and therefore if $v_0=\partial/\partial z$ one finds $\widehat{\mathcal{R}\mathrm{es}}^\theta_{-1}(\alpha)(P_0,v_0)=a_{-1}^\theta(0)$. On the other hand, if $v_0=\partial/\partial w$ then $\alpha\big|_{(P_0,v_0)\times C}=a^{\theta}_{-2}(0)dw/w^2+a_{-1}^\theta (0)dw/w+O(1)dw$, and therefore its integral over $\mathcal{C}(\theta(P_0))$ is also equal to $a_{-1}^\theta(0)$, by Cauchy's theorem, thus proving our statement.
\end{proof}

\subsection{The $1$-form $\pmb K$}\label{sec:introofformK}

\subsubsection{The $1$-form $\pmb K^{[1]}$}

\begin{prop}[see\footnote{In \cite{EnrConfSp} the 1-form $\pmb K^{[1]}$ is denoted by $\alpha_1$.} \cite{EnrConfSp}, Prop. 11]\label{prop:412:2303}
There exists a unique
$$
\pmb K^{[1]}\in\Gamma(\hat C^n,\Omega^1_{\hat C}\boxtimes\mathcal O_{\hat C}^{\boxtimes n-1}(\hat\Delta_1))
\,\hat\otimes\,\hat{\mathfrak t}_{h,n}, 
$$ satisfying the following conditions:\footnote{Here and elsewhere, when we consider the action of $(\hat{\mathsf{B}}_i^{(r)})^*$ on $\pmb K^{[1]}$, as well as its residue and its integral, we are considering by abuse of notation $\pmb K^{[1]}$ as a $\hat{\mathfrak t}_{h,n}$-valued 1-form, rather than an element of $\Gamma((\hat C^n,\Omega^1_{\hat C}\boxtimes\mathcal O_{\hat C}^{\boxtimes n-1}(\hat\Delta_1))
\,\hat\otimes\,\hat{\mathfrak t}_{h,n}$.}\\
(a) For every $r\in[\![1,n]\!]$ and for every $i\in[\![1,h]\!]$
\begin{equation}\label{tag1:2302}
(\hat{\mathsf{B}}_i^{(r)})^*\,\pmb K^{[1]}\,=\,e^{\mathrm{ad}(b_i^{(r)})}\,(\pmb K^{[1]})\,.
\end{equation}
(b) For every $r\in[\![2,n]\!]$
\begin{equation}\label{tag2:2302}
\mathrm{res}_{\Delta_{1r}(\hat C)}(\pmb K^{[1]})\,=\,{1\over{2\pi\mathrm{i}}}\,t_{1r}\,.
\end{equation}
(c) If\footnote{The statement of Prop.~11 should be corrected by the inclusion of this condition.} 
$\textbf{P}\in\hat D^{n-1}\cap \mathrm{Cf}_{n-1}(\hat C)$, for every $i\in[\![1,h]\!]$
\begin{equation}\label{tag3:2302}
\int_{\hat{\mathcal A}_i}\pmb K^{[1]}\big|_{\hat C\times \{\pmb P\}}\,=\,{{\mathrm{ad}(b_i^{(1)})}\over{e^{\mathrm{ad}(b_i^{(1)})}-1}}(a_i^{(1)})\,.
\end{equation}
\end{prop}
\begin{rem}
Using formula~\eqref{residue:formula}, one can show that the conditions (a) and (b) (or (a) and (c)) are sufficient to uniquely determine $\pmb K^{[1]}$.
\end{rem}
 
\subsubsection{The $1$-forms $\pmb K^{[r]}$ and $\pmb K$}

The symmetric group~$S_n$ acts by permutation on~$\hat C^n$, and this action lifts to an action on the bundle  
$\Omega^1_{\hat C^n}$. Consider the decomposition $\Omega^1_{\hat C^n}=\oplus_{r=1}^n(\Omega^1_{\hat C})^{[r]}$, 
where $(\Omega^1_{\hat C})^{[r]}:=\mathcal O_{\hat C}^{\boxtimes r-1}\boxtimes
\Omega^1_{\hat C}\boxtimes\mathcal O_{\hat C}^{\boxtimes n-r}$ , then any $\sigma\in S_n$ sets up an isomorphism 
$(\Omega^1_{\hat C})^{[r]}\to (\Omega^1_{\hat C})^{[\sigma(r)]}$. This decomposition extends to a diagram of 
bundles $\oplus_{r=1}^n(\Omega^1_{\hat C})^{[r]}(\hat\Delta_r)\hookrightarrow
\Omega^1_{\hat C^n}(\hat\Delta)$. 
%, with $\hat\Delta_r$ as in~\eqref{eqDeltai} and
%$$
%\hat\Delta:=\sum_{\substack{1\leq r<s\leq n\\ \theta\in
%\mathrm{Aut}(\hat C/C)}}\theta^{(r)}(\hat\Delta_{rs}).
%$$ 
The 
action of~$S_n$ extends to  
$\Omega^1_{\hat C^n}(\hat\Delta)$ as well as to isomorphisms 
$(\Omega^1_{\hat C})^{[r]}(\hat\Delta_r)\to (\Omega^1_{\hat C})^{[\sigma(r)]}(\hat\Delta_{\sigma(r)})$. 

The group~$S_n$ also acts by Lie algebra automorphisms of $\hat{\mathfrak t}_{h,n}$ (permutation of the indices in $[\![1,n]\!]$). 
It follows that~$S_n$ acts on both sides of the inclusion 
$$
\oplus_{r=1}^n\Gamma_r\hookrightarrow
\Gamma(\hat C^n,\Omega^1_{\hat C^n}(\hat\Delta))\,\hat\otimes\,\hat{\mathfrak t}_{h,n}, 
$$
where $\Gamma_r:=\Gamma(\hat C^n,(\Omega^1_{\hat C})^{[r]}(\hat\Delta_r))\,\hat\otimes\,
\hat{\mathfrak t}_{h,n}$, and that $\sigma\in S_n$ induces an isomorphism $\Gamma_r\to\Gamma_{\sigma(r)}$. 

View~$S_{n-1}$ as the subgroup of~$S_n$ of all permutations which leave $1\in[\![1,n]\!]$ fixed. Then~$S_{n-1}$ acts on~$\Gamma_1$. We recall that $\pmb K^{[1]}\in\Gamma_1$.

\begin{lem}\label{Omega1:invt}
$\pmb K^{[1]}$ is~$S_{n-1}$-invariant.  
\end{lem}

\proof This follows from the invariance of the conditions defining~$\pmb K^{[1]}$. \hfill\qed\medskip

\begin{defn}\label{defn:K}
For $r\in[\![2,n]\!]$, we define $\pmb K^{[r]}:=(1r)\,\pmb K^{[1]}\in\Gamma_r$, and we set
$$
\pmb K:=\sum_{r=1}^n \pmb K^{[r]}\in
\Gamma(\hat C^n,\Omega^1_{\hat C^n}(\hat\Delta))\,\hat\otimes\,\hat{\mathfrak t}_{h,n}.
$$ 
\end{defn}

\begin{lem}
$\pmb K$ is~$S_n$-invariant.
\end{lem}

\proof This follows from Lemma~\ref{Omega1:invt}. \hfill\qed\medskip 

\begin{prop}\label{propospropertiesK}
$\pmb K$ has the following properties: 

(a) $\forall r\in[\![1,n]\!],\,\forall i\in[\![1,h]\!]$, $(\hat{\mathsf{B}}_i^{(r)})^*\pmb K=e^{\mathrm{ad}(b_i^{(r)})}(\pmb K)$, 

(b) $\forall r<s\in[\![1,n]\!]$, $\mathrm{res}_{\Delta_{rs}(\hat C)}(\pmb K)={1\over{2\pi\mathrm{i}}}t_{rs}$, 

(c) $d\pmb K=\pmb K\wedge \pmb K=0$. 
\end{prop}

\proof (a) and~(b) are direct consequences of Proposition~\ref{prop:412:2303}, while~(c) is contained in~\cite{EnrConfSp},~\S6. 
\hfill\qed\medskip  

As a consequence, $d+\pmb K$ is a regular singular meromorphic flat connection on the trivial 
$\mathrm{exp}(\hat{\mathfrak t}_{h,n})$-bundle on~$\hat C^n$, and it is the lift of a regular singular meromorphic flat connection~$\nabla_\mathcal{P}$ on the 
$\mathrm{exp}(\hat{\mathfrak t}_{h,n})$-bundle~$\mathcal{P}$ over~$C^n$ defined in \S\ref{sssec:consP}.

\subsection{Construction of $g_{\tilde x,\pmb\beta}$ and 
$\pmb g_{\tilde x,\pmb\beta}$}\label{sec:finalcorrection}

In \S\ref{ssec:defgtildexbeta} we have constructed a map $\tilde g_{\tilde x,\pmb\beta}:\tilde C\setminus p^{-1}(\infty)\to \exp({\rm Lie}(b_1,\ldots ,b_h)^\wedge)$ such that, for any $i\in[\![1,h]\!]$, 
\begin{align}
\mathsf A_i^*\,\tilde g_{\tilde x,\pmb\beta}&\,=\,\tilde g_{\tilde x,\pmb\beta}\,,\label{eq:31032021A}\\
\mathsf B_i^*\,\tilde g_{\tilde x,\pmb\beta}&\,=\,\tilde g_{\tilde x,\pmb\beta}\,e^{-b_i}. \label{eq:31032021B}
\end{align}

\begin{lem}\label{lemma:hatg:06042021}
There exists a unique map $g_{\tilde x,\pmb\beta}:\hat C\setminus\pi^{-1}(\infty)\to\mathrm{exp}(\mathrm{Lie}(b_1,
\ldots,b_h)^\wedge)$, such that $\hat p^*g_{\tilde x,\pmb\beta}=\tilde g_{\tilde x,\pmb\beta}$. 
\end{lem}

\proof This follows from~\eqref{eq:31032021A} and from the identification of~$\hat C$ with the quotient of~$\tilde C$ by the normal subgroup of $\mathrm{Aut}(\tilde C/C)$ generated by the~$\mathsf A_i$, $i\in[\![1,h]\!]$.  
\hfill\qed\medskip 

\begin{defn}
We define $\pmb g_{\tilde x,\pmb\beta} : (\hat C\setminus \pi^{-1}(\infty))^n\to\mathrm{exp}(\hat{\mathfrak t}_{h,n})$
by setting
\begin{equation*}
\pmb g_{\tilde x,\pmb\beta}\,:=\,
g^{\{1\}}_{\tilde x,\pmb\beta}\cdots g^{\{n\}}_{\tilde x,\pmb\beta}, 
\end{equation*}
where the notation $\bullet^{\{r\}}$ was introduced in \S\ref{sssec:trivP}.
\end{defn}

\begin{lem}\label{lemmamonodromycorrection}
For any $r\in[\![1,n]\!]$ and $i\in[\![1,h]\!]$, one has 
\begin{equation*}
(\hat{\mathsf B}_i^{(r)})^*\,\pmb g_{\tilde x,\pmb\beta}
\,=\,\pmb g_{\tilde x,\pmb\beta}\,e^{-b_i^{(r)}}\,. 
\end{equation*}
\end{lem}

\proof One has %\textcolor{red}{FZ: modified}  
\begin{align*}
(\hat{\mathsf B}_i^{(r)})^*\,\pmb g_{\tilde x,\pmb\beta}
&\,=\,g^{\{1\}}_{\tilde x,\pmb\beta}
\cdots
{\rm pr}_r^*\,((r)\,\circ\,\hat{\mathsf B}_i^*\,g_{\tilde x,\pmb\beta})
\cdots g^{\{n\}}_{\tilde x,\pmb\beta}
\\ & 
\,=\,g^{\{1\}}_{\tilde x,\pmb\beta}
\cdots
{\rm pr}_r^*\,((r)\,\circ \,g_{\tilde x,\pmb\beta}\,e^{-b_i})
\cdots g^{\{n\}}_{\tilde x,\pmb\beta}
\\ & \,=\,g^{\{1\}}_{\tilde x,\pmb\beta}
\cdots g^{\{r\}}_{\tilde x,\pmb\beta}\,e^{-b_i^{(r)}}
\cdots g^{\{n\}}_{\tilde x,\pmb\beta}
\\ & 
\,=\,g^{\{1\}}_{\tilde x,\pmb\beta}
\cdots g^{\{r\}}_{\tilde x,\pmb\beta}
\cdots g^{\{n\}}_{\tilde x,\pmb\beta}\,e^{-b_i^{(r)}}
\,=\,\pmb g_{\tilde x,\pmb\beta}
\,e^{-b_i^{(r)}}
\end{align*}
where the second equality follows from Lemma~\ref{lemma:hatg:06042021} and~\eqref{eq:31032021B}, and the fourth equality follows from the fact that~$b_i^{(r)}$ commutes with the image of~$(s)$ for any $s\neq r$. 
\hfill\qed\medskip 

\subsection{The connection $d+\pmb J_{\tilde x,\pmb\beta}$}\label{ssec:connecJ}

We define\footnote{\label{footnote210506}For any $\xi\in \exp(\hat{\mathfrak t}_{h,n})$, let~${\rm Ad}_\xi$ denote the image of~$\xi$ via the group morphism ${\rm Ad}:\exp(\hat{\mathfrak t}_{h,n})\to {\rm GL}(\hat{\mathfrak t}_{h,n})$ induced by the adjoint action of $\exp(\hat{\mathfrak t}_{h,n})$ on~$\hat{\mathfrak t}_{h,n}$, so that ${\rm Ad}_\xi(X)=\xi X\xi^{-1}$ for any $X\in \hat{\mathfrak t}_{h,n}$.}
\begin{align*}
\pmb I_{\tilde x,\pmb\beta}\,&:=\,\pmb g_{\tilde x,\pmb\beta}
\,d(\pmb g_{\tilde x,\pmb\beta})^{-1}\,,\\
\pmb H_{\tilde x,\pmb\beta}\,&:=\,{\rm Ad}_{\pmb g_{\tilde x,\pmb\beta}}\pmb K\,.
\end{align*}

The fact that $\pmb g_{\tilde x,\pmb\beta}$ is a regular map $(\hat C\setminus\pi^{-1}(\infty))^n\to
\mathrm{exp}(\hat{\mathfrak t}_{h,n})$ implies that $\pmb I_{\tilde x,\pmb\beta}$ belongs to 
$\Gamma((\hat C\setminus\pi^{-1}(\infty))^n,
\Omega^1_{\hat C^n})\,\hat\otimes\,\hat{\mathfrak t}_{h,n}$, while the fact that $\pmb K$ belongs to 
$\Gamma(\hat C^n,\Omega^1_{\hat C^n}(\hat\Delta))\,\hat\otimes\,\hat{\mathfrak t}_{h,n}$ implies that 
$\pmb H_{\tilde x,\pmb\beta}$ belongs to $\Gamma((\hat C\setminus\pi^{-1}(\infty))^n,
\Omega^1_{\hat C^n}(\hat\Delta))\,\hat\otimes\,\hat{\mathfrak t}_{h,n}$.  The lemma below shows that stronger statements hold true. 
%Let $\Delta$ be the divisor of $(C\setminus\infty)^n$ of equations $z_r=z_s$ for any $1\leq r<s\leq n$, so that $\hat \Delta=(\pi^n)^{-1}(\Delta)$.
\begin{lem}\label{lemmamonIH}
(a) $\pmb I_{\tilde x,\pmb\beta}$ belongs\footnote{This is a slight abuse of notation, justified by the existence of a natural injection from $\Gamma((C\setminus\infty)^n,
\Omega^1_{C^n})$ into $\Gamma((\hat C\setminus\pi^{-1}(\infty))^n,
\Omega^1_{\hat C^n})$.} to $\Gamma((C\setminus\infty)^n,
\Omega^1_{C^n})\,\hat\otimes\,\hat{\mathfrak t}_{h,n}$. 

(b) $\pmb H_{\tilde x,\pmb\beta}$ belongs\footnote{This is a slight abuse of notation, justified by the existence of a natural injection from $\Gamma((C\setminus\infty)^n,\Omega^1_{C^n}(\Delta))$ into $\Gamma((\hat C\setminus\infty)^n,\Omega^1_{\hat C^n}(\hat\Delta))$.} to $\Gamma((C\setminus\infty)^n,\Omega^1_{C^n}(\Delta))\,\hat\otimes\,\hat{\mathfrak t}_{h,n}$. 
\end{lem}
\begin{proof}
We need to prove that $(\hat{\mathsf B}_i^{(r)})^*\pmb I_{\tilde x,\pmb\beta}=\pmb I_{\tilde x,\pmb\beta}$ and that 
$(\hat{\mathsf B}_i^{(r)})^*\pmb H_{\tilde x,\pmb\beta}=\pmb H_{\tilde x,\pmb\beta}$ for any $r\in[\![1,n]\!]$ and any 
$i\in[\![1,h]\!]$. For the first statement it is enough to show that $\pmb g_{\tilde x,\pmb\beta}^{-1}(\hat{\mathsf B}_i^{(r)})^*\pmb g_{\tilde x,\pmb\beta}$ is constant, which follows from Lemma~\ref{lemmamonodromycorrection}. The second statement follows by combining Lemma~\ref{lemmamonodromycorrection} with part (a) of Proposition~\ref{propospropertiesK}.
\end{proof}

\begin{defn} We set
\begin{equation*}
\pmb J_{\tilde x,\pmb\beta}\,:=\,\pmb I_{\tilde x,\pmb\beta}\,+\,\pmb H_{\tilde x,\pmb\beta}\,
\in\,\Gamma((C\setminus\infty)^n,\Omega^1_{C^n}(\Delta))\,\hat\otimes\,\hat{\mathfrak t}_{h,n}. 
\end{equation*}
\end{defn}
In particular, 
\begin{equation*}
d\,+\,\pmb J_{\tilde x,\pmb\beta}\,=\,\pmb g_{\tilde x,\pmb\beta}\,(d+\pmb K)\,(\pmb g_{\tilde x,\pmb\beta})^{-1}\,.
\end{equation*}
The previous lemma implies the following:
\begin{thm}\label{thmJflat}
$d+\pmb J_{\tilde x,\pmb\beta}$ is a flat connection on the trivial 
$\mathrm{exp}(\hat{\mathfrak t}_{h,n})$-bundle over~$(C\setminus\infty)^n$.
\end{thm}

\begin{defn}\label{def:210611}
For $r\in[\![1,n]\!]$, set 
\begin{align*}
\pmb I_{\tilde x,\pmb \beta}^{[r]}&\,:=\,g^{\{r\}}_{\tilde x,\pmb\beta}\,d((g^{\{r\}}_{\tilde x,\pmb\beta})^{-1})\,\in\,\Gamma((\hat C\setminus\pi^{-1}(\infty))^n,(\Omega^1_{\hat C})^{[r]})\,\hat\otimes\,\hat{\mathfrak t}_{h,n}\,,
\\
\pmb H_{\tilde x,\pmb \beta}^{[r]}&\,:=\,\mathrm{Ad}_{\pmb g_{\tilde x,\pmb\beta}}\,(\pmb K^{[r]})\,\in\,\Gamma((\hat C\setminus\pi^{-1}(\infty))^n,(\Omega^1_{\hat C})^{[r]}(\hat\Delta_r))\,\hat\otimes\,\hat{\mathfrak t}_{h,n}\,,
\\
\pmb J_{\tilde x,\pmb \beta}^{[r]}&\,:=\,\pmb I_{\tilde x,\pmb \beta}^{[r]}\,+\,\pmb H_{\tilde x,\pmb \beta}^{[r]}\,\in\,\Gamma((\hat C\setminus\pi^{-1}(\infty))^n,(\Omega^1_{\hat C})^{[r]}(\hat\Delta_r))\,\hat\otimes\,\hat{\mathfrak t}_{h,n}\,.
\end{align*} 
\end{defn}

\begin{lem}\label{lemma:decomp}
(a) For $r\in[\![1,n]\!]$,~$\pmb I_{\tilde x,\pmb \beta}^{[r]}$ belongs to 
$\Gamma((C\setminus\infty)^n,(\Omega^1_C)^{[r]})\,\hat\otimes\,\hat{\mathfrak t}_{h,n}$
while~$\pmb H_{\tilde x,\pmb \beta}^{[r]}$ and~$\pmb J_{\tilde x,\pmb \beta}^{[r]}$ belong to 
$\Gamma((C\setminus\infty)^n,(\Omega^1_C)^{[r]}(\Delta_r))\,\hat\otimes\,\hat{\mathfrak t}_{h,n}$.

(b) One has $\pmb I_{\tilde x,\pmb \beta}=\sum_{r=1}^n \pmb I_{\tilde x,\pmb \beta}^{[r]}$, $\pmb H_{\tilde x,\pmb \beta}=\sum_{r=1}^n \pmb H_{\tilde x,\pmb \beta}^{[r]}$ and 
$\pmb J_{\tilde x,\pmb \beta}=\sum_{r=1}^n \pmb J_{\tilde x,\pmb \beta}^{[r]}$. 
\end{lem}

\begin{proof}
The proof of (a) is the same as that of Lemma~\ref{lemmamonIH}. 

The only non-trivial part of (b) is the identity $\pmb I_{\tilde x,\pmb \beta}=\sum_{r=1}^n \pmb I_{\tilde x,\pmb \beta}^{[r]}$. By definition, we have
\begin{align}\label{eq210425}
\pmb I_{\tilde x,\pmb \beta}&\,=\,\pmb g_{\tilde x,\pmb \beta}{\rm d}(\pmb g_{\tilde x,\pmb \beta})^{-1}\notag\\
&\,=\,\sum_{r=1}^n g^{\{1\}}_{\tilde x,\pmb \beta}\cdots g^{\{n\}}_{\tilde x,\pmb \beta}(g^{\{n\}}_{\tilde x,\pmb \beta})^{-1}\cdots {\rm d}(g^{\{r\}}_{\tilde x,\pmb \beta})^{-1}\cdots (g^{\{1\}}_{\tilde x,\pmb \beta})^{-1}.
\end{align}
Recall that $g^{\{r\}}_{\tilde x,\pmb \beta}$ takes values in $\exp({\rm Lie}(b_1^{(r)},\ldots ,b_h^{(r)})^\wedge)\leq\exp(\hat{\mathfrak t}_{h,n})$. This, together with the relations $[b_i^{(r)},b_j^{(s)}]=0$ for any $i,j\in[\![1,h]\!]$ and any $r,s\in[\![1,n]\!]$ with $r\neq s$, implies that $g^{\{r\}}_{\tilde x,\pmb \beta}g^{\{s\}}_{\tilde x,\pmb \beta}=g^{\{s\}}_{\tilde x,\pmb \beta}g^{\{r\}}_{\tilde x,\pmb \beta}$ for any $r,s\in[\![1,n]\!]$. The statement follows by combining this fact with~\eqref{eq210425}.
\end{proof}

Notice that part~(b) of Lemma~\ref{lemma:decomp} corresponds to the fact that there is a direct sum decomposition $\Gamma((C\setminus\infty)^n,\Omega^1_{C^n}(\Delta))=
\oplus_{r=1}^n\Gamma((C\setminus\infty)^n,(\Omega^1_C)^{[r]}(\Delta_r))$, which induces a direct sum 
decomposition of $\Gamma((C\setminus\infty)^n,\Omega^1_{C^n}(\Delta))\,\hat\otimes\,\hat{\mathfrak t}_{h,n}$. 

\subsection{Dependence of ${\pmb J}_{\tilde x,\pmb \beta}$ in $\tilde x$ and $\pmb \beta$}\label{sec:notationchange}

\begin{prop}
For any $\tilde x,\tilde y\in \tilde C\setminus p^{-1}(\infty)$ and any two families ${\pmb \beta}, {\pmb \beta}'$ as in \S\ref{def:omega}, the Maurer-Cartan elements~${\pmb J}_{\tilde x,\pmb \beta}$ and~${\pmb J}_{\tilde y,\pmb \beta'}$ are related by gauge conjugation by a holomorphic function $\pmb h:(C\setminus\infty)^n\to\exp(\hat{\mathfrak{t}}_{h,n})$. If $\pmb \beta=\pmb \beta'$ then this function is constant.
\end{prop}

\begin{proof}
By Proposition~\ref{prop:211013n2} there exists $h: C\setminus\infty\to \exp(\mathrm{Lie}(b_1,\ldots ,b_h)^\wedge)$ such that\footnote{Throughout this proof we write $h$ to denote also its pull-back $\pi^*h$.} $g_{\tilde x,\pmb \beta}=h\cdot g_{\tilde y,\pmb \beta'}$. Let $\pmb h:=h^{\{1\}}\cdots h^{\{n\}}$. Since 
$g_{y,\pmb \beta'}$ takes values in $\exp(\mathrm{Lie}(b_1,\ldots ,b_h)^\wedge)$ and since $[b_i^{(r)},b_j^{(s)}]=0$ for $r\neq s$, then 
$\pmb g_{\tilde x,\pmb \beta}=\pmb h\cdot \pmb g_{\tilde y,\pmb \beta'}$. Then 
\begin{align*}
d+{\pmb J}_{\tilde x,\pmb \beta}&\,=\,{\rm Ad}_{\pmb h\cdot \pmb g_{\tilde y,\pmb \beta'}}(d+\pmb K)={\rm Ad}_{\pmb h} \circ {\rm Ad}_{\pmb g_{\tilde y,\pmb \beta'}}(d+\pmb K)\\
&\,=\,{\rm Ad}_{\pmb h}(d+{\pmb J}_{\tilde y,\pmb \beta'})\,.
\end{align*}
If moreover $\pmb \beta=\pmb \beta'$, then it follows from Proposition~\ref{prop:211013} that~$\pmb h$ is constant.
\end{proof}

\section{Decompositions of $\pmb K$ and ${\pmb J}_{\tilde x,\pmb \beta}$}\label{sec4}

In \S\ref{ssec:41} and \S\ref{ssec:42} we introduce spaces of multi-valued (i.e. defined on~$C_n(\hat C)$) and single-valued (i.e. defined on~$C_n(C)$) forms, respectively, and we explain how in some cases these forms can be decomposed in terms of forms defined only on one or two copies of the curve. In \S\ref{ssec:43} we explain how to map such multi-valued forms to such single-valued forms using the function~$g_{\tilde x,\beta}$. In \S\ref{ssec:44} we make use of results from~\cite{EnrConfSp} to show that~$\pmb K^{[1]}$, which is the main constituent of $\pmb K$ (see Definition~\ref{defn:K}) and which belongs to the space of multi-valued forms, admits a decomposition in the sense of \S\ref{ssec:41}. From this we infer in \S\ref{ssec:45} a decomposition for~$\pmb H_{\tilde x,\beta}^{[1]}$, which is the main constituent of ${\pmb J}_{\tilde x,\pmb \beta}$ (see Definition~\ref{def:210611}), and which belongs to the space of single-valued forms.

Throughout this section we consider~$\tilde x$ and~$\pmb \beta$ as fixed, so we will in general drop the subscript $(\tilde x,\pmb \beta)$, e.g. we will write~$\pmb J$ and~$g$ instead of~$\pmb J_{\tilde x,\beta}$ and~$g_{\tilde x,\beta}$.

\subsection{Decompositions of multi-valued forms}\label{ssec:41}

For~$\Sigma$ a Riemann surface, let~$\Omega^1_{\Sigma,\log}$ be the sheaf of differentials over~$\Sigma$ whose space of sections on an 
open subset~$U$ is the space of meromorphic 1-forms on~$U$ with at most simple poles.   
\begin{defn}
(a) Let $\mathcal Forms$ be the vector subspace of $\Gamma(\hat C,\Omega^1_{\hat C,\log})\,\hat\otimes\,
U(\mathrm{Lie}(b_1,\ldots ,b_h))^\wedge$
%\textcolor{red}{FZ: shouldn't we allow simple poles?} 
spanned by elements~$\delta$ such that  $\hat{\mathsf B}_i^*\delta=e^{b_i}\delta$ for any $i\in[\![1,h]\!]$.
% \textcolor{red}{FZ: why don't we write $\kappa$ instead $\delta$?}

(b) Let $\mathcal Tuples$ be the complex vector space spanned by tuples $(\kappa_1,\ldots,\kappa_h,\underline\kappa)$, constituted by $\kappa_1,\ldots \kappa_h\in\mathcal Forms$
and by $\underline\kappa\in\Gamma(\hat C^2,(\Omega^1_{\hat C,\log})^{[1]}(\hat\Delta))\,\hat\otimes \,U(\mathrm{Lie}(b_1,\ldots ,b_h))^\wedge$ 
%\textcolor{red}{FZ: + simple poles in 1st variable} 
such that, for any $i\in[\![1,h]\!]$, one has\footnote{Equation \eqref{eq280421mon2} contains a slight abuse of notation: $\kappa_i$ stands for ${\rm pr}_1^*\kappa_i$.}
\begin{align}
(\hat{\mathsf B}_i^{(1)})^*\,\underline\kappa&\,=\,e^{b_i}\,\underline\kappa, \label{eq280421mon1}\\
(\hat{\mathsf B}_i^{(2)})^*\,\underline\kappa&\,=\,\underline\kappa\,e^{-b_i}\,+\,\kappa_i\,{1-e^{-b_i}\over b_i}.\label{eq280421mon2} 
\end{align}

(c) Let $\pmb{\mathcal Forms}$ be the vector subspace of $\Gamma(\hat C^n,
(\Omega^1_{\hat C,\log})^{[1]}%\Omega^1_{\hat C^n}
(\hat\Delta_1))\,\hat\otimes\,\hat{\mathfrak t}_{h,n}$ spanned by elements~$\pmb\kappa$ such that $(\hat{\mathsf B}_i^{(r)})^*\pmb\kappa=e^{{\rm ad}(b_i^{(r)})}(\pmb\kappa)$
for any $r\in[\![1,n]\!]$ and any $i\in[\![1,h]\!]$.

\end{defn}

We refer to the differential forms which belong to these spaces as \emph{multi-valued}, meaning that they are not the pull-backs of differential forms defined on the curve~$C$ (or its powers), i.e. they are not invariant under the action of ${\rm Aut}(\hat C/C)$ (or its powers). Note that, because of Proposition~\ref{prop:412:2303}, $\pmb K^{[1]}\in\pmb{\mathcal Forms}$. 

Recall from \S\ref{sssec:trivP} that, if any function or 1-form~$\phi$ is $U(\mathrm{Lie}(b_1,\ldots,b_h))^\wedge$-valued and if $r\in[\![1,n]\!]$, we have introduced the notation $\phi^{(r)}:=(r)\circ \phi$. If $\phi$ is defined on $\Sigma^2$, and if ${\rm pr}_{rs}:\Sigma^n\to \Sigma^2$ denotes the projection on product of the $r$-th and the $s$-th component, we set
\begin{equation*}
\phi^{\{rs\}}\,:=\,{\rm pr}_{rs}^*\,\phi^{(r)}\,.
\end{equation*}

\begin{lem}\label{lemma:210428multi}
(a) There is a linear map $\mathcal Tuples\to\pmb{\mathcal Forms}$, defined 
by\footnote{We denote by $x\mapsto\mathrm{ad}(x)$ the degree completion of the algebra morphism $U(\mathfrak t_{h,n})\to
\mathrm{End}(\mathfrak t_{h,n})$ attached to the adjoint action of $\mathfrak t_{h,n}$. This morphism restricts to the group morphism ${\rm Ad}:\exp(\hat{\mathfrak t}_{h,n})\to {\rm GL}(\hat{\mathfrak t}_{h,n})$ defined in the footnote~\ref{footnote210506}.}
\begin{equation}\label{eq280421def1}
(\kappa_1,\ldots,\kappa_h,\underline\kappa)\,\mapsto \,\pmb\kappa\,:=\,\sum_{i=1}^h
\mathrm{ad}(\kappa_i^{\{1\}})(a_i^{(1)})
\,+\,\sum_{r=2}^n\mathrm{ad}(\underline\kappa^{\{1r\}})(t_{1r})\,,
\end{equation}
the meaning of the superscript $\{r\}$ being as in~\eqref{defgraffe}.
 
(b) There is an injective linear map $\mathcal Forms\to\mathcal Tuples$, defined by 
$\delta\mapsto (\delta b_1,\ldots,\delta b_h,\mathrm{pr}_1^*\delta)$. 

(c) The composed map $\mathcal Forms\to\mathcal Tuples\to\pmb{\mathcal Forms}$ is zero.  
\end{lem}

\begin{proof} (a) We need to prove that $(\hat{\mathsf B}_j^{(s)})^*\pmb\kappa=e^{{\rm ad}(b_j^{(s)})}(\pmb\kappa)$ for any $s\in[\![1,n]\!]$ and any $j\in[\![1,h]\!]$. If $s=1$ the statement immediately follows from~\eqref{eq280421mon1}. Suppose that $s\neq 1$, then 
\begin{align*}
(\hat{\mathsf B}_j^{(s)})^*\,\pmb\kappa&\,=\,\sum_{i=1}^h
\mathrm{ad}(\kappa_i^{\{1\}})(a_i^{(1)})
+\sum_{\substack{r\in[\![1,n]\!]\\r\neq s}}\mathrm{ad}(\underline\kappa^{\{1r\}})(t_{1r})+\mathrm{ad}\bigg(\underline\kappa^{\{1s\}}e^{-b_j^{(1)}}+\kappa_j^{\{1\}}\,{1-e^{-b_j^{(1)}}\over b_j^{(1)}}\bigg)(t_{1s})\\
&\,=\,\sum_{i=1}^h
\mathrm{ad}(\kappa_i^{\{1\}})(a_i^{(1)})
+\sum_{\substack{r\in[\![1,n]\!]\\r\neq s}}\mathrm{ad}(\underline\kappa^{\{1r\}})(t_{1r})+\mathrm{ad}\bigg(\underline\kappa^{\{1s\}}e^{b_j^{(s)}}+\kappa_j^{\{1\}}\,\frac{e^{b_j^{(s)}}-1}{b_j^{(s)}}\bigg)(t_{1s})\\
&\,=\,\pmb\kappa \,+\, \mathrm{ad}\big(\underline\kappa^{\{1s\}}(e^{b_j^{(s)}}-1)\big)(t_{1s})\,+\,\mathrm{ad}\big(\kappa_j^{\{1\}}(e^{b_j^{(s)}}-1)\big)(a_j^{(1)})\\
&\,=\,\pmb\kappa \,+\, (e^{{\rm ad}(b_j^{(s)})}-1)\,\big(\mathrm{ad}(\underline\kappa^{\{1s\}})(t_{1s})\,+\,\mathrm{ad}(\kappa_j^{\{1\}})(a_j^{(1)})\big)\\
&\,=\,\pmb\kappa +(e^{{\rm ad}(b_j^{(s)})}-1)\,\pmb\kappa\,=\,e^{{\rm ad}(b_j^{(s)})}\pmb\kappa\,,
\end{align*}
where the first identity follows from~\eqref{eq280421mon2} and \eqref{eq280421def1}, the second identity follows from 
$[b_j^{(1)},t_{1s}]=-[b_j^{(s)},t_{1s}]$, the third identity follows from $[b_j^{(s)},a_j^{(1)}]=t_{1s}$, the fourth identity follows from $[b_i^{(1)},b_j^{(s)}]=0$ for any~$i$,
%(so that $\mathrm{ad}(\underline{\kappa}^{\{1s\}}e^{b_j^{(s)}})={\rm ad}(\underline{\kappa}^{\{1s\}})e^{{\rm ad}(b_j^{(s)})}=e^{{\rm ad}(b_j^{(s)})}{\rm ad}(\underline{\kappa}^{\{1s\}})$, and similarly $\mathrm{ad}(\kappa_j^{\{1\}}e^{b_j^{(s)}})=e^{{\rm ad}(b_j^{(s)})}\mathrm{ad}(\kappa_j^{\{1\}})$)
and the fifth identity follows from $[b_j^{(s)},a_i^{(1)}]=0$ for $i\neq j$ (together with $[b_i^{(s)},b_j^{(1)}]=0$, this implies that $(e^{{\rm ad}(b_j^{(s)})}-1)\mathrm{ad}(\kappa_i^{\{1\}})(a_i^{(1)})=0$) and from $[b_j^{(s)},t_{1r}]=0$ for $r\neq s$ (together with $[b_i^{(s)},b_j^{(1)}]=0$, this implies that $(e^{{\rm ad}(b_j^{(s)})}-1)\mathrm{ad}(\underline{\kappa}^{\{1r\}})(t_{1r})=0$).

(b) Since $\delta\in\mathcal Forms$, also $\delta b_i\in\mathcal Forms$ for any~$i\in[\![1,h]\!]$. For the same reason, it is obvious that $(\hat{\mathsf B}_i^{(1)})^*{\rm pr}_1^*\delta=e^{b_i}{\rm pr}_1^*\delta$. We need to verify that  ${\rm pr}_1^*\delta$ satisfies the condition~\eqref{eq280421mon2}. Because $(\hat{\mathsf{B}}_i^{(2)})^*{\rm pr}_1^*\delta={\rm pr}_1^*\delta$, we can write
$(\hat{\mathsf{B}}_i^{(2)})^*{\rm pr}_1^*\delta-{\rm pr}_1^*\delta e^{-b_i}={\rm pr}_1^*\delta (1-e^{-b_j})$. The right hand side can be rewritten as $({\rm pr}_1^*\delta b_j)\tfrac{1-e^{-b_j}}{b_j}$, thus concluding the proof that the image of the map is contained in $\mathcal Tuples$. The injectivity is obvious.

(c) The image of~$\delta$ by the composed map is 
$\sum_{i=1}^h\mathrm{ad}((\delta b_i)^{\{1\}})(a_i^{(1)})
+\sum_{r=2}^n\mathrm{ad}((\mathrm{pr}_1^*\delta)^{\{1r\}})(t_{1r})$. 
One has $(\delta b_i)^{\{1\}}=\delta^{\{1\}}b_i^{(1)}$ and 
$(\mathrm{pr}_1^*\delta)^{\{1r\}}=\delta^{\{1\}}$, so the image of~$\delta$ is $\mathrm{ad}(\delta^{\{1\}})\big(\sum_{i=1}^h[b_i^{(1)},a_i^{(1)}]+\sum_{r=2}^n t_{1r}\big)$, and $\sum_{i=1}^h[b_i^{(1)},a_i^{(1)}]+\sum_{r=2}^n t_{1r}=0$ by~\eqref{reltgnii}. 
\end{proof}

\subsection{Decompositions of single-valued forms}\label{ssec:42}

\begin{defn}
(a) Let $\mathrm{Forms}:=\Gamma(C\setminus\infty,\Omega^1_{C,\log})\,\hat\otimes\,
U(\mathrm{Lie}(b_1,\ldots ,b_h))^\wedge$.

(b) Let $\mathrm{Tuples}$ be the complex vector space spanned by tuples $(\eta_1,\ldots,\eta_h,\underline\eta)$ constituted by $\eta_1,\ldots \eta_h\in\mathrm{Forms}$
and by $\underline\eta\in\Gamma(C^2,(\Omega^1_{C,\log})^{[1]}(\Delta))\,\hat\otimes \,U(\mathrm{Lie}(b_1,\ldots ,b_h))^\wedge$.
 
(c) Let $\pmb{\mathrm{Forms}}:=\Gamma((C\setminus\infty)^n,(\Omega^1_{C,\log})^{[1]}(\Delta_1)))\,\hat\otimes\,
\hat{\mathfrak t}_{h,n}$.
\end{defn}

The pull-back to $\hat C$ (or its powers) of the differential forms which belong to these spaces is invariant under the action of ${\rm Aut}(\hat C/C)$ (or its powers). To stress this difference with the multi-valued differential forms considered in the previous section, we refer to the differential forms of this section as \emph{single-valued}. Note that, because of Lemma~\ref{lemma:decomp}, $\pmb H^{[1]}:=\pmb H^{[1]}_{\tilde x, \pmb\beta}\in\pmb{\rm Forms}$.

%We refer to the elements of these spaces as single-valued forms, meaning that they are well-defined 1-forms on the curve~$C$, as opposed to those considered in the previous section.

\begin{lem}
(a) There is a linear map $\mathrm{Tuples}\to\pmb{\mathrm{Forms}}$, defined by 
\begin{equation*}
(\eta_1,\ldots,\eta_h,\underline\eta)\,\mapsto\, \pmb\eta\,:=\,\sum_{i=1}^h
\mathrm{ad}(\eta_i^{\{1\}})(a_i^{(1)})
\,+\,\sum_{r=2}^n\mathrm{ad}(\underline\eta^{\{1r\}})(t_{1r}). 
\end{equation*}

(b) There is an injective linear map $\mathrm{Forms}\to\mathrm{Tuples}$, defined by $\epsilon\mapsto (\epsilon b_1,\ldots,\epsilon b_h,\mathrm{pr}_1^*\epsilon)$.

(c) The composed map $\mathrm{Forms}\to\mathrm{Tuples}\to\pmb{\mathrm{Forms}}$ is zero.  
\end{lem}
\begin{proof}
The statements~(a) and~(b) are obvious, while the proof of statement~(c) is the same as that of statement~(c) in Lemma~\ref{lemma:210428multi}.
\end{proof}

\subsection{Mapping multi-valued forms to single-valued forms}\label{ssec:43}

Let $\gamma:\hat C\setminus\pi^{-1}(\infty)\to \exp(\mathrm{Lie}(b_1,\ldots ,b_h)^\wedge)$ be a holomorphic function such that $\hat{\mathsf{B}}_i^*\gamma=\gamma e^{-b_i}$ for any $i\in[\![1,h]\!]$. We know that the space of such functions is non-empty, because it contains~$g:=g_{\tilde x,\pmb\beta}$.

To each $\gamma$ as above, we attach a morphism of complexes from $\mathcal Forms\to\mathcal Tuples\to\pmb{\mathcal Forms}$ to $\mathrm{Forms}\to\mathrm{Tuples}\to\pmb{\mathrm{Forms}}$ as follows.

\begin{lem}\label{lem:210503sv}
(a) The assignment $\delta\mapsto \gamma\delta$ defines a linear map $s_\gamma:{\mathcal Forms}\to{\mathrm{Forms}}$.\footnote{We identify $\mathrm{Forms}$ with the subspace of $\Gamma(\hat C\setminus\pi^{-1}(\infty),\Omega^1_{\hat C,\log})\,\hat\otimes\,
U(\mathrm{Lie}(b_1,\ldots ,b_h))^\wedge$ of 1-forms which are invariant under the action of ${\rm Aut}(\hat C/C)$. The same kind of identification is understood also for the other statements.}

(b) Let~$a$ be the antipode map of the Hopf algebra $U(\mathrm{Lie}(b_1,\ldots ,b_h))^\wedge$, induced by $b_{i_1}\cdots b_{i_m}\mapsto a(b_{i_1}\cdots b_{i_m}):=(-1)^mb_{i_m}\cdots b_{i_1}$. For $i\in [\![1,h]\!]$, define $\partial_i$ to be the unique linear endomorphism of $U(\mathrm{Lie}(b_1,\ldots ,b_h))^\wedge$ such that $u=\varepsilon(u)1+\sum_{i=1}^h \partial_i(u)b_i$ for any $u\in U(\mathrm{Lie}(b_1,\ldots ,b_h))^\wedge$, where~$\varepsilon$ is the counit map. Both~$a$ and~$\partial_i$ extend to endomorphisms of any space of functions with values in $U(\mathrm{Lie}(b_1,\ldots ,b_h))^\wedge$. Then the assignments $\kappa_i \mapsto \gamma\kappa_i$ and
\begin{equation}\label{eq:210508uk}
\underline{\kappa}\,\mapsto\, ({\rm pr}_1^*\gamma)\,\underline{\kappa}\,({\rm pr}_2^*\gamma)^{-1}\,+\,\sum_{i=1}^h\,({\rm pr}_1^*\gamma)\,\kappa_i\,({\rm pr}_2^*\,a(\partial_i(\gamma)))
\end{equation}
define a linear map $S_\gamma:{\mathcal Tuples}\to\mathrm{Tuples}$.

(c) The assignment $\pmb\kappa \mapsto {\rm Ad}_{\gamma^{\{1\}}\cdots \gamma^{\{n\}}}(\pmb \kappa)$ defines a linear map $\pmb s_\gamma:\pmb{\mathcal Forms}\to\pmb{\mathrm{Forms}}$.

(d) The two squares in the following diagram are commutative:
\[
\begin{tikzcd}
\mathcal Forms \arrow[r, hook] \arrow[d,"{s_\gamma}"]
& {\mathcal Tuples} \arrow[d, "{S_\gamma}"]  \arrow[r] & \pmb{\mathcal Forms} \arrow[d, "{\pmb s_\gamma}"]  \\
\mathrm{Forms} \arrow[r, hook]
& \mathrm{Tuples} \arrow[r] & \pmb{\mathrm{Forms}}
\end{tikzcd}
\]

\end{lem}

\begin{proof} (a) One needs to check that $\hat{\mathsf{B}}_j^*\gamma\delta=\gamma\delta$ for any $j\in[\![1,h]\!]$, which follows immediately from the definition of $\gamma$ and $\delta$. 

(b) By (a) we already know that $\hat{\mathsf{B}}_j^*\gamma\kappa_i=\gamma\kappa_i$. Let us denote by $\underline{\eta}$ the image of $\underline\kappa$ via $S_\gamma$ defined by eq.~\eqref{eq:210508uk}; we are left with checking that $(\hat{\mathsf{B}}^{(1)}_j)^*\underline{\eta}=\underline{\eta}$ and that $(\hat{\mathsf{B}}^{(2)}_j)^*\underline{\eta}=\underline{\eta}$ for any $j\in[\![1,h]\!]$. The first identity immediately follows from the definition of $\gamma$, $\kappa_i$ and $\underline{\kappa}$. In order to verify the second identity, let us first compute $\hat{\mathsf{B}}_j^*\partial_i(\gamma)$.

By definition, $\gamma=1+\sum_{i=1}^h\partial_i(\gamma)\,b_i$. On the one hand, this implies that 
\begin{equation*}
\hat{\mathsf{B}}_j^*\gamma\,=\,1\,+\,\sum_{i=1}^h\hat{\mathsf{B}}_j^*\partial_i(\gamma)\,b_i\,.
\end{equation*}
On the other hand, one can write
\begin{equation*}
\gamma \,e^{-b_j}\,=\,\gamma\,+\,\gamma\,\frac{e^{-b_j}-1}{b_j}\,b_j\,=\,1\,+\,\sum_{i=1}^h\bigg(\partial_i(\gamma)\,+\,\delta_{ij}\,\gamma\,\frac{e^{-b_j}-1}{b_j}\bigg)\,b_i\,.
\end{equation*}
Since we know that $\hat{\mathsf{B}}_j^*\gamma=\gamma e^{-b_j}$, we can compare these two equations and deduce that
\begin{equation*}
\hat{\mathsf{B}}_j^*\partial_i(\gamma)\,=\,\partial_i(\gamma)\,+\,\delta_{ij}\,\gamma\,\frac{e^{-b_j}-1}{b_j}\,.
\end{equation*}

Using this, we get
\begin{align*}
&(\hat{\mathsf{B}}^{(2)}_j)^*\underline{\eta}\,=\,({\rm pr}_1^*\gamma)\,(\hat{\mathsf{B}}^{(2)}_j)^*\underline{\kappa}\,(\hat{\mathsf{B}}^{(2)}_j)^*({\rm pr}_2^*\gamma)^{-1}\,+\,\sum_{i=1}^h\,({\rm pr}_1^*\gamma)\,\kappa_i\,(\hat{\mathsf{B}}^{(2)}_j)^*({\rm pr}_2^*\,a(\partial_i(\gamma)))\\
&\,=\,({\rm pr}_1^*\gamma)\,\bigg(\underline\kappa\,e^{-b_j}+\kappa_j\,{1-e^{-b_j}\over b_j}\bigg)\,e^{b_j}\,({\rm pr}_2^*\gamma)^{-1}+\sum_{i=1}^h({\rm pr}_1^*\gamma)\,\kappa_i\,\Bigg({\rm pr}_2^*\,a\bigg(\partial_i(\gamma)+\delta_{ij}\,\gamma\,\frac{e^{-b_j}-1}{b_j}\bigg)\Bigg)\\
&\,=\,\underline{\eta}\,+\,({\rm pr}_1^*\gamma)\,\bigg(\kappa_j\,{e^{b_j}-1\over b_j}\bigg)\,({\rm pr}_2^*\gamma)^{-1}\,+\,({\rm pr}_1^*\gamma)\,\kappa_j\,a\bigg({\rm pr}_2^*\gamma\,{e^{-b_j}-1\over b_j}\bigg)\\
&\,=\,\underline{\eta}\,,
\end{align*}
where the last identity follows from the following properties of the antipode: $a(uv)=a(v)a(u)$ for any $u,v$, $a(b_j)=-b_j$ for any~$j$, and $a(u)=u^{-1}$ for any~$u$ which is group-like.

(c) One needs to check that $(\hat{\mathsf{B}}^{(r)}_i)^*{\rm Ad}_{\gamma^{\{1\}}\cdots \gamma^{\{n\}}}(\pmb \kappa)={\rm Ad}_{\gamma^{\{1\}}\cdots \gamma^{\{n\}}}(\pmb \kappa)$ for any $i\in[\![1,h]\!]$ and for any $r\in[\![1,n]\!]$. This follows from the fact that the action of $\hat{\mathsf{B}}^{(r)}_i$ is trivial on~$\gamma^{\{s\}}$ for $s\neq r$, from $(\hat{\mathsf{B}}^{(r)}_i)^*\circ{\rm Ad}_{\gamma^{\{r\}}}={\rm Ad}_{\gamma^{\{r\}}}\circ e^{-{\rm ad}(b_i^{(r)})}$ and from the fact that $[b_i^{(r)},b_j^{(s)}]=0$. 

(d) To prove that the square on the left commutes, one needs to check for any element $\delta\in\mathcal Forms$ that $S_\gamma(\delta b_1,\dots ,\delta b_h,{\rm pr}_1^*\delta)=(\gamma\delta b_1,\ldots ,\gamma\delta b_h,{\rm pr}_1^*\gamma\delta)$. This is obviously true for the first~$h$ components. We are left with showing that
\begin{equation*}
({\rm pr}_1^*\gamma)\,{\rm pr}_1^*\delta\,({\rm pr}_2^*\gamma)^{-1}\,+\,\sum_{i=1}^h\,({\rm pr}_1^*\gamma)\,\delta\,b_i\,({\rm pr}_2^*\,a(\partial_i(\gamma)))\,=\,{\rm pr}_1^*\gamma\delta\,.
\end{equation*}
This is a consequence of the fact that $\gamma^{-1}=a(\gamma)=1-\sum_{i=1}^hb_ia(\partial_i(\gamma))$.

To prove that the square on the right commutes, one needs to check for any element $(\kappa_1\ldots ,\kappa_h,\underline{\kappa})\in\mathcal Tuples$ that
\begin{multline}\label{eq:210510}
\sum_{i=1}^h\,{\rm Ad}_{\gamma^{\{1\}}\cdots \gamma^{\{n\}}}\,\mathrm{ad}(\kappa_i^{\{1\}})(a_i^{(1)})\,+\,\sum_{r=2}^n\,{\rm Ad}_{\gamma^{\{1\}}\cdots \gamma^{\{n\}}}\,\mathrm{ad}(\underline\kappa^{\{1r\}})(t_{1r})\,=\,\sum_{i=1}^h\mathrm{ad}(\gamma^{\{1\}}\kappa_i^{\{1\}})(a_i^{(1)})\\
\,+\,\sum_{r=2}^n\mathrm{ad}\bigg(({\rm pr}_1^*\gamma^{(1)})\underline{\kappa}^{\{1r\}}({\rm pr}_r^*\gamma^{(1)})^{-1}+\sum_{i=1}^h\,({\rm pr}_1^*\gamma^{(1)})\kappa_i^{\{1\}}({\rm pr}_r^*a(\partial_i(\gamma))^{(1)})\bigg)(t_{1r}).
\end{multline}

First of all, notice that ${\rm Ad}_{\gamma^{\{1\}}\cdots \gamma^{\{n\}}}\,\mathrm{ad}(\kappa_i^{\{1\}})=\mathrm{ad}(\gamma^{\{1\}}\cdots \gamma^{\{n\}}\kappa_i^{\{1\}})$. Since $[b_j^{(r)},b_i^{(1)}]=0$ for any $r\in [\![2,n]\!]$ and any $i,j\in [\![1,h]\!]$, we have $\gamma^{\{r\}}\kappa^{\{1\}}_i=\kappa_i^{\{1\}}\gamma^{\{r\}}$ for any $r\in [\![2,n]\!]$. Moreover, since $[b_j^{(r)},a_i^{(1)}]=\delta_{ij}t_{1r}$ for any $r\in [\![2,n]\!]$ and any $i,j\in [\![1,h]\!]$, we have ${\rm ad}(\gamma^{\{r\}})(a_i^{(1)})=a_i^{(1)}+{\rm ad}(\gamma^{\{r\}}-1)(a_i^{(1)})=a_i^{(1)}+{\rm ad}(\partial_i(\gamma)^{\{r\}})(t_{1r})$ for any $r\in [\![2,n]\!]$. Finally, the fact that $[b_j^{(s)},t_{1r}]=0$ for any $s\neq 1,r$ implies that the first term on the l.h.s. of~\eqref{eq:210510} coincides with
\begin{equation}\label{eq:210510b}
\sum_{i=1}^h\,\mathrm{ad}(\gamma^{\{1\}}\kappa_i^{\{1\}})(a_i^{(1)})\,+\,\sum_{r=2}^n\,\sum_{i=1}^h\,{\rm ad}(\gamma^{\{1\}}\kappa^{\{1\}}_i(\partial_i(\gamma))^{\{r\}})(t_{1r})\,.
\end{equation}
The fact that $[b_j^{(r)},t_{1r}]=-[b_j^{(1)},t_{1r}]$ implies that ${\rm ad}(\partial_i(\gamma))^{(r)}(t_{1r})={\rm ad}(a(\partial_i(\gamma))^{(1)})(t_{1r})$, and therefore we can rewrite~\eqref{eq:210510b}, i.e. the first term on the l.h.s. of~\eqref{eq:210510}, as\footnote{Recall that, by definition, $\gamma^{\{r\}}={\rm pr}_r^*\gamma^{(r)}$.}
\begin{equation*}
\sum_{i=1}^h\,\mathrm{ad}(\gamma^{\{1\}}\kappa_i^{\{1\}})(a_i^{(1)})\,+\,\sum_{r=2}^n\,\sum_{i=1}^h\,{\rm ad}\Big(({\rm pr}_1^*\gamma^{(1)})\kappa_i^{\{1\}}({\rm pr}_r^*a(\partial_i(\gamma))^{(1)})\Big)(t_{1r})\,.
\end{equation*}
Comparing this expression with the r.h.s. of~\eqref{eq:210510}, we are left with proving that
\begin{equation}\label{eq:210510c}
{\rm Ad}_{\gamma^{\{1\}}\cdots \gamma^{\{n\}}}\,\mathrm{ad}(\underline\kappa^{\{1r\}})(t_{1r})\,=\,\mathrm{ad}\Big(({\rm pr}_1^*\gamma^{(1)})\underline{\kappa}^{\{1r\}}({\rm pr}_r^*\gamma^{(1)})^{-1}\Big)(t_{1r})\,
\end{equation}
for any $r\in [\![2,n]\!]$.

Using that ${\rm Ad}_{\gamma^{\{1\}}\cdots \gamma^{\{n\}}}\,\mathrm{ad}(\underline\kappa^{\{1r\}})=\mathrm{ad}(\gamma^{\{1\}}\cdots \gamma^{\{n\}}\underline\kappa^{\{1r\}})$, that $[b_j^{(r)},b_i^{(1)}]=0$ for any $r\in [\![2,n]\!]$ and any $i,j\in [\![1,h]\!]$, and that $[b_j^{(s)},t_{1r}]=0$ for any $s\neq 1,r$, we can rewrite the l.h.s. of~\eqref{eq:210510c} as ${\rm ad}(\gamma^{\{1\}}\underline\kappa^{\{1r\}}\gamma^{\{r\}})(t_{1r})$. The identity~\eqref{eq:210510c} is then a consequence of the fact that ${\rm ad}(\gamma^{\{r\}})(t_{1r})={\rm ad}(({\rm pr}_r^*\gamma^{(1)})^{-1})(t_{1r})$, which follows from $[b_j^{(r)},t_{1r}]=-[b_j^{(1)},t_{1r}]$.
\end{proof}

\subsection{Decompositions of $\pmb K^{[1]}$}\label{ssec:44}

Let $P\in\hat C$. A key step in~\cite{EnrConfSp} towards constructing the 1-form~$\pmb K^{[1]}$ was the introduction of two families of 1-forms\footnote{In~\cite{EnrConfSp} the forms $\omega_{P,i_1\ldots i_m}(z)$ are denoted $\omega^{\underline{z}P}_{i_1\ldots i_m}$, while the forms $\psi_{P,i_1\ldots i_m}(z_1,z_2)$ are denoted $\psi^{\underline{z}_1Pz_2}_{i_1\ldots i_m}$.} 
$$
\Big(\omega_{P,i_1\ldots i_m}\Big)_{\substack{m\geq 1\\i_1,\ldots ,i_m\in [\![1,h]\!]}}\,,\quad \quad \quad \Big(\psi_{P,i_1\ldots i_m}\Big)_{\substack{m\geq 0\\i_1,\ldots ,i_m\in [\![1,h]\!]}}\,.
$$  

The first family is contained in $\Gamma(\hat C,\Omega^1_{\hat C}(\pi^{-1}(\pi(P))))\subset \Gamma(\hat C,\Omega^1_{\hat C,\log})$; it satisfies, and is uniquely determined by, the following properties:
\begin{itemize}
\item If $m=1$ then the forms $\omega_{P,i}$ do not depend on~$P$, and they coincide with (the pullback to~$\hat C$ of) the holomorphic 1-forms~$\omega_i$ introduced in \S\ref{def:omega} (see Lemma~6 of~\cite{EnrConfSp}).
\item For any $j\in [\![1,h]\!]$ one has (see Lemma~6 of~\cite{EnrConfSp})
\begin{equation}\label{eq:210503n1}
\hat{\mathsf{B}}_j^*\,\omega_{P,i_1\ldots i_m}\,=\,\sum_{l\geq 0} \,\frac{1}{l!}\,\delta_{ji_1\cdots i_l}\,\omega_{P,i_{l+1}\ldots i_m}\,.
\end{equation}
\item One has (see Lemma~6 of~\cite{EnrConfSp})
\begin{equation*}
{\rm res}_{P}\,\omega_{P,i_1\ldots i_m}\,=\,-\frac{\delta_{m2} \,\delta_{i_1i_2}}{2\pi\mathrm{i}}\,.
\end{equation*}
\item For $P\in \hat{D}$, one has (see the proof of Lemma~6 of~\cite{EnrConfSp})
\begin{equation}\label{eq:210516n1}
\int_{\hat{\mathcal{A}}_j}\,\omega_{P,i_1\ldots i_m}\,=\,\frac{\mathsf{b}_{m-1}}{(m-1)!}\,\delta_{ji_1\cdots i_m}\,,
\end{equation}
where $\mathsf{b}_{m-1}$ is the $(m-1)$-th Bernoulli number\footnote{The Bernoulli numbers are defined by $\sum_{m\geq 0}\mathsf{b}_m\tfrac{t^m}{m!}:=\tfrac{t}{e^t-1}$.}.
\end{itemize}

The second family is contained in $\Gamma(\hat C^2,(\Omega^1_{\hat C}(\pi^{-1}(\pi(P))))^{[1]}(\hat\Delta))\subset \Gamma(\hat C^2,(\Omega^1_{\hat C,\log})^{[1]}(\hat\Delta))$; it satisfies, and is uniquely determined by, the following properties:
\begin{itemize}
\item If $m=0$ then~$\psi_P$ is the fundamental form of the third kind defined in \S\ref{sec:fundformIIkind}.
\item For any $m\geq 1$ the forms $\psi_{P,i_1\ldots i_m}$ belong to $\Gamma\big(\hat C^2,(\Omega^1_{\hat C})^{[1]}\big)$, and the integrals along the loops $\hat{\mathcal{A}}_{i}$ of their restrictions on $\hat C\times\{Q\}$ for any $Q\in\hat C$ vanish identically (see Lemma~8 of~\cite{EnrConfSp}).
\item For any $j\in [\![1,h]\!]$ one has (see Prop.~10 of~\cite{EnrConfSp})
\begin{equation}\label{eq:210503n2}
(\hat{\mathsf{B}}^{(1)}_j)^*\,\psi_{P,i_1\ldots i_m}\,=\,\sum_{l\geq 0} \,\frac{1}{l!}\,\delta_{ji_1\cdots i_l}\,\psi_{P,i_{l+1}\ldots i_m}\,.
\end{equation}
\item For any $j\in [\![1,h]\!]$ one has (see Prop.~10 of~\cite{EnrConfSp})
\begin{multline}\label{eq:210503n3}
(\hat{\mathsf{B}}^{(2)}_j)^*\,\psi_{P,i_1\ldots i_m}\,=\\
\sum_{l\geq 0} \,\frac{(-1)^l}{l!}\,\delta_{ji_m\cdots i_{m-l+1}}\,\psi_{P,i_{1}\ldots i_{m-l}}\,+\,\sum_{l\geq 1} \,\frac{(-1)^{l-1}}{l!}\,\delta_{ji_m\cdots i_{m-l+2}}\,\omega_{P,i_{1}\ldots i_{m-l+1}j}\,.
\end{multline}
\end{itemize}

Furthermore, these families are related by the following relation: for any $j\in [\![1,h]\!]$ and for any $P'\in\hat C$ one has (see Prop.~10 of~\cite{EnrConfSp})
\begin{equation}\label{eq:210503n4}
\omega_{P,i_1\ldots i_m}-\omega_{P',i_1\ldots i_m}\,=\,\delta_{i_{m-1}i_m}\psi_{P,i_1\ldots i_{m-2}}\big|_{\hat C\times\{P'\}}\,.
\end{equation}

The dependence of these two families on $P\in \hat C$ is described in Lemma~9 of~\cite{EnrConfSp}.

\begin{defn}\label{def210514n1}
For $P,P'\in\hat C$ and for $j\in [\![1,h]\!]$ we set
\begin{align*}
D_{P,P'}\,&:=\,\sum_{\substack{m\geq 0\\i_1,\ldots ,i_m\in [\![1,h]\!]}}\,\psi_{P,i_1\ldots i_m}\big|_{\hat C\times \{P'\}}\,\otimes\,b_{i_1}\cdots b_{i_m}\,,\\
K_{P,j}\,&:=\,\sum_{\substack{m\geq 0\\i_1\ldots i_m\in [\![1,h]\!]}}\,\omega_{P,i_1\ldots i_m,j}\,\otimes\,b_{i_1}\cdots b_{i_m}\,,\\
\underline{K}_{P}\,&:=\,\sum_{\substack{m\geq 0\\i_1,\ldots ,i_m\in [\![1,h]\!]}}\,\psi_{P,i_1\ldots i_m}\,\otimes\,b_{i_1}\cdots b_{i_m}\,,
\end{align*}
so that
\begin{align*}
D_{P,P'}\,&\in\,\Gamma(\hat C,\Omega^1_{\hat C}(\pi^{-1}(\pi(P))+\pi^{-1}(\pi(P'))))\,\hat\otimes\,U(\mathrm{Lie}(b_1,\ldots ,b_h))^\wedge\,,\\
K_{P,j}\,&\in\,\Gamma(\hat C,\Omega^1_{\hat C}(\pi^{-1}(\pi(P))))\,\hat\otimes\,U(\mathrm{Lie}(b_1,\ldots ,b_h))^\wedge\,,\\
\underline{K}_{P}\,&\in\,\Gamma(\hat C^2,(\Omega^1_{\hat C}(\pi^{-1}(\pi(P))))^{[1]}(\hat\Delta))\,\hat\otimes\,U(\mathrm{Lie}(b_1,\ldots ,b_h))^\wedge\,.
\end{align*}
\end{defn}

\begin{lem}\label{lem:210503multi}
(a) For any $P\in\hat C$ one has $(K_{P,1},\ldots,K_{P,h},\underline K_P)\in\mathcal Tuples$, and the 
image of this element in $\pmb{\mathcal Forms}$ is~$\pmb K^{[1]}$. 

(b) For any $P,P'\in\hat C$ one has $D_{P,P'}\in\mathcal Forms$, and the image of this element in $\mathcal Tuples$ is $(K_{P,1},\ldots,K_{P,h},\underline K_P)-(K_{P',1},\ldots,K_{P',h},\underline K_{P'})$.

(c) For any $P,P',P''\in \hat C$ one has $D_{P,P'}+D_{P',P''}=D_{P,P''}$.
\end{lem}

\begin{proof}
(a) First of all, note that the property~\eqref{eq:210503n1} implies that $\hat{\mathsf{B}}_i^*K_{P,j}=e^{b_i}K_{P,j}$ for any $i,j\in[\![1,h]\!]$, and therefore that $K_{P,j}\in\mathcal Forms$ for any $j\in[\![1,h]\!]$. Moreover, the properties~\eqref{eq:210503n2} and~\eqref{eq:210503n3} imply that~$\underline{K}_P$ satisfies~\eqref{eq280421mon1} and~\eqref{eq280421mon2}, respectively, from which we conclude that $(K_{P,1},\ldots,K_{P,h},\underline K_P)\in\mathcal Tuples$. 

The image of this element in $\pmb{\mathcal Forms}$ is\footnote{The notation $[b_{i_1},\ldots [b_{i_m},a_j]]$ stands for the nested Lie bracket expression ${\rm ad}(b_{i_1})\cdots {\rm ad}(b_{i_m})(a_j)$.}
\begin{equation*}
\sum_{j=1}^h\sum_{\substack{m\geq 0\\i_1,\ldots ,i_m\in [\![1,h]\!]}}\omega_{P,i_1\ldots i_m,j}\,\otimes\,[b^{(1)}_{i_1},\ldots [b^{(1)}_{i_m},a^{(1)}_j]]\,+\,\sum_{s=2}^n\sum_{\substack{m\geq 0\\i_1,\ldots ,i_m\in [\![1,h]\!]}}\psi_{P,i_1\ldots i_m}\,\otimes\,[b^{(1)}_{i_1},\ldots [b^{(1)}_{i_m},t_{1s}]],
\end{equation*}
and this is precisely the definition of~$\pmb K^{[1]}$ (denoted $\alpha_1$ in the reference) from~\cite{EnrConfSp},~\S4.3. Alternatively, one can also check that this expression satisfies the properties listed in Proposition~\ref{prop:412:2303}, which uniquely determine~$\pmb K^{[1]}$. 

(b) The fact that $D_{P,P'}\in \mathcal Forms$ follows from~\eqref{eq:210503n2}, while the fact that its image in $\mathcal Tuples$ is $(K_{P,1},\ldots,K_{P,h},\underline K_P)-(K_{P',1},\ldots,K_{P',h},\underline K_{P'})$ follows from~\eqref{eq:210503n4}.

(c) This follows from statement~(b) and from the injectivity of $\mathcal Forms\to\mathcal Tuples$. Alternatively, a direct proof can be deduced from the fact that $\psi_P\big|_{\hat C\times \{P'\}}$ satisfies this identity, which is true because $\psi_P\big|_{\hat C\times \{P'\}}=\big(\int_{P}^{P'}\big)^{(2)}\varphi$. 
\end{proof}
In particular, it follows from (a) that the image of $(K_{P,1},\ldots,K_{P,h},\underline K_P)$ in $\pmb{\mathcal Forms}$ does not depend on the choice of~$P$, thus leading to a family of different decompositions
\begin{equation}\label{eq:decompK1}
\pmb K^{[1]}\,=\,\sum_{j=1}^h
\mathrm{ad}(K_{P,j}^{\{1\}})(a_j^{(1)})
\,+\,\sum_{r=2}^n\mathrm{ad}(\underline{K}_P^{\{1r\}})(t_{1r})\,.
\end{equation}
Notice that the independence of the image of $(K_{P,1},\ldots,K_{P,h},\underline K_P)$ on~$P$ may also be viewed as a consequence of~(b): indeed, (b) implies that for $P,P'\in\hat C$ the image of the difference $(K_{P,1},\ldots,K_{P,h},\underline K_P)-(K_{P',1},\ldots,K_{P',h},\underline K_{P'})$ in $\pmb{\mathcal Forms}$ is the image of~$D_{w,w'}$ in $\pmb{\mathcal Forms}$ via the composed map $\mathcal Forms\to\mathcal Tuples\to\pmb{\mathcal Forms}$, which vanishes by Lemma~\ref{lemma:210428multi},~(c).

\subsection{Decompositions of $\pmb H^{[1]}$}\label{ssec:45}

\begin{defn}
(a) For $P,P'\in\hat C$ we set\footnote{$g=g_{\tilde x,\pmb \beta}:\hat C\setminus\pi^{-1}(\infty)\to \exp(\mathrm{Lie}(b_1,\ldots ,b_h)^\wedge)$ is the map constructed in \S\ref{sec:finalcorrection}.} $E_{P,P'}:=s_g(D_{P,P'})\in\mathrm{Forms}$. 

(b) For $P\in\hat C$ we set $(H_{P,1},\ldots,H_{P,h},\underline H_P):=S_g(K_{P,1},\ldots,K_{P,h},\underline K_P)
\in\mathrm{Tuples}$.
\end{defn}

\begin{lem}\label{lem49}
(a) For $P\in\hat C$, the image of $(H_{P,1},\ldots,H_{P,h},\underline H_P)$ in $\pmb{\mathrm{Forms}}$ is $\pmb H^{[1]}$.

(b) For $P,P'\in\hat C$, the image of $E_{P,P'}$ in $\mathrm{Tuples}$ is 
 $(H_{P,1},\ldots,H_{P,h},\underline H_P)-(H_{P',1},\ldots,H_{P',h},\underline H_{P'})$. 

(c) For $P,P',P''\in\hat C$, $E_{P,P'}+E_{P',P''}=E_{P,P''}$. 
\end{lem}

\begin{proof} Statement~(a) of this lemma follows by combining statement~(a) of Lemma~\ref{lem:210503multi} with statement~(d) of Lemma~\ref{lem:210503sv} and with the fact that, by definition, $\pmb H^{[1]}=\pmb s_g(\pmb K^{[1]})$.

Statement~(b) (resp.~(c)) of this lemma follows by combining statement~(b) (resp.~(c)) of Lemma~\ref{lem:210503multi} with statement~(d) of Lemma~\ref{lem:210503sv}. Statement~(c) can also be proven by combining statement~(b) of this lemma with the fact that ${\rm Forms}\to {\rm Tuples}$ is injective.
\end{proof}

%BACKGROUND 
%
%Pbm : infty // omega^infty_{i_1,…,i_n}. Formule globale pour omega^infty_{i_1,…,i_n} en infty avec outils locaux en infty : dépend de (tilde x,beta) avec beta famille de formes avec sings en infty. // 
%
%Objets : KK multival sur C // K_{infty,i} forme sur hat C avec ps en infty et B_j^*K_{infty,i}=e^{b_j}*K_{infty,i} // g fonction sur hat C avec pole en infty B_j^*g=g*e^{-b_j} , dépend de tilde x,beta // H_{infty,i}=g*K_{infty,i} 1-val sur C, sing en infty, dépend de tilde x,beta.  // 

\section{Algorithms for the computation of $\pmb K$ and ${\pmb J}_{\tilde x,\pmb \beta}$}\label{sec5}

In this section we provide algorithms to write some of the forms and functions considered in this article in terms of iterated integrals or residues of classical differential forms on Riemann surfaces. In particular, one obtains explicit formulas for~$\pmb K$ and~${\pmb J}_{\tilde x,\pmb \beta}$ when $n=1$. As in the previous section, from now on we will omit~$\tilde x$ and~$\pmb\beta$ from the notation.

In \S\ref{ssec:51}, we review how to compute~$g$ in terms of iterated integrals directly from its definition, and we provide an alternative algorithm to compute the logarithm of~$g$ and the differential form $I=gd(g^{-1})$. In \S\ref{ssec:52} and \S\ref{ssec:53} we recall some facts about differential forms on closed Riemann surfaces which we use in \S\ref{ssec:K} to obtain recursive formulas in terms of iterated integrals and residues for some of the constituents of~$\pmb K$ and~$\pmb H$ introduced in \S\ref{ssec:44} and \S\ref{ssec:45}, respectively. We report in \S\ref{sec:computations} examples up to degree~3 of computations performed using these methods in the case $n=1$.

\subsection{Computation of $g$ and $\pmb I$}\label{ssec:51}

\subsubsection{Computation of $\pmb \Lambda$ and $g$}\label{ssec:g}

It follows by combining Lemma~\ref{lemma:hatg:06042021}, Definition~\ref{def:g:30032021}, Proposition~\ref{prop:2:1741} and eq.~\eqref{eq:diffeqFLambda}, that the map~$g=g_{\tilde x,\pmb\beta}:\hat C\setminus\pi^{-1}(\infty)\to\exp({\rm Lie}(b_1,\ldots ,b_h)^\wedge)$ defined in \S\ref{sec:finalcorrection} satisfies the differential equation
\begin{equation}\label{eq:210616n6}
d(g^{-1})\,=\,g^{-1}\,\sum_{i=1}^{2h}\gamma_i\otimes\Lambda_i\,,
\end{equation}
where $\pmb\Lambda\in ({\rm Lie}(b_1,\ldots ,b_h)^\wedge)^{2h}$ is uniquely determined (see Remark~\ref{rem:210616}) by the system of equations
\begin{align}\label{eq:210616n5}
\sum_{m\geq 1}\int_{\mathcal{A}_i}\underbrace{\gamma_{\pmb \Lambda}\cdots \gamma_{\pmb \Lambda}}_m&\,=\,0\,,\\
\sum_{m\geq 1}\int_{\mathcal{B}_i}\underbrace{\gamma_{\pmb \Lambda}\cdots \gamma_{\pmb \Lambda}}_m&\,=\,e^{b_i}-1\,,\notag
\end{align}
with $i\in [\![1,h]\!]$, where $\gamma_{\pmb \Lambda}:=\sum_{i=1}^{2h}\gamma_i\otimes\Lambda_i$. From this it follows that
\begin{equation*}
g^{-1}\,=\,1\,+\,\sum_{m\geq 1}\sum_{i_1,\ldots ,i_m\in [\![1,2h]\!]}[i_1,\ldots,i_m|\bullet]\otimes \Lambda_{i_1}\cdots \Lambda_{i_m}\,,
\end{equation*}
where we introduce the shorthand notation\footnote{These are integrals over $\tilde C\setminus p^{-1}(\infty)$ which in general depend on the integration path and define functions on $\widetilde{C\setminus\infty}$, but we know that the special combination which gives rise to~$g^{-1}$ does not depend on the path and it is well-defined on $\tilde C\setminus p^{-1}(\infty)$, and even on $\hat C\setminus \pi^{-1}(\infty)$.}
\begin{equation}\label{eqnotttt}
[i_1,\ldots ,i_m|\bullet]\,:=\,\int_{\tilde x}^{\bullet}\gamma_{i_1}\cdots \gamma_{i_m}\,.
\end{equation}
This implies that
\begin{equation}\label{eq2:211006}
g\,=\,1\,+\,\sum_{m\geq 1}(-1)^m\sum_{i_1,\ldots ,i_m\in [\![1,2h]\!]}[i_1,\ldots,i_m|\bullet]\otimes \Lambda_{i_m}\cdots \Lambda_{i_1}\,.
\end{equation}

The only obstacle towards computing~$g$ is then inverting the relations~\eqref{eq:210616n5} to obtain the elements $\Lambda_i\in {\rm Lie}(b_1,\ldots ,b_h)^\wedge$, which can be done recursively by writing them in the form\footnote{Alternatively, in these formulas one may use the notation~\eqref{eqnotttt} and replace $[i_1,\ldots ,i_m|j]$ by $[i_1,\ldots ,i_m|\hat{\mathsf{A}}_j(\hat{x})]$ and $[i_1,\ldots ,i_m|h+j]$ by $[i_1,\ldots ,i_m|\hat{\mathsf{B}}_j(\hat{x})]$. We stress again the fact that, in general, $[i_1,\ldots ,i_m|\hat y]$ is not well-defined, but it is in this case.}
\begin{align}\label{eq1:211006}
\sum_{m\geq 1}\sum_{i_1,\ldots ,i_m\in [\![1,2h]\!]}[i_1,\ldots,i_m|j]\otimes \Lambda_{i_1}\cdots \Lambda_{i_m}&\,=\,0\,,\\
\sum_{m\geq 1}\sum_{i_1,\ldots ,i_m\in [\![1,2h]\!]}[i_1,\ldots,i_m|h+j]\otimes \Lambda_{i_1}\cdots \Lambda_{i_m}&\,=\,e^{b_j}-1\notag\,,
\end{align}
where for $j\in [\![1,h]\!]$ we introduce the shorthand notations
$$
[i_1,\ldots ,i_m|j]\,:=\,\int_{\mathcal{A}_j}\gamma_{i_1}\cdots \gamma_{i_m}\,,\quad\quad\quad\quad\quad [i_1,\ldots ,i_m|h+j]\,:=\,\int_{\mathcal{B}_j}\gamma_{i_1}\cdots \gamma_{i_m}\,.
$$ 
We will give in \S\ref{sec:computations} explicit low-degree formulas computed with this method.

\subsubsection{Computation of $\log(g)$ and $\pmb I$}\label{ssec:lambda}
We set 
$$
\lambda\,:=\,\log(g)\,\in\Gamma(\hat C\setminus\pi^{-1}(\infty),\mathcal{O}_{\hat C})\,\hat\otimes\,{\rm Lie}(b_1,\ldots ,b_h)^\wedge\,.
$$
In particular, $\lambda(\hat x)=0$ (because $g(\hat x)=1)$). We also set 
$$
I\,:=\,gd(g^{-1})\,\in \Gamma(C\setminus\infty,\Omega^1_{C})\,\hat\otimes\,
\mathrm{Lie}(b_1,\ldots,b_h)^\wedge\,.
$$
Note that, by Lemma~\ref{lemma:decomp}, $\pmb I=\pmb I^{[1]}+\cdots +\pmb I^{[n]}=g^{\{1\}}d((g^{\{1\}})^{-1})+\cdots +g^{\{n\}}d((g^{\{n\}})^{-1})$, and $g^{\{r\}}d((g^{\{r\}})^{-1})=I^{\{r\}}$, with $I$ as above. 

We have seen in the previous section how to compute~$g$, which in turn allows us to compute also~$\lambda$ and~$I$ (and therefore $\pmb I$). In this section we present an alternative approach: we obtain a recursive formula for~$\lambda$, which can then be used to compute~$g$ and~$I$. The formulas for~$\lambda$ and~$I$ obtained with this method, as opposed to those obtained by first computing~$g$, highlight the fact that~$\lambda$ and~$I$ take values in ${\rm Lie}(b_1,\ldots ,b_h)^\wedge$; this fact also makes such formulas more compact.
%Let $\lambda=\lambda[1]+\lambda[2]+\cdots $ and $I=I[1]+I[2]+\cdots $. In particular, we have $\lambda[1]=g[1]=-\sum_{i=1}^h\,[h+i|\bullet]\otimes b_i$, and $I[1]=-dg[1]=\sum_{i=1}^h\gamma_{h+i}\otimes b_i$. 
\begin{lem}\label{lemmmm}
For each $i\in[\![1,h]\!]$ we have $\int_{\hat{\mathcal{A}}_i}d\lambda=0$ and $\int_{\hat{\mathcal{B}}_i}d\lambda=-b_i$.
\end{lem}
\begin{proof}
The identity $\int_{\hat{\mathcal{A}}_i}d\lambda=0$ follows from the fact that $\lambda$ is a holomorphic function on $\hat C\setminus\pi^{-1}(\infty)$, and $\hat{\mathcal{A}}_i$ is a cycle in $\hat C\setminus\pi^{-1}(\infty)$.

As for the identity $\int_{\hat{\mathcal{B}}_i}d\lambda=-b_i$, note first that $\int_{\hat{\mathcal{B}}_i}d\lambda=\int_{\hat x}^{\hat{\mathsf{B}}_i(\hat x)}d\lambda=\lambda(\hat{\mathsf{B}}_i(\hat x))-\lambda(\hat x)$. We know that $\lambda(\hat x)=0$, so we only need to prove that $\lambda(\hat{\mathsf{B}}_i(\hat x))=-b_i$, or equivalently that $g(\hat{\mathsf{B}}_i(\hat x))=e^{-b_i}$. This is a consequence of the fact that $\hat{\mathsf{B}}_i^*g=ge^{-b_i}$ and that $g(\hat x)=1$.
\end{proof}

\begin{cor}\label{CorAltConstr}
For $i\in[\![1,h]\!]$ let $\hat{\mathcal{C}}_i=\hat{\mathcal{A}}_i$  and $\hat{\mathcal{C}}_{h+i}=\hat{\mathcal{B}}_i$. Then
\begin{equation*}
I=\sum_{i=1}^h\gamma_{h+i}\otimes b_i\,+\,\sum_{i=1}^{2h}\gamma_i\otimes\bigg(\int_{\hat{\mathcal{C}}_i}I+d\lambda\bigg)\,.
\end{equation*}
\end{cor}
\begin{proof}
It follows from eq.~\eqref{eq:210616n6} that
\begin{equation}\label{eq:210616n1}
I=\sum_{i=1}^{2h}\gamma_i\otimes \Lambda_i\,.
\end{equation}  
Then $\Lambda_i=\int_{\hat{\mathcal{C}}_i}I=\int_{\hat{\mathcal{C}}_i}(I+d\lambda)+\delta_{i>h}b_{i-h}$,
where the first equality follows from~\eqref{eq:210616n1} and $\int_{\hat{\mathcal{C}}_i} \gamma_j=\delta_{ij}$, and the second equality follows from Lemma~\ref{lemmmm}. Plugging this identity in~\eqref{eq:210616n1} yields the result. 
\end{proof}
\begin{cor}\label{cor:210618n1}
One has
\begin{equation*}
\lambda\,=\,-\sum_{i=1}^h\,[h+i|\bullet]\otimes b_i\,+\,\sum_{k\geq 1}\Bigg(\sum_{i=1}^{2h}\bigg([i|\bullet]\otimes\int_{\hat{\mathcal{C}}_i}\frac{({\rm ad}(\lambda))^k}{(k+1)!}(d\lambda)\bigg)\,-\,\int_{\hat x}^z\frac{({\rm ad}(\lambda))^k}{(k+1)!}(d\lambda)\Bigg)\,.
\end{equation*}
\end{cor}
\begin{proof}
First of all, by the formula for the derivative of the exponential map, one has
\begin{equation}\label{eq:210616n2}
I\,=\,\frac{1-e^{{\rm ad}(\lambda)}}{{\rm ad}(\lambda)}(d\lambda)\,.
\end{equation} 

It follows that
\begin{equation}\label{eq:210616n3}
d\lambda\,=\,-\,I\,+\,\frac{1+{\rm ad}(\lambda)-e^{{\rm ad}(\lambda)}}{{\rm ad}(\lambda)}(d\lambda)\,=\,-\,I\,-\,\sum_{k\geq 1}\frac{({\rm ad}(\lambda))^k}{(k+1)!}(d\lambda)\,.
\end{equation}
Combining this with Corollary~\ref{CorAltConstr}, we get
\begin{equation}\label{eq:210616n4}
d\lambda\,=\,-\,\sum_{i=1}^h\gamma_{h+i}\otimes b_i\,-\,\sum_{i=1}^{2h}\gamma_i\otimes\bigg(\int_{\hat{\mathcal{C}}_i}I+d\lambda\bigg)-\,\sum_{k\geq 1}\frac{({\rm ad}(\lambda))^k}{(k+1)!}(d\lambda)\,.
\end{equation}
But~\eqref{eq:210616n3} also tells us that $I+d\lambda=-\sum_{k\geq 1}\frac{({\rm ad}(\lambda))^k}{(k+1)!}(d\lambda)$, and if we can plug this into eq.~\eqref{eq:210616n4} we obtain 
$$
d\lambda\,=\,-\,\sum_{i=1}^h\gamma_{h+i}\otimes b_i\,+\,\sum_{i=1}^{2h}\bigg(\gamma_i\otimes\int_{\hat{\mathcal{C}}_i}\sum_{k\geq 1}\frac{({\rm ad}(\lambda))^k}{(k+1)!}(d\lambda)\bigg)-\,\sum_{k\geq 1}\frac{({\rm ad}(\lambda))^k}{(k+1)!}(d\lambda)\,.
$$
The statement follows by integrating this equation with the initial condition $\lambda(\hat x)=0$.
\end{proof}

Let ${\rm Lie}(b_1,\ldots ,b_h)[d]$ be the degree-$d$ part of ${\rm Lie}(b_1,\ldots ,b_h)$, and let $\lambda[d]$ be the component of $\lambda$ in $\Gamma(\hat C\setminus\pi^{-1}(\infty),\mathcal{O}_{\hat C})\,\hat\otimes\,{\rm Lie}(b_1,\ldots ,b_h)[d]$, so that $\lambda=\sum_{d\geq 1}\lambda[d]$. Corollary~\ref{cor:210618n1} immediately implies that 
\begin{equation*}
\lambda[1]\,=\,-\,\sum_{i=1}^{h}\,[h+i|\bullet]\otimes b_i\,,
\end{equation*}
as well as the following recursive formula for higher-degree terms:
\begin{cor}\label{cor:210618n2}
For any $d\geq 2$ one has
\begin{align}\label{eq3:211006}
&\lambda[d]\,=\,\sum_{k\geq 1}\frac{1}{(k+1)!}\,\times\\
&\sum_{\substack{d_1,\ldots ,d_{k+1}\geq 1\\d_1+\cdots +d_{k+1}=d}}\Bigg(\sum_{i=1}^{2h}\bigg([i|\bullet]\otimes\int_{\hat{\mathcal{C}}_i}[\lambda[d_1],\ldots [\lambda[d_k],d\lambda[d_{k+1}]]]\bigg)
-\int_{\hat x}^\bullet[\lambda[d_1],\ldots [\lambda[d_k],d\lambda[d_{k+1}]]]\Bigg)\notag.
\end{align}
\end{cor}
Let now $I[d]$ denotes the degree-$d$ component of $I$, so that in particular 
\begin{equation*}
I[1]\,=\,\sum_{i=1}^{h}\gamma_{h+i}\otimes b_i\,.
\end{equation*}
Then by combining Corollary~\ref{CorAltConstr} with eq.~\eqref{eq:210616n3} we find, for $d\geq 2$,
\begin{equation}\label{eq4:211006}
I[d]\,=\,-\,\sum_{k\geq 1}\frac{1}{(k+1)!}\sum_{\substack{d_1,\ldots ,d_{k+1}\geq 1\\d_1+\cdots +d_{k+1}=d}}\sum_{i=1}^{2h}\bigg(\gamma_i\otimes\int_{\hat{\mathcal{C}}_i}[\lambda[d_1],\ldots [\lambda[d_k],d\lambda[d_{k+1}]]]\bigg)
\,,
\end{equation}
hence one can plug in the values of~$\lambda[d_i]$ computed with Corollary~\ref{cor:210618n2} for all $d_i<d$ and compute~$I[d]$ (see \S\ref{sec:computations} for explicit low-degree formulas computed with this method).

\subsection{The map $\Psi$}\label{ssec:52}

Let~$t$ be a local coordinate on~$C$ at~$\infty$. The space of Laurent formal differentials at~$\infty$ is then $\mathbb C((t))dt$.  
Laurent expanding at~$\infty$ enables one to view $\Gamma(C\setminus\infty,\Omega^1_C)$ as a vector subspace of $\mathbb C((t))dt$.
Moreover, integration along the $A$-cycles defines a linear map $\Gamma(C\setminus\infty,\Omega^1_C)\to\mathbb C^h$,  which yields a decomposition $\Gamma(C\setminus\infty,\Omega^1_C)={\rm Span}_\mathbb{C}(\omega_1,\ldots ,\omega_h)\oplus\mathrm{ker}(\Gamma(C\setminus\infty,\Omega^1_C)\to\mathbb C^h)$. Note that $\mathrm{ker}(\Gamma(C\setminus\infty,\Omega^1_C)\to\mathbb C^h)={\rm Span}_\mathbb{C}(\beta_1,\ldots ,\beta_h)\oplus \textrm{d}\,\Gamma(C\setminus\infty,\mathcal{O}_C)$.

\begin{lem}\label{lemm:decomp}
There is a direct sum decomposition 
\begin{equation}\label{eq:210514n1}
\mathbb C((t))\,dt\,=\,\mathbb C[[t]]\,\frac{dt}{t}\,\oplus
\,\mathrm{ker}(\,\Gamma(C\setminus\infty,\Omega^1_C)\,\to\,\mathbb C^h\,)\,.
\end{equation} 
\end{lem}

\begin{proof}
The residue map $\text{res}:\mathbb{C}((t))dt\rightarrow \mathbb{C}$ yields a decomposition $\mathbb{C}((t))dt=\tfrac{dt}{t}\mathbb{C}\oplus {\rm ker}({\rm res})$. We are left with showing that ${\rm ker}({\rm res})=\mathbb C[[t]]dt\oplus\mathrm{ker}(\Gamma(C\setminus\infty,\Omega^1_C)\to\mathbb C^h)$.

It follows\footnote{Let~$\Omega^1_C(nP)$ denote the sheaf of 1-forms with a pole of order at most~$n$ at~$P$, and holomorphic elsewhere, and let~$\mathcal{O}_C(-nP)$ denote the sheaf of holomorphic functions with a zero of order at least~$n$ at~$P$. The Riemann-Roch theorem implies that ${\rm dim}(\Gamma(C,\Omega^1(nP))-{\rm dim}(\Gamma(C,\mathcal{O}_C(-nP))=h-1+n$. Since ${\rm dim}(\Gamma(C,\mathcal{O}_C(-nP))=1$ for $n=0$, and $=0$ otherwise, it follows that ${\rm dim}(\Gamma(C,\Omega^1(nP))=h$ for $n=0$ and $n=1$, and then ${\rm dim}(\Gamma(C,\Omega^1(nP))=h+n-1$ strictly increases for any $n\geq 2$, thus implying the existence of the family~$(\xi_i)$.} from Riemann-Roch's theorem that, for any point~$P$ of a closed Riemann surface~$C$ of genus $h\geq 1$, there exists a family of meromorphic 1-forms $(\xi_i)_{i\geq 2}\subset\Gamma(C\setminus P,\Omega^1_C)$ such that $\text{ord}_{P}(\xi_i)=-i$. Therefore, one can write ${\rm ker}({\rm res})=\mathbb{C}[[t]]dt+\Gamma(C\setminus\infty,\Omega^1_C)$. The statement follows from the fact that $\mathbb{C}[[t]]dt\cap\Gamma(C\setminus\infty,\Omega^1_C)\simeq \mathbb{C}^h$ via the map $\xi\rightarrow (\int_{A_1}\xi,\cdots ,\int_{A_h}\xi)$.
\end{proof}
 
\begin{defn}
We denote by $\Psi : \mathbb C((t))dt\twoheadrightarrow \mathrm{ker}(\Gamma(C\setminus\infty,\Omega^1_C)\to\mathbb C^h)$
the projection on the second component of the decomposition~\eqref{eq:210514n1}. 
\end{defn}

Consider the non-degenerate pairing $\langle\,,\,\rangle:\mathbb{C}((t))\otimes \mathbb{C}((t))dt\rightarrow \mathbb{C}$ given by $\langle f,\alpha\rangle:=\text{res}(f\alpha)$. Since $t\mathbb{C}[[t]]^\perp = \mathbb{C}[[t]]\tfrac{dt}{t}$ with respect to this pairing, it follows from~(\ref{eq:210514n1}) that
\begin{equation*}
t\,\mathbb{C}[[t]]^*\,\simeq\, \frac{\mathbb{C}((t))\,dt}{\mathbb{C}[[t]]\,\tfrac{dt}{t}} \,\simeq\, \mathrm{ker}(\,\Gamma(C\setminus\infty,\Omega^1_C)\,\to\,\mathbb C^h\,)\,.
\end{equation*}
Let $(\xi_i)_{i\geq 1}$ be a basis of regular 1-forms on $C\setminus\infty$ with vanishing $A$-cycle integrals which is dual to the standard basis $(t^i)_{i\geq 1}$ of $t\mathbb{C}[[t]]$. We define 
\begin{equation*}
\psi:=\sum_{i\geq 1}t^i\otimes \xi_i \quad \in \,t\,\mathbb{C}[[t]]\,\hat\otimes\,\mathrm{ker}(\,\Gamma(C\setminus\infty,\Omega^1_C)\,\to\,\mathbb C^h\,) \,.
\end{equation*}

The projection $\Psi:\mathbb{C}((t))dt\to \mathrm{ker}(\Gamma(C\setminus\infty,\Omega^1_C)\to\mathbb C^h)$ can then be written in terms of~$\psi$ as follows:
\begin{equation*}\label{GreenProj}
\Psi(\alpha)\,=\,\sum_{i\geq 1}\,\langle t^i,\alpha \rangle\,\xi_i\,=\,({\rm res}\otimes {\rm id})\,(\alpha\bullet \psi)\,,
\end{equation*}
where $\alpha\bullet\psi\in \mathbb{C}((t))dt\,\hat\otimes\,\mathrm{ker}(\Gamma(C\setminus\infty,\Omega^1_C)\to\mathbb C^h)$ denotes the term-by-term multiplication by~$\alpha$ on the first component of~$\psi$.
%where bullet is the action of $\mathbb{C}((t))dt$ on $\mathbb{C}((t))dt \hat \otimes \mathbb{C}[[t]]$ (series sum_{n geq 0}a_n(t)dt otimes t^n
%avec a_n qcq) and res_{t=0} otimes id is the map C((t))dt hat otimes C[[t]]->CC[[t]] given by sum_{n geq 0}a_n(t)dt otimes t^n->sum_{n geq 0}res_{t=0}(a_n(t)dt) otimes t^n 

\subsection{The map $\Psi_{\hat\infty}$}\label{ssec:53}

Let~$\hat\infty$ be the preimage of~$\infty$ by~$\pi$ which lies in~$\hat D$. Let $\mathrm{inj}_{\hat\infty} : \Gamma(\hat C\setminus\pi^{-1}(\infty),\Omega^1_{\hat C})\hookrightarrow\mathbb C((t))dt$
be the injection given by sending a 1-form to its formal Laurent series expansion in a local coordinate~$\hat t$ at~$\hat\infty$, and by identifying $\mathbb{C}((\hat t))d\hat t$ with  $\mathbb{C}((t))dt$. We set $\Psi_{\hat\infty}:=
\pi^*\circ\Psi\circ\mathrm{inj}_{\hat\infty}$; this is a linear endomorphism of 
$\Gamma(\hat C\setminus\pi^{-1}(\infty),\Omega^1_{\hat C})$. One can check that this map does not depend on the choice of local coordinates~$t$ and~$\hat t$.

\begin{lem}\label{lem:210516n1}
If $\alpha\in\Gamma(\hat C\setminus\pi^{-1}(\infty),\Omega^1_{\hat C})$, then $(\mathrm{id}-\Psi_{\hat \infty})(\alpha)$
has at most a simple pole at~$\hat\infty$. 
\end{lem}
\begin{proof}
Let $\mathrm{inj}_\infty:\Gamma(C\setminus\infty,\Omega^1_C)\hookrightarrow\mathbb C((t))dt$ be the injection given by sending a 1-form to its formal Laurent series at~$\infty$. 
One then has 
$$
\mathrm{inj}_{\hat\infty}\circ(\mathrm{id}-\Psi_{\hat \infty})(\alpha)=(\mathrm{id}-\mathrm{inj}_{\hat\infty}\circ\pi^*\circ\Psi)(\mathrm{inj}_{\hat\infty}(\alpha))
=(\mathrm{id}-\mathrm{inj}_\infty\circ\Psi)(\mathrm{inj}_{\hat\infty}(\alpha)), 
$$
where the first equality follows from the definition of~$\Psi_{\hat \infty}$, and the second one from 
$\mathrm{inj}_\infty=\mathrm{inj}_{\hat\infty}\circ\pi^*$. It follows from the definition of $\Psi$ that 
for any~$\xi$ in $\mathbb C((t))dt$, $(\mathrm{id}-\mathrm{inj}_\infty\circ\Psi)(\xi)$ belongs to 
$\mathbb{C}[[t]]\tfrac{dt}{t}$. The above equality then implies that so does 
$\mathrm{inj}_{\hat\infty}\circ(\mathrm{id}-\Psi_{\hat \infty})(\alpha)$, which proves the statement.  
\end{proof}

The next result relates the map $\Psi_{\hat{\infty}}$ with the fundamental form of the third kind $\psi_{\hat\infty}$ defined in \S\ref{sec:fundformIIkind}. 

\begin{prop}
For any 1-form $\alpha\in \Gamma(\hat C\setminus \pi^{-1}(\infty),\Omega^1_{\hat C})$ and any fixed $(P,v)\in T^{1,0}(C\setminus\infty)$, one has $\Psi_{\hat\infty}(\alpha)(P,v)={\rm res}_{\hat\infty}(\psi_{\hat \infty}|_{(P,v)\times \hat C}\cdot\alpha)$.
\end{prop}
\begin{proof}
Since $\Gamma(\hat C\setminus \pi^{-1}(\infty),\Omega^1_{\hat C})\hookrightarrow \mathbb{C}((t))dt$,  using the decomposition \eqref{eq:210514n1} it is enough to show the statement for $\alpha\in \mathbb{C}[[t]]\tfrac{dt}{t}$ (we identify $\alpha$ with its Laurent expansion) and for $\alpha\in \mathrm{ker}(\Gamma(C\setminus\infty,\Omega^1_C)\to\mathbb{C}^h)$.

Suppose first that $\alpha\in \mathbb{C}[[t]]\tfrac{dt}{t}$, so that $\Psi_{\hat\infty}(\alpha)=0$. Since $\psi_{\hat \infty}$ has a simple zero along $C\times \hat\infty$, for any fixed $(P,v)$ with $P\neq \infty$ we get an asymptotic expansion $\psi_{\hat \infty}|_{(P,v)\times C}\in t\mathbb{C}[[t]]$, and therefore ${\rm res}_{\hat\infty}(\psi_{\hat \infty}|_{(P,v)\times \hat C}\cdot\alpha)=0$.

Suppose now that $\alpha\in \mathrm{ker}(\Gamma(C\setminus\infty,\Omega^1_C)\to\mathbb{C}^h)$, so that $\Psi_{\hat\infty}(\alpha)=\alpha$. Let $\sigma_{\infty}$, $\sigma_{P}$ and $\sigma_{P,\infty}$ denote small positively oriented contours contained in the simply connected fundamental domain $D\subset C$ encircling  $\infty$, $P$ and both $\infty$ and $P$, respectively. Note that one can locally view $\psi_{\hat \infty}|_{(P,v)\times \hat C}$ as a function on $D$. 

By definition of $\psi_{\hat \infty}|_{(P,v)\times \hat C}$, Riemann's bilinear relations, and the fact that $\int_{\mathcal{A}_i}\varphi|_{(P,v)\times C}=\int_{\mathcal{A}_i}\alpha=0$, we get
\begin{align*}
&\frac{1}{2\pi i}\int_{\sigma_{P,\infty}}\psi_{\hat \infty}|_{(P,v)\times C}\,\cdot\,\alpha\,=\,\int_{z\in\sigma_{P,\infty}}\alpha(z)\bigg(\int_{\infty}^z\varphi|_{(P,v)\times C}\bigg)\\
&\,=\,\sum_{i=1}^g \bigg(\int_{\mathcal{A}_i}\varphi|_{(P,v)\times C}\int_{\mathcal{B}_i} \alpha\,-\,\int_{\mathcal{A}_i} \alpha\int_{\mathcal{B}_i}\varphi|_{(P,v)\times C}\bigg)=0
\end{align*}
Moreover, since $\psi_{\hat \infty}|_{(P,v)\times \hat C}=\tfrac{1}{t(P)-t}+O(1)$ for $t$ a local coordinate in a neighborhood of $P$, then one has
\begin{equation*}
\frac{1}{2\pi i}\int_{\sigma_P}\psi_{\hat \infty}|_{(P,v)\times \hat C}\,\cdot\,\alpha\,=\,\mathrm{res}_{P}(\psi_{\hat \infty}|_{(P,v)\times \hat C}\,\cdot\,\alpha)\,=\,-\alpha(P,v).
\end{equation*}

Combining these two identities, we get
\begin{equation*}
\mathrm{res}_{\hat{\infty}}(\psi_{\hat \infty}|_{(P,v)\times \hat C}\,\cdot\,\alpha)\,=\,\frac{1}{2\pi i}\int_{\sigma_\infty}\psi_{\hat \infty}|_{(P,v)\times \hat C}\,\cdot\,\alpha\,=\,\frac{1}{2\pi i}\bigg(\int_{\sigma_{P,\infty}}-\int_{\sigma_P}\bigg)\psi_{\hat \infty}|_{(P,v)\times C}\,\cdot\,\alpha\,=\,\alpha(P,v).
\end{equation*}
\end{proof}
%\begin{lem}
%Let $\alpha\in\Gamma(\hat C\setminus\pi^{-1}(\infty),\Omega^1_{\hat C})$ and let $\psi_{\hat\infty}\in\Gamma(C\setminus\infty\times \hat C,\Omega^1_C\boxtimes \mathcal{O}_{\hat C}(\Delta))$ be the fundamental form of the third kind introduced in \S\ref{sec:fundformIIkind}. Then \textcolor{red}{FZ: new (bad?) notation. Problem: the pole is not simple, so it's not a Poincaré residue as we defined it.}
%\begin{equation}\label{eq:210604n6}
%\Psi_{\hat\infty}(\alpha)\,=\,{\rm res}^{(2)}_{\hat\infty}(\alpha * \psi_{\hat\infty})\,,
%\end{equation}
%%\Psi_{\hat\infty}(\alpha)(z)\,=\,{\rm res}_{\hat t=0}\,(\psi_{\hat\infty}(z,\hat t)\,\alpha(\hat t))\,.
%\end{lem}
%where $\alpha * \psi_{\hat\infty}\in \Gamma\big((C\setminus\infty)\times(\hat C\setminus\pi^{-1}(\infty)),\Omega^1_{C}\boxtimes \Omega^1_{\hat C}(\Delta)\big)$ denotes the multiplication by~$\alpha$ on the second component of~$\psi_{\infty}$, and ${\rm res}^{(2)}_{\hat\infty}$ denotes the residue on the second copy of the curve.
%\begin{proof}
%\textcolor{red}{FZ: proof of Lemma 1.6 road map (to be adapted). Since that proof makes use of properties of $\psi_{\hat\infty}$ coming from its definition as integral of fundamental form of the second kind, I suggest we first write the background section about $\psi_{\hat\infty}$, and then write this proof} 
%\end{proof}
 
\subsection{General formulas for $K_{\hat\infty,j}$ and $H_{\hat\infty,j}$}\label{ssec:K}

In Definition~\ref{def210514n1} we have introduced the family of forms $K_{P,j}$, where $P\in\hat C$ and $j\in[\![1,h]\!]$. Throughout this section, we focus only on $K_{\hat\infty,j}$. 

Recall that $K_{\hat\infty,j}\in\Gamma(\hat C\setminus\pi^{-1}(\infty),\Omega^1_{\hat C})\,\hat\otimes\, 
U(\mathrm{Lie}(b_1,\ldots,b_h))^\wedge$, and it has simple poles at the points of $\pi^{-1}(\infty)$ (by Definition~\ref{def210514n1} and the properties of the forms $\omega_{P,i_1\ldots i_m}$), while $g\in\Gamma(\hat C\setminus\pi^{-1}(\infty),\mathcal O_{\hat C})\,\hat\otimes\,U(\mathrm{Lie}(b_1,\ldots,b_h))^\wedge$, and it has higher order poles at the points of $\pi^{-1}(\infty)$.

\begin{defn}
Let ${\rm op}$ denote the linear endomorphism of $\Gamma(\hat C\setminus\pi^{-1}(\infty),\Omega^1_{\hat C})$ defined by\footnote{Here and elsewhere~$\omega_k$ denotes by an abuse of terminology the pull-back to~$\hat C$ of the holomorphic 1-form~$\omega_k$, originally defined on~$C$.} 
$\alpha\mapsto\sum_{k=1}^h\big(\int_{\hat{\mathcal{A}}_k}\alpha\big)\,\omega_k$.
\end{defn}

%\begin{defn}\label{def:210601}
%Let us set $K^\dagger_{\hat\infty,j}:=(1-g)K_{\hat\infty,j}$. 
%\end{defn} 
%
%In particular, $K^\dagger_{\hat\infty,j}\in\Gamma(\hat C\setminus\pi^{-1}(\infty),\Omega^1_{\hat C})\,\hat\otimes\,U(\mathrm{Lie}(b_1,\ldots,b_h))^\wedge$, and the poles at the points of $\pi^{-1}(\infty)$ are not simple. Notice that, by definition, $K^\dagger_{\hat\infty,j}=K_{\hat\infty,j}-H_{\hat\infty,j}$. 

\begin{prop}\label{lem:210601}
One has
\begin{equation}\label{eq:recursion}
g\,K_{\hat\infty,j}\,=\big((\Psi_{\hat\infty}+{\rm op})\,\otimes\,{\rm id}\big)\big((g-1)\,K_{\hat\infty,j}\big)
\,+\,\omega_j\otimes\frac{b_j}{e^{b_j}-1}
\end{equation}
(equality in $\Gamma(\hat C\setminus\pi^{-1}(\infty),\Omega^1_{\hat C})\,\hat\otimes\,
U(\mathrm{Lie}(b_1,\ldots,b_h))^\wedge$). 
\end{prop}

\begin{proof} 
Let $\alpha:=gK_{\hat\infty,j}-(\Psi_{\hat\infty}\otimes {\rm id})((g-1)K_{\hat\infty,j})$. 
By the definition of $\Psi_{\hat \infty}$ and the fact that $gK_{\hat\infty,j}=H_{\hat\infty,j}$,~$\alpha$ is the pull-back via~$\pi$ of an element of $\Gamma(C\setminus\infty,\Omega^1_{C})\,\hat\otimes\,
U(\mathrm{Lie}(b_1,\ldots,b_h))^\wedge$. 

On the other hand, one can write $\alpha=K_{\hat\infty,j}+((\mathrm{id}-\Psi_{\hat\infty})\otimes {\rm id})((g-1)K_{\hat\infty,j})$, and Lemma~\ref{lem:210516n1} implies that $((\mathrm{id}-\Psi_{\hat\infty})\otimes {\rm id})((g-1)K_{\hat\infty,j})$ has at most a simple pole at~$\hat\infty$. Since~$K_{\hat\infty,j}$ has a simple pole at~$\hat\infty$, it follows that~$\alpha$ has at most a simple pole at~$\hat\infty$. 

All this implies that $\alpha$ is the pull-back of an element of $\Gamma(C\setminus\infty,\Omega^1_{C})\,\hat\otimes\,
U(\mathrm{Lie}(b_1,\ldots,b_h))^\wedge$ with at most a simple pole at~$\infty$. Because there are no global meromorphic 1-forms on~$C$ with a single simple pole at~$\infty$, we conclude that~$\alpha$ must be the pull-back of a holomorphic 1-form, which is uniquely determined by its integrals on the $A$-cycles:
\begin{equation}\label{210504}
\alpha\,=\,\sum_{k=1}^{h}\,\omega_k\otimes \int_{\hat{\mathcal A}_k}\alpha\,.
\end{equation}

One then has 
\begin{align*}
\int_{\hat{\mathcal A}_k}\alpha&\,=\,\int_{\hat{\mathcal A}_k}K_{\hat\infty,j}
\,+\,\int_{\hat{\mathcal A}_k}(g-1)\,K_{\hat\infty,j}\,-\,\int_{\hat{\mathcal A}_k}(\Psi_{\hat\infty}\otimes {\rm id})((g-1)\,K_{\hat\infty,j})\\
&\,=\,\delta_{jk}\,\frac{b_k}{e^{b_k}-1}\,+\,\int_{\hat{\mathcal A}_k}(g-1)\,K_{\hat\infty,j}\,,
\end{align*}
where the last equality follows from the property~\eqref{eq:210516n1} of the 1-forms $\omega_{\hat\infty,i_1\ldots i_m}$, and from the fact that $\int_{\hat{\mathcal A}_k}(\Psi_{\hat\infty}\otimes {\rm id})((g-1)K_{\hat\infty,j})=0$, by the definition of~$\Psi_{\hat\infty}$. 
Combining this with~\eqref{210504}, one obtains 
$$%\begin{equation}\label{eq:210516n2}
\alpha\,=\,\omega_j\otimes\frac{b_j}{e^{b_j}-1}\,+\,\sum_{k=1}^{h}\,\omega_k\otimes \int_{\hat{\mathcal A}_k}(g-1)\,K_{\hat\infty,j}\,,
$$%\end{equation}
which implies the statement. 
\end{proof}

In the rest of this section we will see how to use the above result to obtain explicit formulas for the coefficients of~$K_{\hat{\infty},j}$ and of~$H_{\hat{\infty},j}$ in terms of those of~$g$. Let us set $\mathcal{E}:=\Gamma(\hat C\setminus\pi^{-1}(\infty),\Omega^1_{\hat C})\,\hat\otimes\,U({\rm Lie}(b_1,\ldots ,b_h))^\wedge$ and $\mathcal{F}:=\Gamma(\hat C\setminus\pi^{-1}(\infty),\mathcal{O}_{\hat C})\,\hat\otimes\,U({\rm Lie}(b_1,\ldots ,b_h))^\wedge$, so that $K_{\hat\infty,j},H_{\hat\infty,j}\in\mathcal{E}$ and $g\in\mathcal{F}$. Note that~$\mathcal{F}$ is an algebra, equipped with a left action on~$\mathcal{E}$ given by multiplying on the first component and concatenating on the second.
\begin{defn}
(a) Let ${\rm Op}$ denote the linear endomorphism $(-{\rm id}+\Psi_{\hat{\infty}}+{\rm op})\otimes {\rm id}$ of $\mathcal{E}$.

(b) Let $g-1$ denote the linear endomorphism of $\mathcal{E}$ given by $\kappa\to(g-1)\,\kappa$.
\end{defn}
Using this notation, the statement of Proposition~\ref{lem:210601} can be written as follows:
\begin{equation}\label{eq:210604n1}
\big({\rm id}\,-\,{\rm Op}\circ (g-1)\big)(K_{\hat\infty,j})\,=\,\omega_j\,\otimes\,\frac{b_j}{e^{b_j}-1}\,.
\end{equation}

Consider now the decreasing filtration on $\mathcal{E}$ given by $F^k\mathcal{E}:={\rm ker}(\mathcal{E}\to\prod_{d=0}^{k-1}\mathcal{E}[d])$, where $\mathcal{E}[d]:=\Gamma(\hat C\setminus\pi^{-1}(\infty),\Omega^1_{\hat C})\,\otimes\,U({\rm Lie}(b_1,\ldots ,b_h))[d]$, and $U({\rm Lie}(b_1,\ldots ,b_h))[d]$ is the degree-$d$ part of $U({\rm Lie}(b_1,\ldots ,b_h))^\wedge$. This induces a topology on $\mathcal{E}$, for which $\mathcal{E}$ is complete and separated.
\begin{lem}\label{lem:210604n1}
For any $v\in\mathcal{E}$ the series $\sum_{r\geq 0}({\rm Op}\circ(g-1))^r(v)$ converges. The linear endomorphism ${\rm id}-{\rm Op}\circ(g-1)$ is invertible, and its inverse is the map $v\mapsto \sum_{r\geq 0}({\rm Op}\circ(g-1))^r(v)$.
\end{lem}
\begin{proof}
Consider the decreasing filtration on $\mathcal{F}$ given by $F^k\mathcal{F}:={\rm ker}(\mathcal{F}\to\prod_{d=0}^{k-1}\mathcal{F}[d])$, where $\mathcal{F}[d]:=\Gamma(\hat C\setminus\pi^{-1}(\infty),\mathcal{O}_{\hat C})\,\otimes\,U({\rm Lie}(b_1,\ldots ,b_h))[d]$. One has $F^k\mathcal{F}\cdot F^l\mathcal{E}\subset F^{k+l}\mathcal{E}$. 

Since $g-1\in F^1\mathcal{F}$, it follows that $(g-1)(F^k\mathcal{E})\subset F^{k+1}\mathcal{E}$ for any $k\geq 0$. Moreover, one has ${\rm Op}(F^k\mathcal{E})\subset F^k\mathcal{E}$. All this implies that ${\rm Op}\circ (g-1)(F^k\mathcal{E})\subset F^{k+1}\mathcal{E}$ for any $k\geq 0$, from which it follows that $\sum_{r\geq 0}({\rm Op}\circ(g-1))^r(v)$ converges, and that it therefore provides an inverse to the linear endomorphism ${\rm id}-{\rm Op}\circ(g-1)$.
\end{proof}

\begin{cor}\label{cor1}
One has
\begin{equation}\label{eq:210604n2}
K_{\hat{\infty},j}\,=\,\sum_{r\geq 0}({\rm Op}\circ(g-1))^r\,\bigg(\omega_j\otimes\frac{b_j}{e^{b_j}-1}\bigg)\,.
\end{equation}
\end{cor}
\begin{proof}
This follows by combining Lemma~\ref{lem:210604n1} with eq.~\eqref{eq:210604n1}.
\end{proof}

\begin{cor}\label{cor2}
One has
\begin{equation}\label{eq:210604n3}
H_{\hat{\infty},j}\,=\,\bigg({\rm id}\,+\,\big((\Psi_{\hat{\infty}}+{\rm op})\otimes{\rm id}\big)\circ(g-1)\circ\sum_{r\geq 0}({\rm Op}\circ(g-1))^r\bigg)\,\bigg(\omega_j\otimes\frac{b_j}{e^{b_j}-1}\bigg)\,.
\end{equation}
\end{cor}
\begin{proof}
Since $H_{\hat{\infty},j}=gK_{\hat{\infty},j}$, Proposition~\ref{lem:210601} tells us that $H_{\hat{\infty},j}=\big((\Psi_{\hat\infty}+{\rm op})\otimes{\rm id}\big)\circ(g-1)(K_{\hat{\infty},j})+\omega_j\otimes\tfrac{b_j}{e^{b_j}-1}$. The statement follows by combining this identity with eq.~\eqref{eq:210604n2}.
\end{proof}

For $d\geq 0$ let~$K_{\hat\infty,j}[d]$ and~$H_{\hat\infty,j}[d]$ denote the components of~$K_{\hat\infty,j}$ and~$H_{\hat\infty,j}$ in~$\mathcal{E}[d]$, and let~$g[d]$ denote the component of~$g$ in~$\mathcal{F}[d]$. In particular, one has $K_{\hat\infty,j}[0]=\omega_j$, $g[0]=1$ and $H_{\hat\infty,j}[0]=g[0]\,K_{\hat\infty,j}[0]=\omega_j$.
\begin{cor}\label{cor515}
For any $d\geq 0$ one has
\begin{equation}\label{eq:210604n4}
K_{\hat\infty,j}[d]\,=\,\sum_{r\geq 0}\sum_{\substack{d_1,\ldots ,d_r\geq 1,d_{r+1}\geq 0\\d_1+\cdots +d_{r+1}=d}}{\rm Op}\circ g[d_1]\circ\cdots\circ {\rm Op}\circ g[d_r]\bigg(\omega_j\,\otimes\,\frac{\mathsf{b}_{d_{k+1}}}{d_{k+1}!}\,b_j^{d_{k+1}}\bigg)
\end{equation}
and
\begin{align}\label{eq:last}
&H_{\hat\infty,j}[d]\,=\,\omega_j\,\otimes\,\frac{\mathsf{b}_{d}}{d!}\,b_j^{d}\,+\\
&\sum_{r\geq 0}\sum_{\substack{d_0,\ldots ,d_r\geq 1,d_{r+1}\geq 0\\d_0+\cdots +d_{r+1}=d}}\big((\Psi_{\hat{\infty}}+{\rm op})\otimes{\rm id}\big)\circ g[d_0]\circ {\rm Op}\circ g[d_1]\circ\cdots\circ {\rm Op}\circ g[d_r]\,\bigg(\omega_j\otimes\frac{\mathsf{b}_{d_{r+1}}}{d_{r+1}!}\,b_j^{d_{r+1}}\bigg)\notag.
\end{align}
\end{cor}
\begin{proof}
These two equations follow from eq.~\eqref{eq:210604n2} and~\eqref{eq:210604n3}, respectively, combined with the fact that the linear endomorphism~${\rm Op}$ of~$\mathcal{E}$ has degree zero.
\end{proof}
Recall that, by definition, $K_{\hat\infty,j}[d]=\sum_{i_1,\ldots ,i_d\in[\![1,h]\!]}\omega_{\hat{\infty},i_1\cdots i_dj}b_{i_1}\cdots b_{i_d}$. Therefore, if we define a family of functions $g_{i_1\cdots i_d}\in\Gamma(\hat C\setminus\pi^{-1}(\infty),\mathcal{O}_{\hat{C}})$ by
$$
g\,=:\,1+\sum_{d\geq 1}\sum_{i_1,\ldots ,i_d\in[\![1,h]\!]}g_{i_1\cdots i_d}\,b_{i_1}\cdots b_{i_d}\,,
$$
then eq.~\eqref{eq:210604n4} implies the following:
\begin{cor}\label{cor516}
For any $d\geq 0$ and any $i_1,\ldots ,i_d\in[\![1,h]\!]$ one has
\begin{align}\label{eq:210604n5}
\omega_{\hat{\infty},i_1\cdots i_dj}&\,=\,\sum_{r\geq 0}\sum_{\substack{d_1,\ldots ,d_r\geq 1,d_{r+1}\geq 0\\d_1+\cdots +d_{r+1}=d}}(-{\rm id}+\Psi_{\hat{\infty}}+{\rm op})(g_{i_1\cdots i_{d_1}}\cdot(-{\rm id}+\Psi_{\hat{\infty}}+{\rm op})(g_{i_{d_1+1}\cdots i_{d_1+d_2}}\cdots \\
&\times\,(-{\rm id}+\Psi_{\hat{\infty}}+{\rm op})\bigg(g_{i_{d_1+\cdots+d_{r-1}+1}\cdots i_{d_1+\cdots +d_r}}\cdot \bigg(\frac{\mathsf{b}_{d_{r+1}}}{d_{r+1}!}\,\delta_{j,i_{d_1+...+d_r+1},...,i_d}\,\omega_j\bigg)\bigg)\cdots )\,,\notag
\end{align}
where $\delta_{j,i_{d_1+...+d_r+1},...,i_d}=1$ if the last~$d_r$ indices $i_{d_1+\cdots +d_{r}+1},\ldots ,i_d$ are equal to~$j$, and it vanishes otherwise.
\end{cor}
Explicit low-degree formulas computed with these methods can be found in \S\ref{sec:computations}.

\subsection{Low-degree formulas}\label{sec:computations}

As explained in \S\ref{ssec:g}, in order to compute the first terms of~$g$ one needs to invert the relations~\eqref{eq1:211006} and compute the first few coefficients of the Lie algebra elements $\Lambda_l$. We report below the first terms: for any $l\in [\![1,h]\!]$ one has
\begin{align*}
\Lambda_l&\,=\,-\,\sum_{i,j=1}^{h}\,[h+i,h+j|l]\,\bigg(b_i+\frac{b_i^2}{2}+\cdots\bigg)\bigg(b_j+\frac{b_j^2}{2}+\cdots\bigg)\\
&\,-\,\sum_{i,j,k=1}^{h}\Big([h+i,h+j,h+k|l]\,-\,\sum_{r=1}^{2h}\big([h+j,h+k|r]\,[h+i,r|l]\\
&\,+\,[h+i,h+j|r]\,[r,h+k|l]\big)\Big)(b_i+\cdots)(b_j+\cdots)(b_k+\cdots)\,+\,\cdots\,,\\
\Lambda_{h+l}&\,=b_l+\frac{b^2_l}{2}+\frac{b^3_l}{6}+\cdots\,-\,\sum_{i,j=1}^{h}\,[h+i,h+j|h+l]\,\bigg(b_i+\frac{b_i^2}{2}+\cdots\bigg)\bigg(b_j+\frac{b_j^2}{2}+\cdots\bigg)\\
&\,-\,\sum_{i,j,k=1}^{h}\bigg([h+i,h+j,h+k|h+l]\,-\,\sum_{r=1}^{2h}\Big([h+j,h+k|r]\,[h+i,r|h+l]\\
&\,+\,[h+i,h+j|r]\,[r,h+k|h+l]\Big)\bigg)\,(b_i+\cdots)(b_j+\cdots)(b_k+\cdots)\,+\,\cdots\,.
\end{align*}

Combining these formulas with~\eqref{eq2:211006}, one finds\footnote{We have rewritten these formulas in a more compact way by making use of the shuffle relations~\eqref{eq:shuffle} and the identity $[i|j]=\delta_{ij}$.}
\begin{align*}
g[1]&\,=\,-\sum_{i=1}^h\,[h+i|\bullet]\,b_i\,,\\
g[2]&\,=\,-\sum_{i=1}^h\,[h+i|\bullet]\,\frac{b_i^2}{2}\\
&+\sum_{i,j=1}^h\bigg([h+j,h+i|\bullet]+\sum_{r=1}^{2h}[h+i,h+j|r]\,[r|\bullet]\bigg)\, b_ib_j\,,\\
g[3]&\,=\,-\sum_{i=1}^h \,[h+i|\bullet]\,\frac{b_i^3}{6}\\
&\,+\sum_{i,j=1}^h\bigg([h+j,h+i|\bullet]+\sum_{r=1}^{2h}[h+i,h+j|r]\,[r|\bullet]\bigg)\, \bigg(\frac{b_i^2b_j}{2}+\frac{b_ib_j^2}{2}\bigg)\\
&\,+\sum_{i,j,k=1}^h\Bigg(-[h+k,h+j,h+i|\bullet]+\sum_{r=1}^{2h}\bigg([h+i,h+j,h+k|r]\\
&\,-\sum_{s=1}^{2h}\Big([h+i,h+j|s]\,[s,h+k|r]+[h+j,h+k|s]\,[h+i,s|r]\Big)\bigg)\,[h+r|\bullet]\\
&\,-[h+i,h+j|r]\,[h+k,r|\bullet]-[h+j,h+k|r]\,[r,h+i|\bullet]\Bigg)\, b_ib_jb_k\,.\\
\end{align*}

On the other hand, as explained in \S\ref{ssec:lambda}, in order to deduce compact expressions for the low-degree terms $I=gd(g^{-1})$ it is useful to first compute the low-degree terms of $\lambda=\log(g)$ using formula~\eqref{eq3:211006}, and then to use formula~\eqref{eq4:211006}. One finds
\begin{align*}
\lambda[1]&=-\sum_{i=1}^{2h}[h+i|\bullet]\otimes b_i\,,\\
\lambda[2]&=-\frac{1}{2}\sum_{i,j=1}^h\bigg([h+i,h+j|\bullet]\,-\sum_{r=1}^{2h}\,[h+i,h+j|r]\,[r|\bullet]\bigg)\otimes [b_i,b_j]\,,\\
\lambda[3]&=\sum_{i,j,k=1}^h\Bigg(\frac{1}{6}\,\big([h+j,h+k,h+i|\bullet]-[h+i,h+j,h+k|\bullet]\big)\\
&+\frac{1}{6}\sum_{r=1}^{2h}
\big([h+i,h+j,h+k|r]-[h+j,h+k,h+i|r]\big)\,[r|\bullet]
\\ &+\frac{1}{4}\sum_{s=1}^{2h}[h+j,h+k|s]\bigg([h+i,s|\bullet]-[s,h+i|\bullet]\\
&+\sum_{r=1}^{2h}\Big([s,h+i|r]-[h+i,s|r]\Big) [r|\bullet]\bigg)\Bigg)\otimes [b_i,[b_j,b_k]]\,, 
\end{align*}
as well as
\begin{align*}
I[1]&=\sum_{i=1}^h\beta_{i}\otimes b_i\,,\\
I[2]&=-\frac{1}{2}\bigg(\sum_{i,j=1}^h\sum_{r=1}^{2h}\,[h+i,h+j|r]\,\gamma_r\bigg)\otimes [b_i,b_j]\,,\\
I[3]&=-\sum_{i,j,k=1}^h\bigg(\sum_{r=1}^{2h}
\Big(\frac{1}{6}\,[h+i,h+j,h+k|r]-\frac{1}{6}\,[h+j,h+k,h+i|r]\\
&+\frac{1}{4}\sum_{s=1}^{2h}
[h+j,h+k|s]\,\big([s,h+i|r]-[h+i,s|r]\big)\Big)
 \gamma_r\bigg)\otimes [b_i,[b_j,b_k]]\,.
\end{align*}

Combining the expressions given above of the low-degree terms of~$g$ with the formulas~\eqref{eq:210604n4} and~\eqref{eq:decompK1}, one finds that, in the case $n=1$,
\begin{align*}
&{\pmb K}[1]\,=\,\sum_{i=1}^h \omega_i\otimes a_i\,,\\
&{\pmb K}[2]\,=\,\sum_{i,j=1}^h\Big(({\rm id}-\Psi_{\hat{\infty}}-{\rm op})([h+i|\bullet]\,\omega_j)-\frac{\delta_{ij}}{2}\,\omega_i\Big)\otimes [b_i,a_j]\,,\\
&{\pmb K}[3]\,=\,\sum_{i,j,k=1}^h\Bigg((-{\rm id}+\Psi_{\hat{\infty}}+{\rm op})\bigg(\frac{\delta_{ij}+\delta_{jk}}{2}\,[h+i|\bullet]-[h+i,h+j|\bullet]\,\omega_k\\
&\,-\sum_{r=1}^{2h}\,[h+j,h+i|r]\,[r|\bullet]\,\omega_k+[h+i|\bullet](\Psi_{\hat{\infty}}+{\rm op})([h+j|\bullet]\,\omega_k)\bigg)+\frac{\delta_{ijk}}{12}\,\omega_i\Bigg)\otimes [b_i,[b_j,a_k]]\,.
\end{align*}

Finally, combining the expressions for the low degree terms of~$I$ with the formula~\eqref{eq:last}, one finds that, in the case $n=1$,
\begin{align*}
{\pmb J}[1]&\,=\,\sum_{i=1}^h\beta_i\otimes b_i\,+\,\sum_{i=1}^h \omega_i\otimes a_i\,,\\
{\pmb J}[2]&\,=\,-\frac{1}{2}\bigg(\sum_{i,j=1}^h\sum_{r=1}^{2h}\,[h+i,h+j|r]\,\gamma_r\bigg)\otimes [b_i,b_j]\\
&\,-\,\sum_{i,j=1}^h\bigg((\Psi_{\hat{\infty}}+{\rm op})([h+i|\bullet]\,\omega_j)+\frac{\delta_{ij}}{2}\,\omega_i\bigg)\otimes [b_i,a_j]\,,\\
{\pmb J}[3]&\,=\,-\sum_{i,j,k=1}^h\bigg(\sum_{r=1}^{2h}
\Big(\frac{1}{6}\,[h+i,h+j,h+k|r]-\frac{1}{6}\,[h+j,h+k,h+i|r]\\
&+\frac{1}{4}\sum_{s=1}^{2h}
[h+j,h+k|s]\,\big([s,h+i|r]-[h+i,s|r]\big)\Big)
 \gamma_r\bigg)\otimes [b_i,[b_j,b_k]]\\
&+\sum_{i,j,k=1}^h\Bigg((\Psi_{\hat{\infty}}+{\rm op})\bigg(\frac{\delta_{ij}+\delta_{jk}}{2}\,[h+i|\bullet]-[h+i,h+j|\bullet]\,\omega_k\\
&\,-\sum_{r=1}^{2h}\,[h+j,h+i|r]\,[r|\bullet]\,\omega_k+[h+i|\bullet](\Psi_{\hat{\infty}}+{\rm op})([h+j|\bullet]\,\omega_k)\bigg)+\frac{\delta_{ijk}}{12}\,\omega_i\Bigg)\otimes [b_i,[b_j,a_k]]\,.
\end{align*}

\section*{Acknowledgements}

We would like to thank C. Gasbarri for the useful discussions on Poincar\'e residues.

The first-named author has received funding from the ANR grant ``Project HighAGT ANR-20-CE40-0016''.

The second-named author has received funding from the LabEx IRMIA, and from the European Union’s Horizon 2020 research and innovation programme under the Marie Skłodowska-Curie grant
agreement No 843960 for the project ``HIPSAM''.

This collaboration was supported by a three-month invitation of the second-named author at IRMA.

\end{document}